\documentclass[12pt]{amsart}
\usepackage{amsmath,amssymb}
\usepackage[backref=page]{hyperref}
\usepackage{tikz,float,caption}
\usetikzlibrary{arrows.meta,calc,decorations.markings,patterns,cd,patterns.meta,intersections}

\usepackage[margin=1.25 in,top=1.2in, bottom=1.2 in]{geometry}

\usepackage[bb=ams,scr=euler]{mathalpha}

\usepackage{setspace}
\usepackage{enumitem}
\setenumerate{
  wide=10pt,
  leftmargin=0pt,
  labelwidth=*,
  label=(\roman*),
  topsep=0pt,
  listparindent=0pt,
  parsep=4pt}

\usepackage{thmtools}
\declaretheoremstyle[headfont=\normalsize\normalfont\bfseries,notefont=\mdseries, notebraces={(}{)},bodyfont=\normalfont,postheadspace=0.5em]{basicstyle}
\declaretheoremstyle[headfont=\normalsize\normalfont\bfseries,notefont=\mdseries,
notebraces={(}{)},bodyfont=\normalfont\itshape,postheadspace=0.5em]{italstyle}

\declaretheorem[style=italstyle,name=Theorem,numberwithin=section]{theorem}

\declaretheorem[style=italstyle,name=Claim,sibling=theorem]{claim}
\declaretheorem[style=italstyle,name=Proposition,sibling=theorem]{prop}
\declaretheorem[style=italstyle,name=Lemma,sibling=theorem]{lemma}
\declaretheorem[style=basicstyle,name=Proof,numbered=no,qed=$\square$]{preproof}
\renewenvironment{proof}{\preproof}{\endpreproof}

\makeatletter
\newcommand{\coker}{\mathop{\mathrm{coker}}}
\newcommand{\colim}{\mathop{\mathrm{colim}}}
\newcommand{\im}{\mathop{\mathrm{im}}}
\newcommand{\C}{\mathbb{C}}
\renewcommand{\d}{\mathrm{d}}

\newcommand{\id}{\mathrm{id}}
\newcommand{\ip}[1]{\left\langle#1\right\rangle}
\newcommand{\norm}[1]{\left\lVert#1\right\rVert}
\newcommand{\pd}[2]{\frac{\partial #1}{\partial #2}}
\def\@secnumfont{\bfseries}
\renewcommand\section{\@startsection{section}{1}{0pt}{-3.5ex \@plus -1ex \@minus -.2ex}{2.3ex \@plus.2ex}{\centering\itshape}}
\renewcommand\subsection{\@startsection{subsection}{2}{0pt}{-3.5ex \@plus -1ex \@minus -.2ex}{2.3ex \@plus.2ex}{\centering\itshape}}
\newcommand{\Z}{\mathbb{Z}}

\newcommand{\abs}[1]{\left|#1\right|}
\newcommand{\bd}{\partial}
\newcommand{\pr}{\mathrm{pr}}
\newcommand{\R}{\mathbb{R}}
\newcommand{\set}[1]{\left\{#1\right\}}
\newcommand{\F}{\mathbb{F}}
\newcommand{\al}{\alpha}
\newcommand{\Om}{\Omega}
\newcommand{\La}{\Lambda}
\newcommand{\cl}[1]{\mathscr #1}

\newcommand{\til}[1]{\widetilde{#1}}

\makeatother

\title{The chord conjecture for conormal bundles}

\address{Department of Mathematics and Statistics, University of Montreal, C.P.\, 6128 Succ. Centre-Ville, Montreal, QC, H3C 3J7, Canada}

\author{Filip Bro\'ci\'c}
\email{filipvbrocic@gmail.com}
\author{Dylan Cant}
\email{dylan@dylancant.ca}

\author{Egor Shelukhin}
\email{egor.shelukhin@umontreal.ca}

\date{\today}
\begin{document}

\maketitle

\begin{abstract}
We prove Arnol'd's chord conjecture for all Legendrian submanifolds of cosphere bundles of closed manifolds isotopic to conormal bundles of closed submanifolds. Our method of proof involves an isomorphism between wrapped Floer cohomology and the homology of a path space with coefficients in a local system and a twisted version of the Hurewicz theorem.
\end{abstract}

\section{Introduction}
\label{sec:introduction}

Arnol'd's chord conjecture states that for every closed Legendrian submanifold $\Lambda$ of a closed co-oriented contact manifold $(Y,\xi)$ there should exist a non-contstant Reeb chord with respect to every contact structure. It appears for example in \cite{Arnold-steps}. More specifically, given a nowhere-vanishing one-form $\alpha$ on $Y$ with $\xi = \ker(\alpha),$ there exists a unique Reeb vector field $R$ determined uniquely by the properties $\alpha(R) = 1,$ $\mathcal{L}_{R} \alpha = 0.$ Arnol'd's chord conjecture is the statement that there should exist $T>0$ and a chord $c:[0,T] \to Y$ with $c(0), c(T) \in \Lambda$ and $\dot{c}(t) = R(c(t))$ for all $t \in [0,T]$.

This conjecture was proved in its original case of the standard contact three-sphere $S^3$ by Mohnke in \cite{mohnke_chord}. The proof also applies to the standard $S^{2n+1}$ and more generally to contact boundaries of subcritical Weinstein domains. Related work by Abbas \cite{abbas99,abbas99-duke, abbas04} and Cieliebak \cite{cieliebak_handle_chord} proves similar results under additional conditions on contact forms and Legendrians. Early work by Ginzburg and Givental \cite{Givental-nonlinear} proves the conjecture for standard Legendrian $\Lambda = \R P^n$ in the standard contact real projective space $Y = \R P^{2n+1}.$ Moreover, work \cite{ht11,ht13} of Hutchings and Taubes establishes this conjecture for all Legendrian knots in all closed contact three-manifolds $Y.$ Other work related to this conjecture includes \cite{EES-duality, Ziltener-chord, CieliebakMohnke-punctured, reinke20, chantraine-slices, zhou-minimal, kang-strong, allais-arlove} and the references in the discussion below.

The goal of this paper is to establish Arnol'd's chord conjecture for the conormal Legendrian $\Lambda_{N}\subset Y=S^{*}M$ associated to a submanifold $N\subset M$. Here $S^{*}M$ is the spherical cotangent bundle and $\Lambda_N$ is the spherization of the conormal Lagrangian: $$\nu^* N = \{ (p,q)\;|\; q\in N,\; p|_{T_q N} = 0 \}\subset T^{*}M.$$ In other words, $S^*M = (T^*M\setminus 0_M)/\R_{+}$ and $\Lambda_{N}= (\nu^* N \setminus 0_M)/\R_{+}$ are quotients of the non-zero parts of $T^*M, \nu^*N$ by the natural $\R_{+}$-action on $T^*M.$

\begin{theorem}\label{thm: chord}
Let $N \subset M$ be a closed submanifold of a closed manifold with positive codimension. Arnol'd's chord conjecture holds for $\Lambda_N$ in $S^*M.$
\end{theorem}

Importantly, the statement holds for arbitrary contact forms. In particular, it applies equally well to every $\Lambda \subset Y$ Legendrian isotopic to $\Lambda_N.$

Theorem \ref{thm: chord} generalizes a number of results in Riemannian and Finsler geometry, which correspond to specific choices of contact forms. For example, for a contact form coming from the embedding $S^*_g M \subset T^*M$ of the bundle of unit vectors of a Riemannian metric $g$ on $M,$ Arnol'd's chord conjecture for the conormal lift of a submanifold $N$ of $M$ specializes to a 1970s result of Grove \cite{grove_general_boundary_conditions, grove_condition_C}. Related results were obtained in \cite{asselle-euler-lagrange, alberto_figalli} for Hamiltonian systems on $T^*M$ which are convex in the fiber coordinates, such as Tonelli Hamiltonian, and for Finsler metrics in particular. Finally, special cases of our main result, which hold for all contact forms, are known to hold for $N$ equal to a point, so $\Lambda_N$ equal to a fiber, by \cite{viterbo_functors_and_computations_1, abbondandolo_schwarz, abbondandolo_schwarz_2, abbondandolo_schwarz_3, abouzaid_based_loop, abouzaid_monograph, salamon_weber_floer_homology_heat_flow} and for other pairs $N, M$ satisfying various additional topological conditions by \cite{pushkar, merry-relative-rabinowitz, ritter_tqft, abbondandolo_portaluri_schwarz, zhou-minimal}.

We expect that an extension of our methods would prove that if $\Lambda = \Lambda_N$ and $\Lambda'$ is Legendrian isotopic to $\Lambda$, then there exists a Reeb chord joining $\Lambda$ and $\Lambda'$. This happens if $\Lambda' = \Lambda_{N'}$ for $N'$ smoothly isotopic to $N$ in $M.$ The case when $N, N'$ are two distinct points in $M$ was studied amply in the literature (see the recent work \cite{allais_growth_rate_chords} in the Riemannian case for a review). A similar result does not hold for general conormals $\Lambda$ and $\Lambda'.$ For example, set $M=T^2$ and let $\Lambda = \Lambda_N,$ $\Lambda'=\Lambda_{N'}$ for $N=S^1 \times \{0\} \subset T^2$ and $N' = \{0\} \times S^1.$ Then there are no Reeb chords between $\La$ and $\La'$ for the contact structure on $S^*T^2$ induced by the standard Riemannian metric on $T^2$. Another natural extension of our result would be to require that the Reeb chords satisfy more general non-local boundary conditions as in \cite{grove_general_boundary_conditions, abbondandolo_portaluri_schwarz}. In particular, if for two cleanly intersecting submanifolds $N, N'$ of $M$ the inclusion $N \cap N' \to \cl P_{N,N'},$ the space of paths from $N$ to $N',$ is not a homotopy equivalence, then there exists a Reeb chord from $\La_{N}$ to $\La_{N'},$ and vice versa, for every contact form. Finally, as explained to us by M. Mazzucchelli, another generalization to characteristic chords with respect to exact magnetic symplectic forms on sufficiently large star-shaped hypersurfaces is readily available. We expect to make progress in these directions in a future publication.

We now present a brief outline of our argument. Its main novelty consists in retaining sufficient homotopical information about the relevant path space in terms of homology with coefficients in local systems, which is then readily accessible via homological invariants in symplectic topology which detect Reeb chords.

First, we may always assume that $M$ is orientable by passing to its orientation double cover if it is not (see \cite[pp.\,554]{abbondandolo-macarini-mazzucchelli-paternain} for a similar application of this trick): this induces a double cover of $S^*M$ and we focus on a connected component $K$ of the preimage of $\Lambda = \Lambda_N;$ Reeb chords on $K$ with respect to the pull-back contact form project to the required Reeb chords on $\Lambda$.

Second, inspired by \cite{abbondandolo_portaluri_schwarz}, we prove a version of the Viterbo isomorphism in our setting, taking coefficients in a local system into account. Namely, consider the conormal Lagrangian $L=\nu^* N \subset T^*M.$ Let $\cl P = W^{1,2}(([0,1], \{0,1\}), (M,N))$ be the space of paths in $M $ with endpoints on $N,$ the Sobolev completion taken for technical convenience only. Let $\cl L$ be a local system over $\cl P.$ It corresponds to a local system on the space of paths from $L$ to $L$ in $T^*M,$ which we denote by $\cl L$ again. The positive wrapped Floer cohomology, an invariant introduced by Abouzaid-Seidel \cite{abouzaid_seidel_open_string_analogue}, of the cylindrical Lagrangian $L$ in the Liouville manifold $T^*M,$ has an enhancement $\mathrm{HW}_{+}(L;\cl L)$ with local coefficients in $\cl L.$ It has the property that if $\Lambda_N$ has no Reeb chords, then $\mathrm{HW}_{+}(L;\cl L) = 0.$ We prove that if $N$ is considered as the subspace of constant paths in $\cl P,$ and $H(\cl P, N; \cl L)$ denotes the relative homology with coefficients in $\cl L,$ then the following holds.

\begin{theorem}
  There is an isomorphism \[\mathrm{HW}_{+} (L;\cl L) \cong H(\cl P, N; \cl L).\]
\end{theorem}

Finally, we prove that $H(\cl P, N; \cl L) \neq 0$ for a certain local system $\cl L$ which finishes the proof. This non-vanishing is proved via an application of a twisted version of the relative Hurewicz theorem as follows.

Arguing by contradiction, suppose that for all local systems $\cl L,$ $H(\cl P, N; \cl L) = 0$. Using the twisted relative Hurewicz theorem we deduce below that the inclusion $i: N \to \cl P$ is a weak homotopy equivalence. This in turn implies that the inclusion $N\to M$ is a weak homotopy equivalence. This can be seen by considering the Serre fibrations (i) $\mathrm{ev}_{0}:\cl P \to N$, whose fiber over $x \in N$ is denoted $\mathscr{P}_{x}$, and (ii) $\mathrm{ev}_{1}:\cl P_{x} \to N$ whose fiber over $x$ equals the based loop space $\Om M$ of $M.$ However the inclusion of $N$ into $M$ cannot be a weak homotopy equivalence because $H_{n}(N)=0$ and $H_{n}(M)\ne 0$ for $n=\mathrm{dim}(M)$.

To prove that $i$ is a weak homotopy equivalence, we first observe that $i$ is injective on all $\pi_k,$ $k \geq 0,$ as it has a continuous left inverse given by the evaluation map ${\mathrm{ev}_0: \cl P \to N},$ $\mathrm{ev}_0(\gamma) = \gamma(0).$ Next observe that $i$ is an isomorphism on the set $\pi_0$ of connected components, as otherwise $H_0(\cl P, N; \F_2) \neq 0,$ which contradicts the contrapositive assumption. To prove that $i$ induces an isomorphism on $\pi_1$ we consider the natural local system $\cl L$ on $\cl P$ with fiber: $$\cl L_{\gamma} = H_0(p^{-1}(\gamma);\Z) \cong \Z[\pi_1(\cl P, \gamma)],$$ where $p: \til{\cl P} \to \cl P$ is the universal covering. Then
\begin{equation*}
  \begin{aligned}
    H_0(\cl P; \cl L) &= H_0(\til{\cl P};\Z) = \Z,\\ H_0(N; i^*\cl L) &= \Z[\pi_1(\cl P)/ i_*(\pi_1 (N))],
  \end{aligned}
\end{equation*}
while the long exact sequence of a pair yields \[ H_1(\cl P, N; \cl L) \to H_0(N; i^*\cl L) \to H_0(\cl P;\cl L).\] If $i$ is not surjective on $\pi_1,$ then the map $H_0(N; i^*\cl L) \to H_0(P; \cl L) = \Z$ has a non-trivial kernel, whence $H_1(\cl P, N; \cl L) \neq 0.$

The rest of the argument is completed by induction: assume that $i$ is surjective on $\pi_{\ell}$ for $\ell<k$, equivalently, $\pi_{\ell}(\mathscr{P},N)=0$ for $\ell<k$. We will show that $\pi_k(\cl P, N) = 0$ and hence $i$ is surjective on $\pi_k.$  Otherwise, suppose that $k \geq 2$ is the first integer such that $\pi_k(\cl P, N) \neq 0.$ Then, as $i$ is an isomorphism on $\pi_1,$  by the twisted Hurewicz theorem \cite[Proposition 2.1]{rong} (see \S\ref{sec:choice-local-system} for further discussion), \[\pi_k(\cl P, N) \cong \pi_k(\til{\cl P}, \til{N}) \cong H_k(\cl P, N; \cl L),\] where $\cl L$ is the same local system as in the previous step. Therefore $H_k(\cl P, N; \cl L) \neq 0,$ which finishes the proof.

We end this introduction with a brief comparison with previous results. First, the approach of \cite {grove_general_boundary_conditions, grove_condition_C, asselle-euler-lagrange} involved classical calculus of variations, which is not accessible for general Reeb flows. Second, it is easy to see that the pair of manifolds $M=\C P^{n} \times S^{k}$ and $N = \C P^{n} \times \{\mathrm{pt}\}$, where $1\leq k \leq 2n-1$, does not satisfy the topological conditions from \cite{pushkar, merry-relative-rabinowitz, ritter_tqft, zhou-minimal, EES-duality, chantraine-slices}. We expect that for most conormal bundles, the condition of \cite{allais-arlove} does not hold, and in cases where it does, verifying it would require arguments similar to those in this paper. Finally, the results of \cite{alberto_figalli, abbondandolo_portaluri_schwarz} should imply the chord conjecture if $H_{*}(N; \Z) \to H_{*}(\mathscr{P}; \Z)$ is not an isomorphism. However, it is not known if this condition holds for every pair $M, N$. From this point of view, our paper shows that such a homological condition always holds when local systems are taken into account, for homotopical reasons. We expect this method to be applicable to other questions in symplectic and contact topology. For example, in \cite{BCSS}, it is used to prove the following generalization of the classical result of Lyusternik and Fet \cite{lyusternik-fet}:
\begin{theorem}
  Let $M$ be a closed manifold that is not a $K(\pi,1).$ Then $S^*M$ admits a contractible closed Reeb orbit for every contact form $\al.$
\end{theorem}
For related work see for instance \cite[Corollary 1]{hofer-viterbo}.

\section*{Acknowledgments}
We thank Alberto Abbondandolo, Paul Biran, Octav Cornea, Marco Mazzucchelli, Leonid Polterovich, and Frol Zapolsky for useful communications.
F.B. was supported by an ISM graduate scholarship.
D.C. was supported by a CIRGET postdoctoral fellowship.
E.S. was supported by an NSERC Discovery Grant, the Fondation Courtois, and a Sloan Research Fellowship.

\subsection{Local systems of coefficients}
\label{sec:local-coefficients}

Since this notion is central to the paper, we recall the definition of a local system on a topological space. Let $\mathscr{P}$ be a locally path-connected and locally simply-connected space. A \emph{local system of coefficients} is a functor $\mathscr{L}$ from the fundamental groupoid of $\mathscr{P}$ to the category of free $\Z$-modules. More pedantically, for each point $x\in \mathscr{P}$, one has a free $\Z$-module $\mathscr{L}_{x}$, and for every homotopy class of paths joining $x$ to $y$ there is an associated \emph{monodromy isomorphism} $\mathscr{L}_{x}\to \mathscr{L}_{y}$; these monodromy isomorphisms are required to be compatible with concatenation of homotopy classes of paths.

The algebraic invariants we consider, e.g., Floer cohomology of the conormal, Morse homology of the energy functional, will be defined with respect to such a local coefficient system on the space of paths $\mathscr{P}$ joining $N$ to $N$.

See \cite[\S2.7]{groman-merry-v4} for related discussion.

\subsection{Structure of the paper}
\label{sec:outline-main-argum}

Our approach is based on a comparison between the Morse homology of the Riemannian energy functional and the wrapped Floer cohomology of a conormal Lagrangian, similar to the work of \cite{abbondandolo_schwarz,abouzaid_based_loop,abbondandolo_portaluri_schwarz,abouzaid_monograph}.

For any Riemannian metric, one can consider the Riemannian energy functional $\mathscr{E}$ on the space of paths $\mathscr{P}$; for generic metrics, this will be a Morse functional except for the critical submanifold of constant paths $N\subset \mathscr{P}$ which represents the minimum. In \S\ref{sec:morse-homol-energy} we define the Morse homology $\mathrm{HM}(\mathscr{P},N;\mathscr{E};\mathscr{L})$ of the Riemannian energy functional with local coefficients generated by critical points of $\mathscr{E}$ with positive energy. This homology is a Morse theoretic version of relative singular homology $H(\mathscr{P},N;\mathscr{L})$, similar to the Morse theoretic version of the homology of the path space in \cite{milnor_morse_theory}. Various facts about Morse theory for the Riemannian energy are relegated to Appendix \ref{sec:appendix-B}.

In \S\ref{sec:wrapped-floer-cohom} we define the positive-wrapped Floer cohomology $\mathrm{HW}_{+}(\nu^{*}N;\mathscr{L})$. Well-known arguments show that $\mathrm{HW}_{+}(\nu^{*}N;\mathscr{L})\ne 0$ implies there are Reeb chords of the ideal Legendrian boundary of $\nu^{*}N$, for any choice of contact form.

\begin{theorem}\label{theorem:main-iso}
  There is an isomorphism $\mathrm{HM}(\mathscr{P},N;\mathscr{E};\mathscr{L})\to \mathrm{HW}_{+}(\nu^{*}N;\mathscr{L})$.
\end{theorem}
The isomorphism is constructed in \S\ref{sec:isom-from-morse}, following \cite{abbondandolo_schwarz}, on the chain level by counting certain rigid half-infinite Floer strips.

In \S\ref{sec:choice-local-system}, we explain why there is always \emph{some} choice of local system $\mathscr{L}$ so that $\mathrm{HM}(\mathscr{P},N;\mathscr{E};\mathscr{L})\ne 0$. Consequently we conclude that the ideal Legendrian boundary of $\nu^{*}N$ has a Reeb chord for any choice of contact form.

In Appendix \ref{sec:coher-orient} we give a self-contained account of the relevant theory of orientations to prove our results over the integers.

\section{Proofs}
\label{sec:proofs}

\subsection{Morse homology for the energy functional}
\label{sec:morse-homol-energy}
Fix a Riemannian metric $g$ on $M$ and define the \emph{Riemannian energy functional} on the space of paths $\mathscr{P}$ joining $N$ to $N$ by the formula:
\begin{equation*}
  \mathscr{E}(q(t)):=\int_{0}^{1}g(q'(t),q'(t))\d t;
\end{equation*}
see \cite[\S12]{milnor_morse_theory}. For a generic Riemannian metric, $\mathscr{E}$ is a Morse function except for the Morse-Bott submanifold $N$ of the constant paths, where the minimum of $\mathscr{E}$ is attained; see \S\ref{sec:gener-metr-morse}.

Given a local coefficient system $\mathscr{L}$ on the path space $\mathscr{P}$, define a relative Morse complex:
\begin{equation*}
  \mathrm{CM}(\mathscr{P},N;\mathscr{E};\mathscr{L})
\end{equation*}
as the direct sum of $\mathfrak{o}_{x} \otimes \mathscr{L}_{x}$ where $x$ ranges over the critical points of $\mathscr{E}$ with positive Riemannian energy and $\mathfrak{o}_{x}$ is the orientation line corresponding to the unstable (negative) eigenspace of the Hessian at $x$. Considering only chords with positive Riemannian energy excludes the Morse-Bott submanifold $N$, and the homology should be considered as a homology of $\mathscr{P}$ relative $N$.

The differential on $\mathrm{CM}(\mathscr{P},N;\mathscr{E};\mathscr{L})$ is the homological Morse differential twisted by the action of the monodromy of $\mathscr{L}$. There are various ways to endow $\mathrm{CM}$ with a differential using Morse theory. In this text we follow \cite[\S16]{milnor_morse_theory} and \cite{abouzaid_monograph} and use a finite dimensional approximation to $\mathscr{P}$. Alternatives are the $W^{1,2}$-gradient flow approach of \cite{abbondandolo_majer_morse_theory,abbondandolo_schwarz,abbondandolo_schwarz_smooth_pseudogradient,abbondandolo_portaluri_schwarz} and the heat-flow approach of \cite{salamon_weber_floer_homology_heat_flow,weber_morse_homology_heat_flow_1,weber_morse_homology_heat_flow_2,weber_morse_homology_heat_flow_3}.

In the finite dimensional approach, it is more convenient to work with the subcomplex $\mathrm{CM}_{\ell}$ given as the direct sum of $\mathfrak{o}_{x} \otimes \mathscr{L}_{x}$, where the length of $x$ is at most $\ell$. The recipe in \cite[\S16]{milnor_morse_theory} provides a simple way to endow $\mathrm{CM}_{\ell}$ with a pseudogradient in such a way that there are isomorphisms $H_{*}(\mathscr{E}\le \ell^{2},N;\mathscr{L})\to \mathrm{HM}_{\ell}$ which are natural with respect to $\ell$ and the choice of the pseudogradient. One defines $\mathrm{HM}(\mathscr{P},N;\mathscr{E},\mathscr{L})$ as a colimit of $\mathrm{HM}_{\ell}$ as $\ell$ increases. Since $H_{*}(\mathscr{P},N;\mathscr{L})$ is the colimit of the homologies $H_{*}(\mathscr{E}\le \ell^{2},N;\mathscr{L})$, one concludes an isomorphism:
\begin{equation}
  H_{*}(\mathscr{P},N;\mathscr{L})\simeq \mathrm{HM}(\mathscr{P},N;\mathscr{E},\mathscr{L}).
\end{equation}

The rest of this section explains the construction of the Morse homology with more details.

\subsubsection{Finite dimensional approximation}
\label{sec:finite-dimens-appr}
Our setup follows \cite[\S16]{milnor_morse_theory} closely.

Given $K\in \mathbb{N}$, let $\mathscr{P}_{K}\subset \mathscr{P}$ be the set of piecewise smooth paths $q$ so that, for $j=1,\dots,K$ and $t_{j}=j/K$,
\begin{enumerate}
\item $\mathrm{dist}_{g}(q(t_{j-1}),q(t_{j}))<\mathrm{injrad}(M,g)$, and,
\item the restriction to each interval $[t_{j-1},t_{j}]$ is the minimal geodesic parameterized with constant speed.
\end{enumerate}
The first condition implies that $\mathscr{P}_{K}$ is identified with an open subset of the product $N\times M\times \dots\times M\times N$ via the evaluation at $t_{0},\dots,t_{K}$, and hence naturally has a manifold structure. There is an obvious risk of failure of compactness; however, the following lemma ensures some sort of compactness holds if the Riemannian energy $\mathscr{E}$ is bounded and $K$ is sufficiently large.
\begin{lemma}
  Suppose that $q_{n}\in \mathscr{P}_{K}$ is a sequence satisfying $\mathscr{E}(q_{n})<K\rho^{2}$ for some $\rho<\mathrm{injrad}(M,g)$. Then $q_{n}$ has a convergent subsequence in $\mathscr{P}_{K}$.
\end{lemma}
\begin{proof}
  It is sufficient to prove that $\mathrm{dist}(q_{n}(t_{j-1}),q_{n}(t_{j}))\le \rho$ for all $j$. If this fails for some $j$, then the contribution to $\mathscr{E}(q_{n})$ due to $[t_{j-1},t_{j}]$ is $K\rho^{2}$, a contradiction. This completes the proof.
\end{proof}

Consequently, if $\ell^{2}$ is a regular value of the restriction $\mathscr{E}|_{\mathscr{P}_{K}}$, and $\ell^{2}<K\mathrm{injrad}(M,g)^{2}$, then the sublevel set:
\begin{equation*}
  \mathscr{P}^{\ell}_{K}:=\textstyle\set{\mathscr{E}\le \ell^{2}}\cap \mathscr{P}_{K}\hspace{1cm}\bd\mathscr{P}^{\ell}_{K}=\textstyle\set{\mathscr{E}=\ell^{2}}\cap \mathscr{P}_{K},
\end{equation*}
is a compact manifold with boundary. In this case we say $\mathscr{P}^{\ell}_{K}$ is \emph{admissible}.

If $\mathscr{P}^{\ell}_{K}$ is admissible, then all geodesic chords whose length is at most $\ell$ lie in $\mathscr{P}^{\ell}_{K}$, since each subdivided arc has length $\ell/K<\mathrm{injrad}(M,g)$, and hence satisfies (i), while (ii) is tautologically true.

Moreover, admissibility implies the dimension of the negative eigenspace of the Hessian of $\mathscr{E}|_{\mathscr{P}^{\ell}_{K}}$ at a geodesic chord $c(t)$ equals the full Morse index of $c(t)$. Otherwise one could find a variation which fixes $c(t_{j})$ for $j=0,\dots,K$ whose length decreases at second order. However, since $c([t_{j-1},t_{j}])$ is a minimal geodesic, such a variation cannot exist.

Finally note that there are natural \emph{subdivision embeddings} $\mathscr{P}^{\ell}_{K}\to \mathscr{P}^{\ell'}_{dK}$, for $d\in \mathbb{N}$, if $\ell\le \ell'$ and both $\ell,K$ and $\ell',dK$ are admissible. In the case $\ell=\ell'$, subdivision sends $\bd\mathscr{P}^{\ell}_{K}$ into $\bd\mathscr{P}^{\ell}_{dK}$. In the case $d=1$, the morphism is simply the codimension $0$ inclusion of the sublevel set $\mathscr{P}^{\ell}_{K}\to \mathscr{P}^{\ell'}_{K}$.

\subsubsection{Admissible pseudogradients}
\label{sec:admiss-pseud}

Let $\mathscr{P}^{\ell}_{K}$ be an admissible finite dimensional approximation. An \emph{admissible pseudogradient} is a smooth vector field $V$ which is \emph{gradient-like} for $-\mathscr{E}|_{\mathscr{P}^{\ell}_{K}}$, in the sense of \cite[\S3]{milnor_hcob}, and so that the moduli space of trajectories between any two critical points is cut out transversally. On the submanifold of constant paths $N\subset \mathscr{P}^{\ell}_{K}$ no condition is enforced on $V$ except that it is an attracting invariant set, which follows from the negative gradient-like condition.

\subsubsection{Morse differential twisted by local coefficients}
\label{sec:morse-differential-with-loc-coeffs}

To each admissible pair $(\mathscr{P}^{\ell}_{K},V)$ one puts a differential on the $\Z$-module $\mathrm{CM}_{\ell}$ by counting flow lines following, e.g., \cite{schwarz_book}, incorporating the monodromy of the local system $\mathscr{L}$ as in \cite[(7.1)]{oancea_twisted}. The resulting complex is denoted $\mathrm{CM}_{\ell}(V)$, and its homology is denoted $\mathrm{HM}_{\ell}(V)$.

To fix the conventions, we describe the differential with slightly more details.

Let $S(x)$ be a small unstable sphere around $x$, oriented as the positive boundary of an unstable disk around $x$. The pseudogradient is chosen so that, if $z\in S(x)$ converges to a critical point $y$, then $T_{z}S(x)$ is transverse to the tangent space of the stable manifold of $y$; see Figure \ref{fig:heart-shaped-sphere}.

In particular, if $\mathrm{ind}(x)=\mathrm{ind}(y)+1$, then $T_{z}S(x)$ and the unstable eigenspace of the Hessian at $y$ are both linear complements to tangent spaces of the stable manifold of $y$. This induces a morphism:
\begin{equation*}
  \bd_{z}:\mathfrak{o}_{x}\to \mathfrak{o}(T_{z}S(x))\to \mathfrak{o}_{y},
\end{equation*}
as a composition of two maps. The first map is defined by requiring that $T_{z}S(x)$ is oriented as the positive boundary of the unstable disk at $x$. The second map is defined by picking a frame $F_{1}$ of $T_{z}S(x)$ and a frame $F_{2}$ of the unstable space at $y$, adding to these frames any frame of the tangent space of stable manifold of $y$, and comparing resulting orientations of the full tangent space at $y$. If the orientations agree, say that $[F_{1}]\mapsto [F_{2}]$; otherwise say that $[F_{1}]\mapsto -[F_{2}]$; here we recall that a choice of frame $F$ for a finite dimensional vector space $E$ induces an element $[F]\in \mathfrak{o}(E)$.

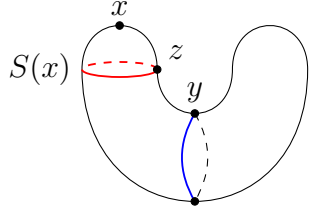
\begin{figure}[H]
  \centering
  \begin{tikzpicture}
    \path (0,0)coordinate(X)--(1,0)coordinate(Y) (2,0)coordinate(Z)--(3,0)coordinate(W);
    \draw (X)to[out=90,in=90,looseness=2]coordinate(E)(Y) (Z)to[out=90,in=90,looseness=2](W) (Y)to[out=-90,in=-90,looseness=2]coordinate(B)(Z) (X)to[out=-90,in=-90,looseness=2]coordinate(C)(W);
    \draw[blue,line width=.7pt](B)to[out=240,in=120](C);
    \draw[dashed](B)to[out=300,in=60](C);
    \draw[dashed,red,line width=0.7pt] (Y) arc (0:180:0.5 and 0.1);
    \draw[red,line width=0.7pt] (Y) arc (0:-180:0.5 and 0.1);
    \path[every node/.style={fill,inner sep=1pt,black,circle,draw}] (C)node{} (B)node{} (E)node{} (Y)node{};
    \path (Y)node[above right]{$z$} (E)node[above]{$x$} (B) node[above]{$y$};
    \node at (X) [left]{$S(x)$};
  \end{tikzpicture}
  \caption{The unstable sphere $S(x)$ contains the point $z$ which lies on the flow line joining $x$ to $y$.}
  \label{fig:heart-shaped-sphere}
\end{figure}

In general, every flow line starting at $x$ and ending at an index difference $1$ critical point $y$ intersects $S(x)$ in a unique point $z$; let us denote by $\mathscr{M}_{1}(x)$ the space of such $z$, and let $y(z)$ be the asymptotic critical point.

The flow line from $x$ to $y$ passing through $z$ induces a monodromy from $\mathscr{L}_{x}$ to $\mathscr{L}_{y}$; let us call this monodromy $\mathscr{L}(z)$.

Define the Morse differential by the formula:
\begin{equation*}
  \bd=\sum_{x}\sum_{z\in \mathscr{M}_{1}(x)} \bd_{z}\otimes \mathscr{L}(z):\sum_{x}\mathfrak{o}_{x}\otimes \mathscr{L}_{x}\to \sum_{y}\mathfrak{o}_{y}\otimes \mathscr{L}_{y};
\end{equation*}
Well-known arguments imply that $\bd^{2}=0$. The local coefficients do not complicate the proof that $\bd^{2}=0$, and the technical aspect is analyzing the maps on orientation lines.

\subsubsection{Continuation lines}
\label{sec:continuation-lines}

If $\ell\le \ell'$ and $(\mathscr{P}^{\ell}_{K},V)$ and $(\mathscr{P}^{\ell'}_{dK},V')$ are admissible, then there is a canonical morphism $\mathrm{HM}_{\ell}(V)\to \mathrm{HM}_{\ell'}(V')$ defined by counting continuation lines with respect to the subdivision embeddings from \S\ref{sec:finite-dimens-appr}; see \cite[\S11.3.2]{abouzaid_monograph} for similar discussion.

To define the morphism, choose a time-dependent vector field $\delta_{s}$ on $(\mathscr{P}^{\ell'}_{dK},V')$, so that (i) $\delta_{s}=0$ for $s\le 0$ and $s\ge 1$, (ii) $\delta$ vanishes on a neighborhood of the critical points, and, (iii) $V'+\delta_{s}$ is negative gradient-like for each $s$. One thinks of $\delta_{s}$ as a small perturbation term.

One then considers the moduli space $\mathscr{M}(\delta_{s})$ of flow lines $\R\to \mathscr{P}^{\ell'}_{dK}$ for the piecewise smooth $s$-dependent vector field illustrated in Figure \ref{fig:continuation-lines}, requiring that the flow lines lie in $\mathscr{P}^{\ell}_{K}$ for $s\le 0$.

Given an input critical point $x$, one counts flow lines whose output critical point $y$ has the same Morse index as $x$. Notice that, for any continuation line $\xi(s)$ whose input is $x$, the evaluation $\xi(0)$ lies on the unstable manifold of $x$. Let $\Sigma(x;\delta)$ be the set of points $\xi(1)$, where $\xi:(-\infty,1]$ solves the ODE shown in Figure \ref{fig:continuation-lines}. One thinks of $\Sigma(x;\delta)$ as a perturbation of the unstable manifold of $x$.

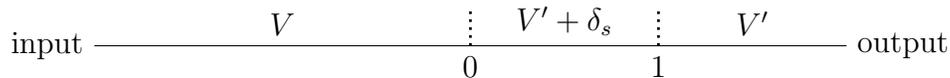
\begin{figure}[H]
  \centering
  \begin{tikzpicture}[xscale=2.5]
    \draw (-2,0)node [left] {input}--(2,0)node [right] {output};
    \draw[line width=1pt,dotted] (0,0)--+(0,0.5) (1,0)--+(0,0.5);
    \path (0,0)node[below]{$0$}--(1,0)node[below]{$1$};
    \path (-1,0)node[above]{$V$}--(0.5,0)node[above]{$V'+\delta_{s}$}--(1.5,0)node[above]{$V'$};
  \end{tikzpicture}
  \caption{Continuation line $\xi(s)$ from $(\mathscr{P}^{\ell}_{K},V)$ to $(\mathscr{P}^{\ell'}_{dK},V')$.}
  \label{fig:continuation-lines}
\end{figure}

The map can therefore be defined by considering half-infinite flow lines: $$\xi:[1,\infty)\to \mathscr{P}^{\ell'}_{dK}$$ starting at $\Sigma(1;\delta)$, for the pseudogradient $V'$. The perturbation term should be chosen generically so that $\Sigma(1;\delta)$ is transverse to the stable manifolds of $V'$.

By counting dimensions, we conclude that the asymptotics of a continuation line satisfy $\mathrm{ind}(x)\geq\mathrm{ind}(y)$. Consequently, we obtain a compactness result for the flow lines with $\mathrm{ind}(x)=\mathrm{ind}(y)$.

Moreover, when $\mathrm{ind}(x)=\mathrm{ind}(y)$, the dimension of $\Sigma(1;\delta)$ equals the codimension of the stable manifold of $y$ determined by $V'$. The same scheme as in \S\ref{sec:morse-differential-with-loc-coeffs} defines a map on orientation lines $\mathfrak{c}_{\xi}:\mathfrak{o}_{x}\to \mathfrak{o}_{y}$. The monodromy along the continuation line $\xi$ induces a map $\mathscr{L}(\xi)$ on local coefficients, and the continuation map is the sum of $\mathfrak{c}_{\xi}\otimes \mathscr{L}(\xi)$ as $\xi$ ranges over all continuation lines with $\mathrm{ind}(x)=\mathrm{ind}(y)$.

Well-known arguments in Morse theory show that $\mathfrak{c}$ is a chain map with respect to the Morse differentials defined by $V$ and $V'$; essentially the idea is to consider the moduli space of continuation lines with $\mathrm{ind}(x)=\mathrm{ind}(y)+1$; see, e.g., \cite{schwarz_book,audin-damian,abouzaid_monograph} for similar arguments.

\subsubsection{Definition of the full Morse homology}
\label{sec:defin-full-morse}

Consider the small category whose objects are pairs $(\mathscr{P}^{\ell}_{K},V)$ of admissible finite-dimensional approximations and pseudogradients, so that there is a morphism between objects $(\mathscr{P}^{\ell}_{K},V)$ and $(\mathscr{P}^{\ell'}_{K'},V')$ if and only if $\ell\le \ell'$ and $K'=dK$ for some $d\in \mathbb{N}$.

Then $(P^{\ell}_{K},V)\mapsto \mathrm{HM}_{\ell}(V)$ is a functor to the category of $\Z$-modules, where the morphisms are given by counting continuation lines for any admissible perturbation $\delta_{s}$.

Define the full Morse homology $\mathrm{HM}(\mathscr{P},N;\mathscr{E};\mathscr{L})$ to be the colimit of this functor. It is also interesting to consider the colimit of the restriction to the subcategory spanned by all objects with fixed length parameter $\ell$; the result is called $\mathrm{HM}_{\ell}(\mathscr{P},N;\mathscr{E};\mathscr{L})$.

A straightforward ``diagonal argument'' proves that the natural map: $$\mathrm{HM}_{\ell}(V)\to \mathrm{HM}_{\ell}(\mathscr{P},N;\mathscr{E};\mathscr{L})$$ is an isomorphism for every $V$; indeed, any continuation map between complexes with the same length parameter $\ell$ is an isomorphism since it equals $1+T$ where $T$ strictly decreases length (the inverse is $1-T+T^{2}-T^{3}+\dots$, which terminates in finitely many steps). One should think of the continuation map as being upper triangular with $1$s on the diagonal with respect to a basis ordered by length. A similar diagonal argument is employed in \S\ref{sec:diagonal-argument-theta} in the proof of Theorem \ref{theorem:main-iso}.

\subsubsection{Comparison with relative singular homology}
\label{sec:comp-with-sing}

We explain how to relate singular homology to Morse homology following a similar approach to \cite[\S4.2]{schwarz_pseudocycles}.

A simplex $u:\Delta\to \mathscr{P}^{\ell}_{K}$ is said to be in \emph{general position} with respect to a pseudogradient $V$ provided it is smooth and all faces are transverse to the stable manifolds determined by $V$.

Let $C(\mathscr{P}^{\ell}_{K},N;V;\mathscr{L})\subset C(\mathscr{P}^{\ell}_{K},N;\mathscr{L})$ be the subcomplex of the relative singular homology complex spanned by simplices which are in general position to $V$; here we recall that:
\begin{equation*}  C(\mathscr{P}^{k}_{\ell},N;\mathscr{L}):=C'\otimes_{\Z[\pi_{1}(\mathscr{P}^{\ell}_{K},\mathrm{pt})]}\mathscr{L}_{\mathrm{pt}},
\end{equation*}
where $C'$ is spanned by relative singular simplices in $(\mathscr{P}^{\ell}_{K},N)$ with a choice of lift to the universal cover of $\mathscr{P}^{\ell}_{K}$. The right module structure on $C'$ is given by deck transformations: if $(x_{1},[x_{t}])$ represents an element of the universal cover, so that $x_{t}$ is a path with $x_{0}=\mathrm{pt}$, and $g\in \pi_{1}(\mathscr{P},\mathrm{pt})$, then $(x_{1},[x_{t}])g=(x_{1},[(x\#g)_{t}])$ defines a right action on $\cl P$ (which preserves $N'$), and hence a right $\Z[\pi_{1}(\mathscr{P},\mathrm{pt})]$-module structure on the singular chain complex $C'$. See \S\ref{sec:singular-chains-cover} for further discussion.

In general, for any countable union $S$ of maps of manifolds to $\mathscr{P}^{\ell}_{K}$, one can consider the subcomplex spanned by simplices in general position to $S$. It is well-known that the inclusion of this subcomplex into the full singular complex is a quasi-isomorphism; one uses \cite[Theorem 18.7]{lee-smooth} to reduce to the case of smooth simplices and then applies standard transversality arguments.

There is a natural map $C(\mathscr{P}^{\ell}_{K},N;V;\mathscr{L})\to \mathrm{CM}_{\ell}(V)$ given by counting rigid trajectories twisted by the monodromy of $\mathscr{L}$, as follows: for each $u:\Delta\to \mathscr{P}^{\ell}_{K}$ in general position, and with a choice of lift to the universal cover, one obtains a trivialization:
\begin{equation*}
  \mathscr{L}_{\mathrm{pt}}\to u^{*}\mathscr{L},
\end{equation*}
given by the monodromy of $\mathscr{L}$ to the chosen basepoint, using the lifts to the universal cover to select a canonical homotopy class of paths.

Let $\mathscr{M}(u)$ be the moduli space of pairs of $(x,q)$ where the a trajectory of $V$ starting at $x$ is asymptotic to the critical point $q$ of $\mathscr{E}$ of positive length, and let $\mathscr{M}_{k}(u)$ be the piece where $\dim(u)=\mathrm{index}(q)+k$. Standard arguments, e.g., those in \cite{milnor_hcob,milnor_morse_theory}, imply that $\mathscr{M}_{0}(u)$ is a finite set whose projection to $\Delta$ lies strictly in the interior, and $\mathscr{M}_{1}(u)$ is a compact 1-manifold whose boundary is $\mathscr{M}_{0}(\bd u)$. Here it is crucial that $u$ is in general position to $V$.

The orientation of the simplex $u$ determines, at each point $(x,q)\in \mathscr{M}_{0}(u)$, an orientation $\mathfrak{e}_{(x,q)}$ for the unstable manifold of $q$, and a homotopy class of paths joining $\mathscr{L}_{u(x)}$ to $\mathscr{L}_{q}$. By preconcatenating with the homotopy class of paths from $\mathrm{pt}$ to $u(x)$, determined by the lifts to the universal cover, we obtain a map $\mu_{(x,q)}:\mathscr{L}_{\mathrm{pt}}\to \mathscr{L}_{q}$.

For $y\in \mathscr{L}_{\mathrm{pt}}$, define:
\begin{equation*}
  \textstyle  C(\mathscr{P}^{\ell}_{K},N;V;\mathscr{L})\ni u\otimes y \mapsto \mathrm{F}(u\otimes y)=\sum_{\mathscr{M}_{0}}\mathfrak{e}_{(x,q)}\otimes \mu_{(x,q)}(y)\in \mathrm{CM}_{\ell}.
\end{equation*}
Note that, if $g\in \pi_{1}(\mathscr{P}^{\ell}_{K},\mathrm{pt})$, then:
\begin{equation*}
  \mathrm{F}(ug\otimes y)=\mathrm{F}(u\otimes \mathscr{L}_{g}y),
\end{equation*}
where $ug$ is the same exact simplex, except we replace the lifts to the universal cover by their pre-concatenations with $g$. Moreover, it is clear that $F(u\otimes y)=0$ whenever $u$ lies in $N$. Therefore $\mathrm{F}$ is well-defined. Moreover:

\begin{lemma}
  The map $\mathrm{F}$ is a chain map and a quasi-isomorphism.
\end{lemma}
\begin{proof}
  That $\mathrm{F}$ is a chain map follows by considering $\mathscr{M}_{1}(u)$ in relation to $\mathscr{M}_{0}(\bd u)$; the details are left to the reader.

  Next we prove $\mathrm{F}$ is injective on homology. Let $\Sigma^{m}$ be a cycle and suppose $\mathrm{F}(\Sigma)$ vanishes in homology. We will modify $\Sigma$ by standard $1$-handle attaching cobordisms, where the attaching $0$-spheres are contained in the interiors of the $m$-dimensional faces. The resulting object $\Sigma'$ will not exactly be a singular cycle anymore, however, one subdivides the result to make it a singular cycle.

  The relevance of this discussion is the following: by attaching $1$-handles, we claim that one can replace $\Sigma$ by a cycle $\Sigma'$ so that there are no flow lines starting on $\Sigma'$ and ending at an index $m$ critical point. Indeed, since $\mathrm{F}(\Sigma)$ vanishes in homology, we have an equation of the form:
  \begin{equation*}
    \mathrm{F}(\Sigma)+d\sum b_{i}=0\implies \mathrm{F}(\Sigma\cup S_{1}\cup \dots \cup S_{k})=0,
  \end{equation*}
  where $S_{i}$ is the unstable sphere around the index $m+1$ critical point $b_{i}$ (decorated with some local coefficient). We replace $\Sigma$ by the homologous sum $\Sigma\cup S_{1}\cup \dots \cup S_{k}$.

  Since $\mathrm{F}(\Sigma)$ vanishes on chain level, the elements in the moduli space of trajectories cancel in pairs. If two trajectories $(x_{0},q)$ and $(x_{1},q)$ cancel then can first flow neighborhoods of $x_{0}$ and $x_{1}$ forward in time so they become close to $q$, then do a standard $1$-handle attachment. The resulting cycle is now ``below'' $q$; see Figure \ref{fig:cross-1-handle}.

  \begin{figure}[H]
    \centering
    \begin{tikzpicture}
      \begin{scope}
        \clip (-2,-1.35)rectangle(2,1.35);
        \draw[decorate,decoration={
          markings,
          mark=between positions 0.045 and 1 step 0.0625 with {\arrow{>};},
        }] (0,0)--(2,0) (0,0)--(-2,0) (0,2)--(0,0) (0,-2)--(0,0);
        \draw[name path={P1}] (0,-2)--(0,2) (2,0)--(-2,0);
        \draw[name path={P2},line width=1pt] (2,1)to[out=200,in=-20](-2,1) (2,-1)to[out=150,in=30](-2,-1);
        \path[name intersections={of=P1 and P2}];
        \path (intersection-1) node[circle,fill,inner sep=1pt]{}node[above right]{$x_{1}$} --(intersection-2)node[circle,fill,inner sep=1pt]{}node[below right]{$x_{0}$};
      \end{scope}

      \draw[dashed,line width=1pt] (2,1)to[out=200,in=150,looseness=2](2,-1) (-2,1)to[out=-20,in=30,looseness=2](-2,-1);
      \node[circle,fill,inner sep=1pt] (q)at (0,0){};
      \node at (q) [above right] {$q$};
    \end{tikzpicture}
    \caption{Attaching a 1-handle to cross the critical point $q$. The solid line is pre-surgery, the dashed line is post-surgery.}
    \label{fig:cross-1-handle}
  \end{figure}
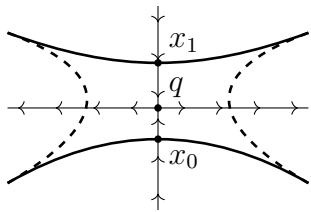

  After this surgery operation (and subdivision), the result $\Sigma'$ represents a homologous element in the chain complex. Importantly, $\Sigma'$ can be flowed into a neighborhood of the unstable manifolds of critical points of index $\le m-1$. Standard techniques in Morse theory, e.g., \cite{milnor_morse_theory}, imply that such a neighborhood has the homotopy type of a CW complex relative $N$ with only cells of index $\le m-1$. In particular, $\Sigma'$, and hence $\Sigma$, vanishes in the relative singular homology group. Thus $\mathrm{F}$ is injective.

  To see that $\mathrm{F}$ is surjective, one can argue similarly. Clearly for any Morse homology class $\mathfrak{q}=\sum q_{i}\otimes \ell_{i}$, the sum of unstable disks $\Sigma=D_{1}\cup \dots \cup D_{k}$, with appropriate cappings and local coefficients, satisfies $\mathrm{F}(\Sigma)=\mathfrak{q}$. This class $\Sigma$ is not a cycle, however, we can correct it to a cycle, as follows. Since $\mathrm{F}$ is a chain map and $\mathfrak{q}$ is closed, $\mathrm{F}(\bd D_{1}\cup \dots \cup \bd D_{k})$ vanishes on chain level. Thus, by the argument used for injectivity, we can attach 1-handles to the unstable spheres $S_{i}=\bd D_{i}$ so that the resulting cycle can be flowed into a neighborhood of the index $\le m-2$ unstable manifolds. Attaching $1$-handles to the unstable spheres is equivalent to doing a boundary connect sum to the disks; once again we appeal to some subdivision algorithm for the result of a boundary connect sum.

  Thus we can replace $\Sigma$ by a chain $\Sigma'$ whose boundary lies in a neighborhood $U$ of the index $\le m-2$ unstable manifolds (union $N$), without changing $\mathrm{F}(\Sigma')=\mathfrak{q}$. Consider $\Sigma'$ as an element of $H_{m}(\mathscr{P}^{\ell}_{K},U;\mathscr{L})$.

  The natural map $H_{m}(\mathscr{P}^{\ell}_{K},N;\mathscr{L})\to H_{m}(\mathscr{P}^{\ell}_{K},U;\mathscr{L})$ is surjective by long exact sequence for a triple since $H_{m-1}(U,N;\mathscr{L})=0$. Thus we can further replace $\Sigma'$ with $\Sigma''$ so that $\mathrm{F}(\Sigma'')=\mathfrak{q}$ and $\Sigma''$ has its boundary entirely in $N$. Thus $\mathrm{F}$ is surjective, as desired.
\end{proof}

\subsubsection{Compatibility with continuation lines}
\label{sec:comp-with-cont-1}

Let $(\mathscr{P}^{\ell}_{K},V)$ and $(\mathscr{P}^{\ell^{\prime}}_{K'},V')$ be choices of admissible data, with $K'=dK$ and $\ell'\ge \ell$, and let $\delta(s)$ be a choice of perturbation defining a continuation map $\mathfrak{c}$, as in \S\ref{sec:continuation-lines}. Abbreviate $\mathscr{P}_{0}=\mathscr{P}^{\ell}_{K}$ and $\mathscr{P}_{1}=\mathscr{P}^{\ell'}_{K'}$.

Consider the parametric space $\Pi(\delta)$ of tuples $(s,x,q)$ where the flow line of the piecewise smooth vector field in Figure \ref{fig:continuation-lines} starting at $x\in\mathscr{P}_{0}$ at time $s$ converges to a critical point $q$.

Note that for $s_{0}\ge 1$, the slice $\Pi(\delta)\cap \set{s=s_{0}}$ is simply the intersection of $\mathscr{P}_{0}$ with the unstable manifolds of $V'$. By picking $V'$ generically, we can ensure that this slice is cut out transversally.

By picking $\delta$ generically, we require that $\Pi(\delta)$ is cut out transversally, as a parametric moduli space. Then we let $C_{0}$ be the subcomplex of $C(\mathscr{P}_{0},N;\mathscr{L})$ consisting of all simplices which are in general position to the evaluation $(s,x,q)\in \Pi(\delta)\mapsto x\in \mathscr{P}_{0}$, and also in general position to $V$, as in \S\ref{sec:comp-with-sing}. Similarly let $C_{1}$ be the subcomplex of $C(\mathscr{P}_{1},N;\mathscr{L})$ consisting of all simplices in general position to $V'$.

The following square is commutative up to chain homotopy:
\begin{equation}\label{eq:commutative_continuation}
  \begin{tikzcd}
    {C_{0}}\arrow[d,"{i}"]\arrow[r,"{\mathrm{F}}"] &{\mathrm{CM}_{\ell}(V)}\arrow[d,"{\mathfrak{c}}"]\\
    {C_{1}}\arrow[r,"{\mathrm{F}'}"] &{\mathrm{CM}_{\ell'}(V')}.
  \end{tikzcd}
\end{equation}
Indeed, for a given simplex $u$, consider the parametric space $\mathscr{M}_{0}$ of tuples $(s,u(\sigma),q)$ where the Morse index of $q$ equals the dimension of $u$ plus $1$. This is a finite set, and counting the elements appropriately determines a map $\kappa:C\to \mathrm{CM}_{\ell}(V')$. Then:
\begin{equation*}
  \mathrm{F}'\circ i-\mathfrak{c}\circ \mathrm{F}=d \kappa+\kappa \bd;
\end{equation*}
To see why, one considers the $1$-dimensional parametric space $\mathscr{M}_{1}$ of tuples $(s,u(\sigma),q)$ where the Morse index of $q$ equals the dimension of $u$, and observes that the counts of $\mathscr{M}_{1}\cap \set{s=1}$ and $\mathscr{M}_{1}\cap \set{s=s_{0}}$, for $s_{0}\ll 0$, give $\mathrm{F}'\circ i$ and $\mathfrak{c}\circ F$, respectively. The counts of the two slices differ by the (i) boundary points $\mathscr{M}_{1}\cap (\R\times \bd \Delta\times \mathrm{Crit}(\mathscr{E}))$, corresponding to $\kappa\bd(u)$, and (ii) the failures of properness due to breaking off of flow lines, corresponding to $d\kappa(u)$. The desired result follows.

\subsubsection{Induced map on the colimit}
\label{sec:induced-map-colimit}

The results in \S\ref{sec:comp-with-sing} produce isomorphisms:
\begin{equation*}
  H_{*}(\mathscr{P}^{\ell}_{K},N;\mathscr{L})\to \mathrm{H}(C(\mathscr{P}^{\ell}_{K},N;V;\mathscr{L}))\to \mathrm{HM}_{\ell}(V)\to \mathrm{HM}_{\ell}(\mathscr{P},N;\mathscr{E};\mathscr{L}),
\end{equation*}
the first map is the homotopy inverse of the inclusion, and the final map is the isomorphism to the colimit. Because of \S\ref{sec:comp-with-cont-1}, this map does not depend on the choice of $V$.

In \cite[\S16]{milnor_morse_theory} it is proved that the inclusion $\mathscr{P}^{\ell}_{K}\to \set{\mathscr{E}\le \ell^{2}}$ is a homotopy equivalence, provided $\mathscr{P}^{\ell}_{K}$ is admissible. In particular, we obtain isomorphisms:
\begin{equation*}
  H_{*}(\set{\mathscr{E}\le \textstyle \ell^{2}},N;\mathscr{L})\to \mathrm{HM}_{\ell}(\mathscr{P},N;\mathscr{E};\mathscr{L}).
\end{equation*}
These isomorphisms are compatible with increasing $\ell$. Since the singular homology of a monotone union is colimit of the singular homologies, we obtain the desired isomorphism between the colimits $H_{*}(\mathscr{P},N;\mathscr{L})\to \mathrm{HM}(\mathscr{P},N;\mathscr{E};\mathscr{L})$.

\subsection{On the choice of local system and the twisted Hurewicz theorem}
\label{sec:choice-local-system}

In this section we describe the choice of local system which guarantees that $H_{*}(\mathscr{P},N;\mathscr{L})$ is non-zero. Throughout, fix a base-point $\mathrm{pt}\in N\subset \mathscr{P}$.

\subsubsection{Free local systems}
\label{sec:free-loc-system}

A local system $\mathscr{L}$ of coefficients on a path-connected and locally-simply-connected space $\mathscr{P}$ is called \emph{free} provided $\mathscr{L}_{\mathrm{pt}}$ is isomorphic to a direct sum $\bigoplus_{S}\Z[\pi_{1}(\mathscr{P},\mathrm{pt})]$ (as a left $\Z[\pi_{1}(\mathscr{P},\mathrm{pt})]$-module) over some indexing set $S$.

\subsubsection{Relative Hurewicz theorem for Morse homology with local coefficients}
\label{sec:relat-hurewicz-theorem}

The main result in this section is the following version of the relative Hurewicz theorem with local coefficients; see \cite[Proposition 2.1]{rong}.

\begin{prop}\label{prop:hurewicz_hard}
  For a free local system $\mathscr{L}$, there is a group isomorphism:
  \begin{equation}\label{eq:fancy_hhh}
    \pi_{k}(\mathscr{P},N,\mathrm{pt})\otimes_{\Z[\pi_{1}(N,\mathrm{pt})]} \mathscr{L}_{\mathrm{pt}}\to H_{k}(\mathscr{P},N;\mathscr{L}),
  \end{equation}
  if $k>1$ is the first integer for which $\pi_{k}(\mathscr{P},N,\mathrm{pt})\ne 0$ but $\pi_{j}(\mathscr{P},N,\mathrm{pt})=0$ for $j<k$.
\end{prop}
\begin{prop}\label{prop:hurewicz_not_a_group}
  If $\pi_{1}(\mathscr{P},N, \mathrm{pt})$ is non-trivial, and $\mathscr{L}$ is a free local system, then the homology group $H_{1}(\mathscr{P},N;\mathscr{L})$ is non-zero.
\end{prop}
Some comments are in order: (i) $\pi_{1}(\mathscr{P},N,\mathrm{pt})$ is not a group, so the restriction $k>1$ in Proposition \ref{prop:hurewicz_hard} is reasonable; (ii) the tensor product is over $\Z[\pi_{1}(N,\mathrm{pt})]$, acting by monodromy on $\mathscr{L}_{\mathrm{pt}}$ and in the usual sense on $\pi_{k}(\mathscr{P},N,\mathrm{pt})$; see \cite{bredon_GT,hatcher} for the definition of the action; (iii) for $k>1$, the group $\pi_{k}(\mathscr{P},N,\mathrm{pt})$ is abelian and surjected upon by $\pi_{k}(\mathscr{P},\mathrm{pt})$;\footnote{The inclusion $N\to \mathscr{P}$ has a right inverse, therefore $\pi_{k-1}(N)\to \pi_{k-1}(\mathscr{P})$ is injective.} (iv) the map \eqref{eq:fancy_hhh} is natural, and it is defined by the diagram \eqref{eq:square-prop-hard}.

\subsubsection{Singular chains in the covering space}
\label{sec:singular-chains-cover}

Let $\mathscr{L}$ be a free local system on the set of paths $\mathscr{P}$. Consider the universal covering space $\til{\mathscr{P}} \to \mathscr{P}$, let $N'$ denote the preimage of $N$ in $\til{\mathscr{P}} $, and let $\mathrm{pt}'\in N'$ be a chosen lift of the base-point $\mathrm{pt}$.

Let $C(\til{\mathscr{P}} ,N')$ be the relative singular chain complex, which has a right-action of $\pi_{1}(\mathscr{P},\mathrm{pt})$ as explained in \S\ref{sec:comp-with-sing}. By definition, the singular chain complex with local coefficients is given as a diagonal quotient:
\begin{equation*}
  C(\mathscr{P},N;\mathscr{L}):=C(\til{\mathscr{P}} ,N')\otimes_{\Z[\pi_{1}(\mathscr{P},\mathrm{pt})]}\mathscr{L}_{\mathrm{pt}}.
\end{equation*}
The right hand side is given the tensor product differential $\bd \otimes 1$.

Observe that there is a natural left $\Z[\pi_{1}(\mathscr{P},\mathrm{pt})]$-module structure on $C(\til{\mathscr{P}} ,N')\otimes_{\Z} \mathscr{L}_{\mathrm{pt}}$ so that $g$ acts by $\sigma\otimes \ell\mapsto \sigma g^{-1}\otimes g \ell$. Then $C(\mathscr{P},N;\mathscr{L})$ is the module of \emph{coinvariants} with respect to this action; see \cite[\S IX.7.6]{aluffi} for further discussion.

Since $\mathscr{L}$ is a free local system the natural map:
\begin{equation}\label{eq:nat-map-is-iso}
  H_{k}(\til{\mathscr{P}} ,N')\otimes_{\Z[\pi_{1}(\mathscr{P},\mathrm{pt})]} \mathscr{L}_{\mathrm{pt}}\to H_{k}(\mathscr{P},N;\mathscr{L})
\end{equation}
is an isomorphism. Indeed, any element $\mathfrak{e}$ in the left hand side can be written as $\mathfrak{e}=\sum [\sigma_{i}]\otimes \ell_{i}$ where $\ell_{i}$ are basis elements for $\mathscr{L}_{\mathrm{pt}}$, and $[\sigma_{i}]$ are homology classes. The natural map \eqref{eq:nat-map-is-iso} sends $\mathfrak{e}$ to $[\sum\sigma_{i}\otimes \ell_{i}]$. If this output is zero, then we can write
\begin{equation*}
  \sum \sigma_{i}\otimes \ell_{i}=\sum \d \alpha_{i}\otimes \ell_{i}.
\end{equation*}
Using that $\ell_{i}$ are elements of a basis of $\mathscr{L}_{\mathrm{pt}}$ as a left $\Z[\pi_{1}(\mathscr{P},\mathrm{pt})]$-module, the representation in the above form is unique, and hence $\sigma_{i}=\d \alpha_{i}$. This proves \eqref{eq:nat-map-is-iso} is injective. A similar argument proves that \eqref{eq:nat-map-is-iso} is surjective.

If $\pi_{1}(\mathscr{P},N,\mathrm{pt})=0$ then $\pi_{1}(N,\mathrm{pt})\to \pi_{1}(\mathscr{P},\mathrm{pt})$ is an isomorphism, and so the natural map:
\begin{equation*}
  H_{k}(\til{\mathscr{P}} ,N')\otimes_{\Z[\pi_{1}(N,\mathrm{pt})]} \mathscr{L}_{\mathrm{pt}}\to H_{k}(\mathscr{P},N;\mathscr{L})
\end{equation*}
is also an isomorphism; we will use this fact in \S\ref{sec:proof-of-hard}.

\subsubsection{Proof of Proposition \ref{prop:hurewicz_hard}}
\label{sec:proof-of-hard}

Let $\til{\mathscr{P}} ,N'$ be the universal covers, as in \S\ref{sec:singular-chains-cover}. It is straightforward to see that there is a commutative square:
\begin{equation}\label{eq:square-prop-hard}
  \begin{tikzcd}
    {\pi_{k}(\mathscr{P},N,\mathrm{pt})\otimes_{\Z} \mathscr{L}_{\mathrm{pt}}}\arrow[d,"{}"]\arrow[r,"{}"] &{H_{k}(\til{\mathscr{P}} ,N')\otimes_{\Z}\mathscr{L}_{\mathrm{pt}}}\arrow[d,"{}"]\\
    {\pi_{k}(\mathscr{P},N,\mathrm{pt})\otimes_{\Z[\pi_{1}(N,\mathrm{pt})]} \mathscr{L}_{\mathrm{pt}}}\arrow[r,"{}"] &{H_{k}(\mathscr{P},N;\mathscr{L})},
  \end{tikzcd}
\end{equation}
where we identify $\pi_{k}(\til{\mathscr{P}} ,N',\mathrm{pt}')\simeq \pi_{k}(\mathscr{P},N,\mathrm{pt})$ via the projection map; this is an isomorphism because $k\ge 2$ and the Serre fibration property.

The vertical morphisms in \eqref{eq:square-prop-hard} are the coinvariant quotient maps associated to the left $\Z[\pi_{1}(N,\mathrm{pt})]$-module structures. Indeed, the top morphism is equivariant with respect to the $\pi_{1}(N,\mathrm{pt})$ action (to see this, one needs to recall the $\pi_{1}(N,\mathrm{pt})$ action on $\pi_{k}(\mathscr{P},N,\mathrm{pt})$ and compare it with the action on $H_{k}(\til{\mathscr{P}} ,N')$ by deck transformations).

The bottom map of \eqref{eq:square-prop-hard} is obtained from the top map by applying the coinvariants functor, and it therefore suffices to prove that:
\begin{equation}\label{eq:want_to_show_injective}
  \pi_{k}(\til{\mathscr{P}} ,N',\mathrm{pt}')\to H_{k}(\til{\mathscr{P}} ,N')
\end{equation}
is an isomorphism, since functors preserve isomorphisms. That \eqref{eq:want_to_show_injective} is an isomorphism follows from the relative Hurewicz theorem; see \cite[\S VII]{bredon_GT}, \cite[Theorem 4.37]{hatcher}. This completes the proof of Proposition \ref{prop:hurewicz_hard}. We give a Morse theoretic proof that \eqref{eq:want_to_show_injective} is an isomorphism in \S\ref{sec:relat-hurew-morse}.

\subsubsection{Proof of Proposition \ref{prop:hurewicz_not_a_group}}
\label{sec:proof-prop-not-a-group}

In the case when $\pi_{1}(\mathscr{P},N)$ is non-zero, the above arguments break-down, mainly because of the failure of commutativity. However, one can argue directly to conclude $H_{1}(\mathscr{P},N;\mathscr{L})$ is non-zero.

Indeed, if $\mathscr{L}$ is a free local system then the arguments in \S\ref{sec:singular-chains-cover} imply that the natural map:
\begin{equation}\label{eq:tricky-iso}
  H_{1}(\til{\mathscr{P}} ,N')\otimes_{\Z[\pi_{1}(\mathscr{P},\mathrm{pt})]} \mathscr{L}_{\mathrm{pt}}\to H_{1}(\mathscr{P},N;\mathscr{L})
\end{equation}
is an isomorphism. Let $\gamma$ be a non-trivial element of $\pi_{1}(\mathscr{P},N,\mathrm{pt})$ which lifts to a path in $\til{\mathscr{P}} $ starting at $\mathrm{pt}'$ and ending in $N'$. Since $\gamma$ is non-trivial and $\til{\mathscr{P}} $ is simply connected, this path must join different components of $N'$.

Therefore the path $\gamma$ induces some non-zero element in the group $H_{1}(\til{\mathscr{P}} ,N';\mathscr{L})$, because of the connecting homomorphism. Thus $\gamma\otimes \ell_{0}$ is non-zero in the left hand side of \eqref{eq:tricky-iso}, since $\mathscr{L}$ is free, and hence $\gamma\otimes \ell_{0}$ is sent to a non-zero element of the right hand side. This completes the proof of Proposition \ref{prop:hurewicz_not_a_group}.

\subsection{Wrapped Floer cohomology for conormal Lagrangians}
\label{sec:wrapped-floer-cohom}

Fix a local system of coefficients $\mathscr{L}$ on the space of paths $\mathscr{P}$. Setting $\mathscr{L}=\pr^{*}\mathscr{L}$, one also considers $\mathscr{L}$ as a local system on the space of paths joining $\nu^{*}N$ to $\nu^{*}N$.

To each admissible Hamiltonian system and complex structure $H_{t},J_{t}$, associate the \emph{Floer cohomology} complex: $$\mathrm{CF}(\nu^{*}N; H_{t}, J_{t};\mathscr{L})  = \bigoplus \mathfrak{o}_{\gamma}\otimes \mathscr{L}_{\gamma},$$  where the direct sum is over all chords $\gamma$ of $H_{t}$ with endpoints on $\nu^{*}N$. The differential on the complex is the usual (cohomological) Floer differential twisted by the monodromy of the local system  (\S\ref{sec:floer-differential-local-coefficients}). Admissible data is defined in \S\ref{sec:admissible-data}.

If $\nu^{*}N,\mathscr{L}$ can be inferred from the context, we will use the abbreviation $\mathrm{CF}(H_{t}, J_{t})$.

Given a \emph{non-negative-at-infinity} path joining one system $H_{0,t}$ to another $H_{1,t}$ (and a path joining $J_{0,t}$ to $J_{1,t}$) there is an associated continuation map: $$\mathfrak{c}:\mathrm{CF}(\nu^{*}N; H_{0,t}, J_{0,t}; \mathscr{L})\to\mathrm{CF}(\nu^{*}N; H_{1,t}, J_{1,t}; \mathscr{L});$$ see \S\ref{sec:continuation-maps-hf} for more details. Technically, one requires the path to be generic on the compact part of $T^{*}M$ so that moduli spaces of continuation strips are cut out transversally; such paths will be called \emph{admissible}.

Consider the diagram (i.e., small category) whose objects are admissible data $(H_{t},J_{t})$ and whose morphisms are homotopy classes of admissible paths. Often this category of Floer data is referred to as a \emph{directed system}. Standard arguments show that $(\mathrm{HF},\mathfrak{c})$ is a functor from this category to the category of $\Z$-modules. The following two invariants are extracted directly from this functor:
\begin{enumerate}
\item The \emph{(full)-wrapped Floer cohomology} $\mathrm{HW}(\nu^{*}N;\mathscr{L})$ is defined as the colimit of $(\mathrm{HF},\mathfrak{c})$. One can compute this colimit by taking any cofinal sequence of contact-at-infinity Hamiltonian systems $H_{n}$, where $H_{n}$ gets more and more positive at infinity, and taking the direct limit of $\mathrm{HF}$ with respect to continuation maps.
\item The \emph{zero-wrapped Floer cohomology}, $\mathrm{HW}_{0}$, defined as the limit of $(\mathrm{HF},\mathfrak{c})$ over the subcategory of Hamiltonian systems which are positive at infinity. One computes the limit by taking a sequence $H_{n}$ which are positive but tending to zero at infinity, and taking the inverse limit of $\mathrm{HF}$ with respect to continuation maps.
\end{enumerate}

The final invariant is the \emph{positive-wrapped Floer cohomology} $\mathrm{HW}_{+}$ and is extracted from a chain-level construction. If one is working over a field, one can simply define $\mathrm{HW}_{+}$ as the mapping cone of the map $\mathrm{HW}_{0}\to \mathrm{HW}$, where $\mathrm{HW}_{0},\mathrm{HW}$ are considered as complexes with zero differential. However, when one is working over the integers, one needs to be more careful, because of extension phenomena; see, e.g., \cite[\S3]{GPS1} which constructs the wrapped Fukaya category over the integers.

The definition we give below has the following crucial property: if $\mathrm{HW}_{+}(\nu^{*}N;\mathscr{L})$ is non-zero, for some $\mathscr{L}$, then $\nu^{*}N$ has a Reeb chord for every contact form; see \S\ref{sec:reeb-chords-pos-wrap-floer}. Before we define $\mathrm{HW}_{+}$ we first recall one notion from homological algebra:

\subsubsection{Cofibrations}
\label{sec:cofibrations}
A chain map $C_{1}\to C_{2}$ is a \emph{cofibration} provided:
\begin{enumerate}
\item there is a decomposition $C_{2}\simeq C_{1}\oplus F$ where $F$ is a free $\Z$-module, and
\item the chain map $C_{1}\to C_{2}$ is identified with the inclusion of the first summand.
\end{enumerate}
Note that the differential of $C_{2}$ may include off-diagonal terms taking $F$ to $C_{1}$.

\subsubsection{Definition of positive wrapped Floer cohomology}
\label{sec:defin-posit-wrapp}
Say that a sequence of admissible data $(H_{n},J_{n})$, with admissible continuation data joining the choices from $n$ to $n+1$, is \emph{cofibrant-cofinal} provided:

\begin{enumerate}
\item it is cofinal, i.e., $H_{n}$ converges to $+\infty$ on the symplectization end of $T^{*}M$,
\item the continuation maps $\mathrm{CF}(H_{n},J_{n})\to \mathrm{CF}(H_{n+1},J_{n+1})$ are cofibrations, and,
\item $H_{0}$ is positive at infinity and the natural map $\mathrm{HW}_{0}\to \mathrm{HF}(H_{0})$ is an isomorphism.
\end{enumerate}

For such a cofibrant-cofinal sequence $\set{H_{n},J_{n}}$, define $\mathrm{HW}_{+}(\set{H_{n},J_{n}})$ as the colimit of the homologies of the quotients $\mathrm{CF}(H_{n},J_{n})/\mathrm{CF}(H_{0},J_{0})$ as $n\to\infty$. It follows from the definitions that, for any choice of cofibrant-cofinal sequence, there is a long-exact sequence:
\begin{equation}\label{eq:hwplus-long-exact}
  \dots \to \mathrm{HW}_{0}\to \mathrm{HW}\to \mathrm{HW}_{+}(\set{H_{n},J_{n}})\to \mathrm{HW}_{0}[1]\to \dots.
\end{equation}
The positive-wrapped homology $\mathrm{HW}_{+}$ is invariant in the following sense:
\begin{lemma}\label{lemma:pos-wrap-invariance}
  For any two cofibrant sequences $\set{H_{n}^{i},J_{n}^{i}}$, $i=0,1$, there is an isomorphism $\mathrm{HW}_{+}(\set{H_{n}^{0},J_{n}^{0}})\to \mathrm{HW}_{+}(\set{H_{n}^{1},J_{n}^{1}})$ respecting the long exact sequence \eqref{eq:hwplus-long-exact}.
\end{lemma}

Lemma \ref{lemma:pos-wrap-invariance} is proved in \S\ref{sec:invariance-pos-wrap}. In \S\ref{sec:spec-cofibr-sequ} we construct cofibrant-cofinal sequences, so the definition of $\mathrm{HW}_{+}$ is not vacuous.

\subsubsection{Admissible data}
\label{sec:admissible-data}
A Hamiltonian $H_{t}$ is \emph{admissible} if:
\begin{enumerate}
\item $H_{t}(q,e^{s}p)=e^{s}H_{t}(q,p)$ outside a compact set,
\item $X_{H_{t}}$ has no chords for $\nu^{*}N$ outside a compact set,
\item the chords of $X_{H_{t}}$ for $\nu^{*}N$ are non-degenerate,
\end{enumerate}
Note that (i) and (iii) imply (ii). See \cite{ritter_tqft,ulja_automorphisms,merry_ulja,ulja_drobnjak,brocic_cant} for a similar class of Hamiltonians.

A complex structure $J$ is \emph{admissible} if:
\begin{enumerate}[label=(\alph*)]
\item $J$ is tame for the symplectic structure,
\item the positive Liouville flow preserves $J$, outside a compact set,
\end{enumerate}

\subsubsection{Grading by Conley-Zehnder index}
\label{sec:grad-conl-zehnd}

For a Hamiltonian chord $\gamma$, define the grading of $\gamma$ to be $\mu(\gamma):= d - n - \mathrm{CZ}(\gamma)$, where $\mathrm{CZ}(\gamma)$ is defined in \S\ref{sec:asymptotic-data}. The grading is chosen so that the Floer differential in \S\ref{sec:floer-differential-local-coefficients} raises the degree by $1$. The shift by $d-n$ is chosen so that the isomorphism between Morse homology and Floer cohomology constructed in \S\ref{sec:isom-from-morse} sends degree $k$ to degree $-k$.

\subsubsection{The Floer differential with local coefficients}
\label{sec:floer-differential-local-coefficients}

Given two chords $\gamma_{\pm}$ of an admissible Hamiltonian $H_t$, with endpoints on $\nu^*N$, we define $\mathscr{M}$ as the set of solutions of:
\begin{equation*}
  \left\{
    \begin{aligned}
      &u:\R \times [0,1]\to T^{*}M,\\
      &\bd_{s}u+J_{t}(u)(\bd_{t}u-X_{H_{t}}(u))=0,\\
      &u(s,0), u(s,1)\in \nu^*N,
    \end{aligned}
  \right.
\end{equation*}
and let $\mathscr{M}(\gamma_{-},\gamma_{+})$ be the subset with asymptotic conditions $\lim_{s\to \pm \infty} u(s, \cdot)= \gamma_{\pm}$. For a generic choice of a Hamiltonian, $\mathscr{M}(\gamma_-, \gamma_+)$ is a manifold of dimension $\mu(\gamma_{-})-\mu(\gamma_{+})$, where $\mu$ was defined in \S\ref{sec:grad-conl-zehnd} (see also \S\ref{sec:conl-zehnd-indic}, \S\ref{sec:asymptotic-data}).

The moduli space $\mathscr{M}$ possesses a canonical $\R$-action, by translation in the $s$-direction. Differentiating in the $s$-direction produces a preferred element $\eta_{u}\in \ker D_{u}$ for every $u\in \mathscr{M}$. If the Fredholm index of $D_{u}$ is $1$, and $u$ is regular, then this element $\eta_{u}$ is a basis for the kernel, and hence determines a generator of $\mathfrak{o}(D_{u})$, the orientation line of the (Fredholm) determinant of $D_{u}$.

It follows from \S\ref{sec:orientation-lines} and \S\ref{sec:asymptotic-data} that for every $u \in \mathscr{M}(\gamma_-, \gamma_+)$ there is a canonical isomorphism:
\begin{equation}\label{eq:orientation-gluing}
  \mathfrak{o}_{\gamma_{+}} \simeq \mathfrak{o}_{\gamma_{-}} \otimes \mathfrak{o}(D_{u}).
\end{equation}
When $\mu(\gamma_{-})=\mu(\gamma_{+})+1$, we use the canonical orientation of $\mathfrak{o}(D_{u})$ given by the basis $-\eta_{u}$ and the identification \eqref{eq:orientation-gluing} to obtain an isomorphism $d_u: \mathfrak{o}_{\gamma_{+}} \to \mathfrak{o}_{\gamma_{-}};$ essentially, one can use \eqref{eq:orientation-gluing}, plus the choice of $-\eta_{u}$, to transfer the orientation from $\mathfrak{o}_{\gamma_{+}}$ to $\mathfrak{o}_{\gamma_{-}}$. The choice of the minus sign can be thought of as ``cohomological conventions,'' i.e., cylinders go from right to left, however, the real significance of the sign is needed in the comparison with Morse theory in \S\ref{sec:isom-from-morse}; see also \S\ref{sec:inter-break-param}.

On the other hand, given a local system $\mathscr{L}$, every $u \in \mathscr{M}(\gamma_-, \gamma_+)$ has an associated monodromy $\mathscr{L}(u): \mathscr{L}_{\gamma_{+}} \to \mathscr{L}_{\gamma_{-}}$. Define the \emph{Floer differential} to be the direct sum of the individual contributions $d_{u}\otimes \mathscr{L}(u)$ as $u$ varies over all Floer strips of index difference $1$. As usual, we have:
\begin{lemma}\label{lemma:differential}
  The map $d:\mathrm{CF}(H_t; \mathscr{L}) \to \mathrm{CF}(H_t; \mathscr{L})$ squares to $0$.
\end{lemma}
\begin{proof}
  It is enough to show that for every two chords $\gamma_{-}, \gamma_{+}$ with $\mu(\gamma_{-}) = \mu(\gamma_{+})+2$, and for every component $C$ of  $\mathscr{M}(\gamma_-, \gamma_+)/ \R$ we have the relation
  \begin{equation}\label{eq:relation-u-v-differential}
    d_{u_1} \circ d_{v_1} + d_{u_2} \circ d_{v_2}=0,
  \end{equation}
  where $(u_1, v_1)$ and $(u_2, v_2)$ are the broken strips which appear as the boundary of $C$. Indeed, since $C$ induces a homotopy between the concatenations $u_1\# v_1$ and $u_2\# v_2$, hence $\mathscr{L}(u_1) \circ \mathscr{L}(v_1) = \mathscr{L}(u_2) \circ \mathscr{L}(v_2)$, therefore the proof in the case of local systems is identical to the one with $\mathscr{L} = \Z$.

  \begin{figure}[H]
    \centering
    \begin{tikzpicture}
      \draw[every node/.style={draw,circle,inner sep=1pt,fill},postaction={decorate,decoration={markings,mark=between positions 0.2 and 1.0 step 0.24 with {\arrow{Latex};}}}] (0,0)node(A){} to[out=0,in=180] (1,1)node(C){} to[out=0,in=180] (2,-1) to[out=0,in=180] (3,0)node(B){};
      \node at (A)[left]{$(u_{1},v_{1})$};
      \node at (C)[above]{$u$};
      \node at (B)[right]{$(u_{2},v_{2})$};
    \end{tikzpicture}
    \caption{The one-dimensional component $C$ of the moduli space of index difference two chords.}
    \label{fig:one-dimensional-moduli}
  \end{figure}
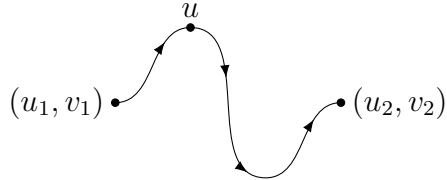

  The orientation line $\mathfrak{o}(D_{u})$ is identified with the orientation line for the two-dimensional moduli space $\mathscr{M}(\gamma_{-},\gamma_{+})$. The component of $\mathscr{M}$ lying over $C$ is diffeomorphic to $C\times \R$ where the translation action is identified with translation along the $\R$ factor.

  Near the end of $C$ corresponding to the breaking $(u_{i},v_{i})$, the gluing construction produces a canonical identification:
  \begin{equation}\label{eq:k-g-map}
    \mathfrak{o}(\ker D_{u_{i}}\oplus \ker D_{v_{i}})\simeq \mathfrak{o}(D_{u_{i}})\otimes \mathfrak{o}(D_{v_{i}})\to \mathfrak{o}(D_{u}).
  \end{equation}
  As explained above, $\ker D_{u_{i}}$ and $\ker D_{v_{i}}$ are generated by the elements $\eta_{u_{i}}$ and $\eta_{v_{i}}$, respectively. It follows from \cite[\S2e]{floer-ham} that $(\eta_{u_1}, \eta_{v_1})$ and $(\eta_{u_2}, \eta_{v_2})$ give rise to the opposite generators of $\mathfrak{o}(D_{u})$ via the kernel gluing map \eqref{eq:k-g-map}.

  Linear gluing also gives an identification:
  \begin{equation}\label{eq:linear-g-o-l}
    \mathfrak{o}_{\gamma_{-}} \otimes \mathfrak{o}(D_{u_{i}}) \otimes \mathfrak{o}(D_{v_{i}}) \simeq \mathfrak{o}_{\gamma_{+}},
  \end{equation}
  and the contribution $\mathfrak{o}_{\gamma_{+}}\to \mathfrak{o}_{\gamma_{-}}$ to the Floer differential is defined by inserting the generator $\eta_{u_{i}}\otimes \eta_{v_{i}}$ into $\mathfrak{o}(D_{u_{i}})\otimes \mathfrak{o}(D_{v_{i}})$; see Figure \ref{fig:gluing-for-d2-zero}.

  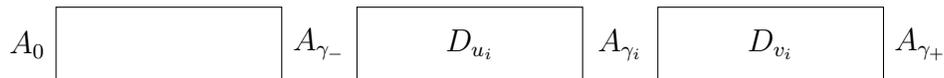
\begin{figure}[H]
    \centering
    \begin{tikzpicture}
      \draw (0,0) rectangle +(3,1) (4,0) rectangle +(3,1) (8,0) rectangle +(3,1);
      \path (0,0.5) node[left] {$A_{0}$} (3.5,0.5) node {$A_{\gamma_{-}}$} (7.5,0.5) node{$A_{\gamma_{i}}$} (11,0.5) node[right] {$A_{\gamma_{+}}$};
      \path (5.5,0.5) node{$D_{u_{i}}$} (9.5,0.5) node{$D_{v_{i}}$};
    \end{tikzpicture}
    \caption{Gluing linearized operators when proving $d^2=0$; the two ends of the moduli space correspond to setting $i=1,2$.}
    \label{fig:gluing-for-d2-zero}
  \end{figure}

  By the associativity of gluing, one can combine \eqref{eq:k-g-map} and \eqref{eq:linear-g-o-l} to conclude that the two identifications (obtained by gluing $u_{1},v_{1}$ and $u_{2},v_{2}$):
  \begin{equation*}
    \mathfrak{o}_{\gamma_{-}} \otimes \mathfrak{o}(D_{u})\simeq \mathfrak{o}_{\gamma_{+}}
  \end{equation*}
  are opposite. The desired relation \eqref{eq:relation-u-v-differential} follows. For similar arguments, see \cite[pp.\,35]{fh_coherent}, \cite[\S4.2]{floer-hofer-sh-i}, and \cite[\S9.5.2]{abouzaid_monograph}.
\end{proof}

\subsubsection{Continuation maps}
\label{sec:continuation-maps-hf}

A path of Hamiltonian systems $H_{s,t}$ and almost complex structures $J_{s,t}$ is called \emph{admissible continuation data} provided:
\begin{enumerate}
\item There is a fixed star-shaped domain $\Omega$ so that $H_{s,t}$ is $1$-homogeneous and $J_{s,t}$ is translation invariant outside $\Omega$,
\item There are admissible Hamiltonian systems $H_{\pm,t}$ and almost complex structures $J_{\pm,t}$ so that $H_{s,t}=H_{\pm,t}$ and $J_{s,t}=J_{\pm,t}$ for $\pm s$ sufficiently large,
\item $\bd_{s} H_{s,t}(z)\le 0$ holds for $z$ outside a compact set,
\item the moduli space of continuation strips is cut out transversally (see below).
\end{enumerate}

Associated to such data is the moduli space $\mathscr{M}(H_{s,t},J_{s,t})$ of \emph{continuation strips} satisfying:
\begin{equation*}
  \left\{
    \begin{aligned}
      &u:\R \times [0,1]\to T^{*}M,\\
      &\bd_{s}u+J_{s,t}(u)(\bd_{t}u-X_{H_{s,t}}(u))=0,\\
      &u(s,0), u(s,1)\in \nu^*N,
    \end{aligned}
  \right.
\end{equation*}
and, by (iv), this moduli space is cut out transversally.

Standard asymptotic analysis implies any continuation strip converges to chords of $H_{\pm,t}$ as $s\to\pm\infty$. The moduli space splits into components of varying dimensions depending on the asymptotic chords; the component with asymptotics $\gamma_{\pm}$ is a smooth manifold of dimension $\mu(\gamma_{-})-\mu(\gamma_{+})$.
Let $\mathscr{M}_{0}$ be the zero-dimensional component of $\mathscr{M}$; the compactness results in \S\ref{sec:compactness-cont-strips} imply $\mathscr{M}_{0}$ is a finite set of points. Moreover, the linearization $D_{u}$ of the continuation map equation at any solution $u$ is an isomorphism, and hence the orientation line $\mathfrak{o}(D_{u})$ has a canonical generator. On the other hand, the framework in \S\ref{sec:coher-orient} furnishes a canonical identification between $\mathfrak{o}_{\gamma_{+}} \simeq \mathfrak{o}_{\gamma_{-}} \otimes \mathfrak{o}(D_{u})$, and hence a generator of $\mathfrak{o}(D_{u})$ induces an isomorphism $\mathfrak{c}_{u}:\mathfrak{o}(\gamma_{+})\to \mathfrak{o}(\gamma_{-})$.

Twisting this map by the monodromy of the local system $\mathscr{L}$ and summing over $u\in \mathscr{M}_{0}$ gives the \emph{continuation map}:
\begin{equation*}
  \mathfrak{c}(H_{s,t},J_{s,t})=\sum \mathfrak{c}_{u}\otimes \mathscr{L}(u);
\end{equation*}
as with the Floer differential, this is interpreted as a map $\mathrm{CF}(H_{-},J_{-})\to \mathrm{CF}(H_{+},J_{+})$.

Standard gluing results, similar to the proof of Lemma \ref{lemma:differential}, imply that $\mathfrak{c}(H_{s,t},J_{s,t})$ is a chain map (see \cite[\S9.6.2]{abouzaid_monograph}). Counting elements in parametric moduli spaces shows that the chain homotopy class of $\mathfrak{c}$ is unchanged under homotopies of the path $H_{s,t},J_{s,t}$ with fixed endpoints, and so that (i), (ii), and (iii) hold during the homotopy.
Given two sets of continuation data $H_{s,t}^{i}$, $J_{s,t}^{i}$, $i=0,1$, with the same endpoints, one can linearly interpolate $H_{s,t}^{0}$ to $H_{s,t}^{1}$ and use contractibility of the space of complex structures to join $J_{s,t}^{0}$ the paths $J_{s,t}^{1}$. In this fashion, one shows the chain homotopy class of $\mathfrak{c}$ is independent of the choice of continuation data.

\subsubsection{Compactness for continuation strips}
\label{sec:compactness-cont-strips}

It follows from \cite[\S2.2.5]{brocic_cant} that any sequence of continuation maps $u_{n}$ satisfying an a priori energy bound and bound on the first derivatives remains in a fixed compact set (depending on the energy bound and first derivative bound). We note that, because of the first derivative bound, it suffices to bound the image of $(\R\setminus [s_{0},s_{1}])\times [0,1]$, i.e., we can ignore any sub-cylinder of finite length when proving the $C^{0}$ bound. By this trick, we can ignore the intricacies of the continuation map equation and focus on the translation invariant Floer's equation (which is what \cite{brocic_cant} considers).

An energy bound implies a first derivative bound, using bubbling analysis as in \S\ref{sec:a-priori-estimate-sobolev}. It is important that the size of a Hamiltonian perturbation is bounded with respect to a translation invariant metric, and hence the rescaled equations always converge to holomorphic equations (spheres or disks with boundary on $\nu^{*}N$). A priori first derivative bounds also imply bounds on the higher derivatives, using elliptic regularity, see, e.g., \cite[\S3.1.2]{brocic_cant}; see also \cite[\S2.2.4]{cant_sh_barcode} for a similar argument.

\subsubsection{Invariance of positive wrapped Floer cohomology}
\label{sec:invariance-pos-wrap}

Consider two sequences $\mathrm{CF}_{n}$ and $\mathrm{CF}_{n}'$ so that the data for $\mathrm{CF}_{0}'$ is more positive than the data for $\mathrm{CF}_{0}$. Then one can define a first continuation map $\mathrm{CF}_{0}\to \mathrm{CF}_{0}'$. Using the cofinal property, one can find:
\begin{enumerate}
\item a subsequence $k:\mathbb{N}\to \mathbb{N}$ (strictly increasing map) with $k(0)=0$,
\item continuation maps $\mathrm{CF}_{n}\to \mathrm{CF}_{k(n)}'$ for each $n$.
\end{enumerate}
Here the continuation maps are defined using some continuation data. The importance of the cofinal assumption is that the data used to define $\mathrm{CF}'_{k}$ will eventually be more positive than the data used to define $\mathrm{CF}_{n}$, and so the continuation map can eventually be defined.

Because all of our maps are defined as (compositions of) continuation maps, all the squares of the form:
\begin{equation}\label{eq:CF_square}
  \begin{tikzcd}
    {\mathrm{CF}_{n}}\arrow[d,"{}"]\arrow[r,"{}"] &{\mathrm{CF}_{n+1}}\arrow[d,"{}"]\\
    {\mathrm{CF}'_{k(n)}}\arrow[r,"{}"] &{\mathrm{CF}'_{k(n+1)}}
  \end{tikzcd}
\end{equation}
commute up to homotopy. The key property we will require from cofibrations is that it allows one to upgrade ``commutes up to homotopy'' to ``commutes on the nose'':
\begin{lemma}\label{lemma:cofibrant-lemma}
  Suppose that one has a square in the category of differential graded $\Z$-modules which commutes up to chain homotopy:
  \begin{equation*}
    \begin{tikzcd}
      {C_{0}}\arrow[d,"{f_{0}}"]\arrow[r,"{c}"] &{C_{1}}\arrow[d,"{f_{1}}"]\\
      {C'_{0}}\arrow[r,"{c'}"] &{C'_{1}},
    \end{tikzcd}
  \end{equation*}
  and suppose $c$ is a cofibration. Then one can alter $f_{1}$ in its chain homotopy class so as to make the diagram commute.
\end{lemma}
\begin{proof}
  Let $K$ be a chain homotopy so $f_{1}c-c'f_{0}=d_{1}'K+Kd_{0}$. Since $c$ is a cofibration, there is a map $r:C_{1}\to C_{0}$, not necessarily a chain map, so that $rc=1$. Define the replacement of $f_{1}$ by the formula: $$f_{1}':=f_{1}-d_{1}'Kr-Krd_{1}.$$ One checks that $f_{1}'c=f_{1}c-d_{1}'K-Krd_{1}c$. Using $rd_{1}c=rcd_{0}=d_{0}$, one concludes $f_{1}'c-c'f_{0}=0$, as desired.
\end{proof}

Applying this result, replace all of the chosen continuation maps $\mathrm{CF}_{n}\to \mathrm{CF}'_{k(n)}$ by chain-homotopic maps so that all the squares \eqref{eq:CF_square} commute. One obtains commuting short exact sequences:
\begin{equation*}
  \begin{tikzcd}
    0\arrow[r]&{\mathrm{CF}_{0}}\arrow[d,"{}"]\arrow[r,"{}"] &{\mathrm{CF}_{n}}\arrow[d,"{}"]\arrow[r,"{}"] &{\mathrm{CF}_{n}/\mathrm{CF}_{0}}\arrow[d,"{}"]&0\arrow[from=1-4]\\
    0\arrow[r]&{\mathrm{CF}_{0}'}\arrow[r,"{}"] &{\mathrm{CF}_{k(n)}'}\arrow[r,"{}"] &{\mathrm{CF}_{k(n)}'/\mathrm{CF}_{0}'}&0\arrow[from=2-4]
  \end{tikzcd}.
\end{equation*}
These sequences are natural with respect to the maps $\mathrm{CF}_{n}\to \mathrm{CF}_{n+1}$, etc, and hence one can take the colimit to obtain a diagram of short exact sequences:
\begin{equation*}
  \begin{tikzcd}
    0\arrow[r]&{\mathrm{CF}_{0}}\arrow[d,"{}"]\arrow[r,"{}"] &{\colim\mathrm{CF}_{n}}\arrow[d,"{}"]\arrow[r,"{}"] &{\colim \mathrm{CF}_{n}/\mathrm{CF}_{0}}\arrow[d,"{}"]&0\arrow[from=1-4]\\
    0\arrow[r]&{\mathrm{CF}_{0}'}\arrow[r,"{}"] &{\colim \mathrm{CF}_{n}'}\arrow[r,"{}"] &{\colim \mathrm{CF}_{n}'/\mathrm{CF}_{0}'}&0\arrow[from=2-4]
  \end{tikzcd};
\end{equation*}
here we use that taking colimits preserves exactness and the colimit of along the subsequence $k(n)$ computes the full colimit; this follows, for example, from \cite[\S{IX}.2]{maclean-cat-working-math} and \cite[\S07N7, Lemma 10.8.8]{stacks-project}.

One concludes from the long exact sequence in cohomology and the five lemma that the induced map:
\begin{equation*}
  \mathrm{H}(\colim \mathrm{CF}_{n}/\mathrm{CF}_{0})\to \mathrm{H}(\colim\mathrm{CF}'_{n}/\mathrm{CF}'_{0})
\end{equation*}
is an isomorphism which commutes with the long-exact sequence. This completes the proof of the invariance of positive wrapped Floer cohomology (Lemma \ref{lemma:pos-wrap-invariance}) in the special case when $\mathrm{CF}_{0}'$ is more positive than $\mathrm{CF}_{0}$. In the general case when the data for $\mathrm{CF}_{0}$ and $\mathrm{CF}_{0}'$ are not comparable, one picks a cofibrant-cofinal sequence $\mathrm{CF}_{n}''$ so that the data for $\mathrm{CF}_{0}''$ is less positive than the data for both $\mathrm{CF}_{0}$ and $\mathrm{CF}_{0}'$, and then applies the preceding discussion (it follows from \S\ref{sec:spec-cofibr-sequ} that we can always find such $\mathrm{CF}''$). This completes the proof of invariance in general.

\subsubsection{Strong maximum principle}
\label{sec:strong-maxim-princ}

In the construction of a particular cofibrant-cofinal sequence in \S\ref{sec:spec-cofibr-sequ}, we will have occasion to use the following strong maximum principle:
\begin{prop}
  Let $r$ be a positive function which is $1$ homogeneous with respect to the Liouville flow, let $H_{s}=f_{s}(r)$ be a Hamiltonian system so that $\bd_{s}\bd_{r}f_{s}(r)\le 0$, and suppose that $J$ is the SFT-type almost complex structure satisfying $JZ=X_{r}$. Then no non-constant solution to $\bd_{s}u+J(u)(\bd_{t}-X_{H_{s}}(u))=0$ attains a local maximum at an interior point or on $s=0,1$ boundary components satisfying conormal Lagrangian boundary conditions.
\end{prop}
\begin{proof}
  Set $\mathrm{r}(s,t) := r(u(s,t) )$. The fact that $\lambda$ vanishes on conormal Lagrangians implies that: $$\bd_{t} \mathrm{r} =\d r(J\bd_{s}u+f'_{s}(r)X_{r})=\d r(J\bd_{s}u)=-\lambda(\bd_{s}u)=0,$$ where we have used the SFT condition $\d r\circ J=-\lambda$. Therefore $r$ can be doubled as a $C^{2}$ function across the boundary. Using that $\bd_{s} \mathrm{r} = \lambda(\bd_{t} u) - \mathrm{r} f'_{s}(\mathrm{r})$, one shows: $$\Delta \mathrm{r}= | \bd_{s} u |^2_J - \mathrm{r} \cdot (\bd_{s} \bd_{r} f_{s}) (\mathrm{r}) - \mathrm{r} \cdot f_{s}''(\mathrm{r}) \cdot \bd_{s} \mathrm{r} .$$ Hence, at every critical point of $\mathrm{r}$ we have $\Delta \mathrm{r} \ge 0$, which yields the desired result. The computation is well-known; see, e.g., \cite[\S1.3]{viterbo_functors_and_computations_1}, \cite[\S{D.3}]{ritter_tqft}.
\end{proof}

\subsubsection{Reeb chords and positive wrapped Floer cohomology}
\label{sec:reeb-chords-pos-wrap-floer}

In this section, we prove that in the absence of Reeb chords $\mathrm{HW}_{+}$ vanishes for any choice of local system. This result is well-known, see \cite{abouzaid_seidel_open_string_analogue,ritter_tqft}, and can be considered as an open string analogue of the existence result for Reeb orbits of ideal boundaries in \cite{viterbo_functors_and_computations_1} assuming $\mathrm{SH}_{+}$ is non-zero.
\begin{prop}
  If the ideal boundary of $\nu^{*}N$ has no $\alpha$-Reeb chords of positive length for some contact form $\alpha$ then $\mathrm{HW}_{+}(\nu^{*}N;\mathscr{L}) = 0$ for all local systems $\mathscr{L}$.
\end{prop}
\begin{proof}
  It is enough to construct a cofinal-cofibrant sequence $H_{n,t}$ such that all continuation maps $CF(H_{n,t}) \to CF(H_{n+1,t})$ are equal to the identity.

  Let $r$ be the radial coordinate on the symplectization end of $T^{*}M$ so that $\alpha=\lambda/r$. The Hamiltonian vector field for $r$ induces the Reeb flow for $\alpha$ on the ideal boundary. Let $\epsilon>0$ and $H_{0,t}$ be a generic Hamiltonian system so that $H_{0,t}=\epsilon r$ holds on the end $\set{r\ge 1}$, and $H_{0,t}\le \epsilon$ holds on the compact domain $\set{r\le 1}$. Let $J$ be the SFT type almost complex structure for $\alpha$ (extended arbitrarily to $\set{r\le 1}$). By picking $\epsilon$ small enough, the canonical map $\mathrm{HW}_{0}\to \mathrm{HF}(H_{0,t},J)$ is an isomorphism.

  We pick $H_{n,t}$ so that it agrees with $H_{0,t}$ on the set $\set{r\le1}$, $H_{n,t} =n r$ on $\set{r\ge2}$ and to be convex on the region $\set{1\le r\le 2}$. Since the flow by $H_{n,t}$ agrees with the flow of $H_{0,t}$ on the invariant set $\set{r\le 1}$, and there are no Reeb chords of $\nu^{*}N$, the chords of $H_{0,t}$ are exactly those of $H_{n,t}$. Consider the linear interpolation from $H_{n-1,t}$ to $H_{n,t}$ as continuation data. The strong maximum principle from \S\ref{sec:strong-maxim-princ} implies that the only rigid continuation strips from $H_{n,t}$ to $H_{n+1,t}$ are the stationary\footnote{A strip $u(s,t)$ is called \emph{stationary} if it is independent of the $s$ coordinate.} strips; this uses that the continuation data is $s$-independent in $\set{r\le 1}$; see \S\ref{sec:spec-cofibr-sequ} below for a variant of this argument.

  It follows that the maps $\mathrm{CF}(H_{n,t},J)\to \mathrm{CF}(H_{n+1,t},J)$ are trivially cofibrations, and the map $\mathrm{CF}(H_{0,t},J)\to \mathrm{CF}(H_{n,t},J)$ is an isomorphism. In particular, $$\mathrm{CF}(H_{n,t},J)/\mathrm{CF}(H_{0,t},J)=0,$$ and hence $\mathrm{HW}_{+}$ is $0$, as it is the colimit of a sequence of trivial groups, as desired.
\end{proof}

\subsubsection{A cofibrant-cofinal sequence}
\label{sec:spec-cofibr-sequ}

In this section we construct a cofibrant-cofinal sequence, showing that the definition of $\mathrm{HW}_{+}(\nu^{*}N;\mathscr{L})$ is not vacuous. Moreover, this construction is used to prove the isomorphism in \S\ref{sec:isom-from-morse}.

Fix $r=r_{g}=|p|_{g}$. Say that $a$ is a \emph{critical value} if $a$ is the length of a normal geodesic chord with boundary on $N$. Picking our metric generically ensures that the set $\set{0, \sigma_{1}, \sigma_{2}, ...}$ of critical values is discrete (and closed). We can further assume that $0 < \sigma_{1} < \sigma_{2} < \dots$. It follows that we can pick sequences $0<x_{1}<x_{2}<\dots$ so that the following holds:
\begin{enumerate}
\item The $k$th positive critical value $\sigma_k$ is the midpoint of $(x_{2k},x_{2k+1})$,
\item $x_{2k+1}^{2}-x_{2k}^{2}<x_{2k}^{2}-x_{2k-1}^{2}$.
\end{enumerate}
Note that the first condition implies that $\lim_{n\to\infty}x_{n}=\infty$. The second condition will be used for the energy estimate of continuation strips, to exclude certain configurations.

\begin{figure}[H]
  \centering
  \begin{tikzpicture}[xscale=1.2]
    \draw (-1,0)--(10,0);
    \path[every node/.style={fill, black,circle, inner sep=1pt}] (-1,0)node{} (0,0)node{} (1,0)node{}--(1.4,0)node{}--(2.2,0)node{}--(2.6,0)node{}--(4,0)node{}--(4.4,0)node{}--(6,0)node{}--(6.4,0)node{}--(8.5,0)node{}--(9.1,0)node{};
    \path[every node/.style={below}]
    (-1,0)node{0} (0,0)node{$x_{1}$} (1,0)node{$x_{2}$}--(1.4,0)node{$x_{3}$}--(2.2,0)node{$x_{4}$}--(2.6,0)node{$x_{5}$}--(4,0)node{$x_{6}$}--(4.4,0)node{$x_{7}$}--(6,0)node{$x_{8}$}--(6.4,0)node{$x_{9}$}--(8.5,0)node{$x_{10}$}--(9.1,0)node{$x_{11}$};
    \path[every node/.style={fill=red,circle, inner sep=1pt}] (1.2,0)node{}--(2.4,0)node{}--(4.2 , 0)node{}--(6.2, 0)node{}--(8.8,0)node{};
    \path[every node/.style={above}] (1.2,0)node{$\sigma_{1}$}--(2.4,0)node{$\sigma_{2}$}--(4.2,0)node{$\sigma_{3}$}--(6.2,0)node{$\sigma_{4}$}--(8.8,0)node{$\sigma_{5}$};
  \end{tikzpicture}
  \label{fig:xk_picture}
  \caption{Sequence $x_{k}$ used in the construction.}
\end{figure}
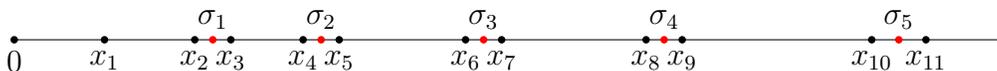

Define a sequence of functions $f_{n}:[0, +\infty) \to [0, +\infty)$ as follows:
\begin{equation}\label{eq:special_construction}
  f_{n}=\left\{
    \begin{aligned}
      &\frac{1}{2}(r^{2}+x_{n}^{2})&&\text{ for }r<x_{n}-\delta_{n},\\
      &x_{n}r&&\text{ for }r>x_{n}+\delta_{n},
    \end{aligned}
  \right.
\end{equation}
so that $f_{n}$ is smooth, increasing and satisfies $x_{n}-\delta_{n}\le f_{n}' (r)\le r$ in the region defined by $x_{n}-\delta_{n} \le r \le x_{n} + \delta_{n};$ see Figure \ref{fig:function_f_n}. We require that $\delta_{n}$ is smaller than $\frac{1}{3}\min\set{x_{n}-x_{n-1},x_{n+1}-x_{n}}$; see Figure \ref{fig:supp-per}.

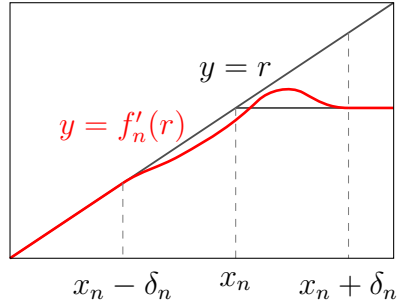
\begin{figure}[H]
  \centering
  \begin{tikzpicture}[xscale=3,yscale=2]
    \draw[line width=0.7pt, black!70!white] plot[domain= 0: 1.7,variable=\x] ({\x},{\x});
    \draw[line width=0.7pt, black!70!white] plot[domain=1 : 1.7, variable=\x] ({\x}, 1);
    \draw[line width=1pt,red] plot[domain= 0: 0.5,variable=\x] ({\x},{\x});
    \draw[line width=1pt,red] (0.5,0.5)  to [out=45, in=220] (0.75, 0.68) to [out=40, in=235] (1.1,1.05) to [out=55, in=140] (1.3, 1.1) to [out=320, in=180] (1.5,1);
    \foreach \x in {0.5, 1, 1.5} {
      \draw[dashed,black!50!white] ({\x},{\x})coordinate(X\x)--({\x},0);
    }
    \draw[line width=1pt, red] plot[domain=1.5 : 1.7, variable=\x] ({\x}, 1);
    \draw (0,0) rectangle (1.7,1.7);
    \node at (1,1.1) [above] {$y=r$};
    \node at (0.5, 0.7)[above, red] {$y= f_{n}'(r)$};
    \node at (0.5,0)[below,inner sep=5pt] {$x_{n} - \delta_{n}$};
    \node at (1,0)[below,inner sep=5pt] {$x_{n}$};
    \node at (1.5, 0)[below,inner sep=5pt] {$x_{n}+\delta_{n}$};
  \end{tikzpicture}
  \caption{Graph of the function $f_{n}'$.}
  \label{fig:function_f_n}
\end{figure}

\begin{figure}[H]
  \centering
  \begin{tikzpicture}
    \draw (0,0) -- (12,0);
    \path[every node/.style={draw,circle,fill,inner sep=1pt}] (2,0)node{} (6,0)node{} (8,0)node[fill=red]{} (10,0)node{};
    \path[every node/.style={below}] (2,0)node{$x_{2k-1}$} (6,0)node{$x_{2k}$} (8,0)node{$\sigma_{k}$} (10,0)node{$x_{2k+1}$};
    \foreach \x in {2, 6, 10} {
      \draw (\x,0)+(1,0.2)--+(1,-0.2) (\x,0)+(-1,0.2)--+(-1,-0.2);
    }
    \draw[<->] (7,0.3)--node[above]{$2\delta_{2k}$}(5,0.3);
  \end{tikzpicture}
  \caption{The perturbation term $\kappa_{n,t}$ are supported in the union of the intervals $[x_{j}-\delta_{j},x_{j}+\delta_{j}]$ for $j=1,\dots,n$.}
  \label{fig:supp-per}
\end{figure}
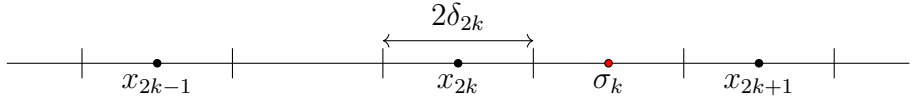

Now define a sequence of Hamiltonians by:
\begin{equation*}
  H_{n} = f_{n}(r) + b,
\end{equation*}
where $b\ge 0$ is $C^2$ small function supported in the region $\set{r\le x_{1}}$. We pick $b$ so that the chords at $r=0$ are non-degenerate.

In order to achieve the transversality of moduli spaces, we perturb $H_{n,t}:=H_{n}+\kappa_{n,t}$ with a \emph{non-negative} time-dependent perturbation $\kappa_{n,t}\ge 0$. For $n\ge 2$, we require that $\kappa_{n,t}$ coincides with $\kappa_{n-1, t}$ on $r \le x_{n}-\delta_{n}$, and $\kappa_{n,t} - \kappa_{n-1,t}$ is supported in the region where $r\in (x_{n} -\delta_{n},x_{n}+\delta_{n})$ and $t\in (1/3,2/3)$; see Figure \ref{fig:supp-per}. In particular, $H_{n,t}$ agrees with $H_{n}$ whenever $r$ is close to a critical value $\sigma_{k}$. By requiring that $\kappa_{n,t}$ is small enough, we may assume the chords of $H_{n}$ and $H_{n} + \kappa_{n,t}$ with endpoints on $\nu^{*}N$ coincide, i.e., no new chords were created by adding the perturbation terms. Because the metric is non-degenerate, all of the chords of $H_{n}$ (and hence $H_{n,t}$) with endpoints on $\nu^{*}N$ are non-degenerate (and there are only finitely many); see \cite[\S2.1]{abbondandolo_schwarz} for details on the non-degeneracy.

Define $F_{k}$ to be the direct sum of $\mathfrak{o}_{\gamma}\otimes \mathscr{L}_{\gamma}$ where $\gamma$ ranges over the chords with $r=\sigma_k$, yielding a decomposition:
\begin{equation}\label{eq:decomposition}
  \mathrm{CF}(H_{n,t})=\mathrm{CF}(H_{1})\oplus F_{1}\oplus \dots \oplus F_{\lfloor \frac{n-1}{2}\rfloor}.
\end{equation}

A straightforward computation shows that the action $$\mathscr{A}_{n}(\gamma) = \int_0^1 H_{n,t}(\gamma)\mathrm{d}t - \gamma^* \lambda$$ of a chord $\gamma$ on the level $r=\sigma_{k}$ in the system $H_{n}$ satisfies $\mathscr{A}_{n}(\gamma) = \frac{1}{2}(x_{n}^{2}-\sigma_{k}^{2})$, where $\sigma_{k}=\frac{1}{2}(x_{2k}+x_{2k+1})$ is the $k$th positive critical value. The action $\mathscr{A}_{n}(\gamma)$ of chords $\gamma$ in the summand $\mathrm{CF}(H_{1},J)$ is close to $\frac{1}{2}x_{n}^{2}$ (how close depends on $b$).

\begin{prop}
  Let $H_{n,t,s}$ be the continuation data $(1-\beta(s))H_{n+1,t}+\beta(s)H_{n,t}$ where $\beta$ is a cut-off function $\R\to [0,1]$, and let $J$ to be SFT type almost complex structure as in \S\ref{sec:strong-maxim-princ}. The continuation morphism: $$\mathfrak{c}(H_{n,t,s},J):\mathrm{CF}(H_{n,t},J)\to \mathrm{CF}(H_{n+1,t},J)$$ equals the inclusion, with respect to the decompositions in \eqref{eq:decomposition}. In particular, $(H_{n,t},J)$ with this continuation data forms a cofinal-cofibrant sequence.
\end{prop}
\begin{proof}
  There are two sorts of continuation maps to consider, those when $n=2k-1$ is odd, which are supposed to be the identity with respect to \eqref{eq:decomposition}, and those when $n=2k$ is even, which is supposed to be a proper inclusion with cokernel $F_{k}$.

  First, we consider the case when $n$ is odd. We will show that all continuation strips are contained in the region where $r\le x_{n}- \delta_{n}$. By a compactness argument, we can assume that all perturbations satisfy $\kappa_{n,t}=0$. In this case, $H_{n,s}=h_{s}(r)$ holds in the region where $r\ge x_{n}-\delta_{n}$ and $\bd_{s}h'_{s}(r)\le 0$. Indeed, the functions $f_n$ and $f_{n+1}$ are chosen in a way that $f_{n}' \leq f_{n+1}'$, so: $$\bd_{s}((1-\beta(s)) f_{n+1}'(r) + \beta(s) f_{n}'(r)) = \beta'(s) (f_{n}'(r) - f_{n+1}'(r)) \le 0.$$  Thus we may apply the strong maximum principle of \S\ref{sec:strong-maxim-princ} to conclude that all continuation strips are contained in the region where $r\le x_{n}- \delta_{n}$.

  In this region, the Hamiltonian vector field $X_{H_{n,t,s}}$ is independent of the $s$ variable; therefore, all continuation strips solve the equation for the Floer differential. By picking the perturbation terms $\kappa_{n,t}$ sufficiently generically, the moduli spaces used to define the Floer differential are cut out transversally, hence there are no non-stationary continuation strips for index reasons. Indeed, a reparametrization by translation in the $s$ direction implies that the component of any non-stationary solution $u$ must be of positive dimension. Consequently, the continuation map is the identity when $n$ is odd.

  Next consider the case $n=2k$. Write $\mathrm{CF}(H_{n,t})=C$ and $\mathrm{CF}(H_{n+1,t})=C\oplus F_{k}$. The same strong maximum principle argument given above implies that only terms $C\to C$ are given by constant continuation strips.

  It remains to prove that there are no terms $C\to F_{k}$ in the continuation map. We will show there are no such continuation strips by considering the action.

  Since the continuation data $H_{n,t,s}$ is non-increasing (i.e., $\bd_{s}H_{n,s}\le 0$), we conclude that the energy of any continuation strip is bounded above by $\mathscr{A}_{n+1}(\gamma_{-})-\mathscr{A}_{n}(\gamma_{+})$. Using (ii) from the construction, we show that this action difference is strictly negative. Indeed, recalling that for $n=2k$, $\mathscr{A}_{n+1}(\gamma_{-})$ equals $\frac{1}{2}(x_{2k+1}^{2}-\sigma_{k}^{2})$ while $\mathscr{A}_{n}(\gamma_{+})$ is bounded from below by $\frac{1}{2}(x_{2k}^{2}-\sigma_{k-1}^{2})$, so:
  \begin{equation*}
    \begin{aligned}
      \mathscr{A}_{n+1}(\gamma_{-})-\mathscr{A}_{n}(\gamma_{+})
      &\le \frac{1}{2}(x_{2k+1}^{2}-x_{2k}^{2}-\sigma_{k}^{2}+\sigma_{k-1}^{2})\\
      &\le \frac{1}{2}(x_{2k+1}^{2}-x_{2k}^{2}-x_{2k}^{2}+x_{2k-1}^{2})<0,
    \end{aligned}
  \end{equation*}
  using (ii) to conclude the final inequality. Thus there are no continuation strips $C\to F_{k}$, and so the continuation map $C\to C\oplus F_{k}$ equals the inclusion $C\to C$, as desired.
\end{proof}

\subsubsection{Comparison with a quadratic Hamiltonian}
\label{sec:comp-with-quadr}

In this section we define the Floer complex $\mathrm{CF}(Q_{t})$ for a quadratic Hamiltonian system $Q_{t}=\frac{1}{2}r^{2}+\kappa_t$, and prove it represents the colimit of $\mathrm{CF}(H_{n,t})$ where $H_{n,t}$ is the cofinal-cofibrant sequence constructed in \S\ref{sec:spec-cofibr-sequ}. Here $\kappa_{t}:=b+\lim_{n\to\infty} \kappa_{n,t}$. See \cite[\S{C}]{ritter_tqft} for more details regarding quadratic Hamiltonian systems and their role in wrapped Floer cohomology.

First of all, note that $Q_{t}=\lim_{n\to\infty} (H_{n,t}-\frac{1}{2}x_{n}^{2})$ holds on compact sets; indeed, on the region $\set{r\le x_{n}}$, $Q_{t}$ equals $H_{n,t}-\frac{1}{2}x_{n}^{2}$.

Define $\mathrm{CF}(Q_{t})$ to be the direct sum of $\mathfrak{o}_{\gamma}\otimes  \mathscr{L}_{\gamma}$ as $\gamma$ ranges over all chords of the system generated by $Q_{t}$. Then $\mathrm{CF}(Q_{t})$ splits as a direct sum $\mathrm{CF}(H_{1})\oplus F_{1}\oplus F_{2}\oplus \dots$. Every Floer cylinder for $Q_{t}$ is a Floer cylinder for some $H_{n,t}$ if $n$ is large enough, by \S\ref{sec:strong-maxim-princ}; all that is required is that the asymptotics of $\gamma$ lie in the region where $Q_{t}$ equals $H_{n,t}-\frac{1}{2}x_{n}^{2}$. Moreover, the summand $\mathrm{CF}(H_{2k+1,t})=\mathrm{CF}(H_{1})\oplus F_{1}\oplus \cdots \oplus F_{k}$ is a subcomplex of $\mathrm{CF}(Q_{t})$, i.e., the obvious inclusion $\mathrm{CF}(H_{2k+1,t})$ into $\mathrm{CF}(Q_{t})$ is a chain map.

It follows from these facts that the quotient $\mathrm{CF}(Q_{t})/\mathrm{CF}(H_{1})$ is a chain-level model for $\mathrm{HW}_{+}(\nu^{*}N;\mathscr{L})$.

\subsection{Isomorphism from Morse homology to wrapped Floer cohomology}
\label{sec:isom-from-morse}

The goal in this section is to construct an isomorphism:
\begin{equation*}
  \Theta:\mathrm{HM}(\mathscr{P},N;\mathscr{E}; \mathscr{L}) \to \mathrm{HW}_{+}(\nu^*N; \mathscr{L}).
\end{equation*}
Our approach to defining $\Theta$ and to proving it is an isomorphism follows \cite[Theorem 3.1]{abbondandolo_schwarz}; see also \cite{abouzaid_based_loop,abouzaid_monograph}. Fix throughout this section a pseudogradient for the Riemannian energy functional.

The morphism is defined by counting elements in the moduli space $\mathscr{M}^{\Theta}(H,J)$ of pairs $(u,q)$ where $q$ lies in the unstable manifold of a geodesic normal chord and $u$ solves the elliptic boundary problem:
\begin{equation}\label{eq:theta-moduli-space}
  \begin{tikzpicture}[baseline={(0,-0.1)}]
    \draw (0,-0.5)coordinate(A)--+(4,0)coordinate(B)--+(4,1)coordinate(C)--+(0,1)coordinate(D);
    \draw[dashed] (A)--node[left]{$\gamma$}(D);
    \path (A)--node[below]{$\nu^{*}N$}(B) (D)--node[above]{$\nu^{*}N$}(C)--node[right]{$T^{*}M_{q(t)}$}(B);
  \end{tikzpicture}
  \hspace{1cm}
  \left\{
    \begin{aligned}
      &u:(-\infty,0]\times [0,1]\to T^{*}M,\\
      &\bd_{s}u+J_{t}(u)(\bd_{t}u-X_{H_{t}}(u))=0,\\
      &u(0,t)=(q(t),p(t)),\\
      &u(s,0), u(s,1)\in \nu^*N.\\
    \end{aligned}
  \right.
\end{equation}

By construction, each element of $\mathscr{M}^{\Theta}$ is asymptotic at its left end to a Hamiltonian chord $\gamma(t)$ for the system $H_{t}$, while the flow line starting at $q(t)$ converges to some geodesic normal chord $x(t)$ under the flow by the positive pseudogradient. Note that we can think of $\mathscr{M}^{\Theta}$ as a parametric moduli space of pairs $(u,q)$ where $q$ (locally) varies in a finite dimensional space (namely, the unstable manifold of $x$).

If $(u,q)$ is a rigid element in $\mathscr{M}^{\Theta}$, then the linearized operator $T_{u,q}$ associated to the problem is an isomorphism; as usual for parametric moduli spaces, this determines an isomorphism between $\mathfrak{o}(D_{u,q})$ (the orientation line of the linearization with $q$ fixed) and $\mathfrak{o}_{x}$ (the orientation line for the unstable manifold of $x$); see \S\ref{sec:param-moduli-space}. In \S\ref{sec:ident-orient-line} we describe a coherent way of identifying $\mathfrak{o}(D_{u,q})$ with the orientation line $\mathfrak{o}_{\gamma}$. Let us denote by $\Theta_{(u,q)}$ this induced map on orientation lines $\mathfrak{o}_{x}\to \mathfrak{o}_{\gamma}$.

With regard to the local system $\mathscr{L}$, observe that every element $(u,q)$ determines a path from the geodesic chord $x(t)$ to (the projection of) the Hamiltonian chord $\gamma(t)$, and hence determines a map $\mathscr{L}(u,q)$ between $\mathscr{L}_{x}$ and $\mathscr{L}_{\gamma}$.

Therefore, modulo the question of sums converging, we can interpret the count of rigid elements in $\mathscr{M}^{\Theta}$ as a map from $\mathrm{CM}(\mathscr{P},N;\mathscr{E};\mathscr{L})$ to $\mathrm{CF}(H,J;\mathscr{L})$ by the formula:
\begin{equation}\label{eq:theta_map_defn}
  \Theta := \sum_{(u,q)}\Theta_{(u,q)}\otimes \mathscr{L}(u,q).
\end{equation}
It is necessary to impose restrictions on the system $H_{t}$ in order for this $\Theta$ map to be well-defined and a chain map; in \cite[\S3]{abbondandolo_schwarz}, the authors prove that a variant of $\Theta$ is a chain map and a quasi-isomorphism in the case when $H_{t}$ has quadratic growth in the $p$ coordinate.

We note that, since we work with relative Morse homology, $\Theta$ will be a chain map only after we quotient the codomain $\mathrm{CF}(H,J;\mathscr{L})$ by a certain subcomplex; see \S\ref{sec:chain-map-property} for the precise statement.

For special choices of the system $H_{t}$ (namely, small perturbations of $H=\frac{1}{2}r^{2}$), $\Theta$ is shown to be an quasi-isomorphism in \cite{abbondandolo_schwarz} via a \emph{diagonal argument}, i.e., with respect to bases ordered by their action, $\Theta$ is a triangular matrix with all ones on the diagonal; see \S\ref{sec:defin-full-morse} and \S\ref{sec:diagonal-argument-theta} for more details. Their results, especially those in the follow-up paper \cite{abbondandolo_portaluri_schwarz}, seem to apply to the approximately quadratic system $Q_t$ defined in \S\ref{sec:comp-with-quadr}, and should yield the isomorphism:
\begin{equation*}
  H_{*}(\mathscr{P},N;\mathscr{L})\simeq \mathrm{HM}(\mathscr{P},N;\mathscr{E};\mathscr{L})\simeq \mathrm{H}(\mathrm{CF}(Q_{t},J)/\mathrm{CF}(H_{1},J))\simeq \mathrm{HW}_{+}(\nu^{*}N;\mathscr{L}),
\end{equation*}
proving Theorem \ref{theorem:main-iso}. The rest of this section is devoted to a careful proof of this isomorphism in our framework, using the particular quadratic system $Q_{t}$ constructed in \S\ref{sec:comp-with-quadr}. The final subsections \S\ref{sec:adapt-compl-struct}--\ref{sec:regul-moving-lagr} are devoted to some analysis of the solutions to $\mathscr{M}^{\Theta}(Q_{t},J)$ near the corners of the domain, and some discussion of the moving Lagrangian boundary conditions along the $s=0$ boundary.

\subsubsection{Energy bounds for a quadratic Hamiltonian}
\label{sec:energy-bounds-quad}

Suppose $Q_{t}$ is the perturbation of the quadratic system as in \S\ref{sec:comp-with-quadr}. Using that $Q_{t}\geq \frac{1}{2}r^2$, we estimate:
\begin{equation*}
  \begin{aligned}
    E(u)&=\mathscr{A}(\gamma_{-})+\textstyle\int\ip{p(t),q'(t)}\d t-\textstyle\int Q_{t}(q(t),p(t))\,\d t\\
        &\le \mathscr{A}(\gamma_{-})-\textstyle\frac{1}{2}\norm{p-g_{\flat}(q')}_{L^{2}}^{2}+\textstyle\frac{1}{2}\norm{q'}_{L^{2}}^{2}\\
        &=\textstyle\frac{1}{2}(\norm{q'}_{L^{2}}^{2}-r_{-}^{2} - \norm{p-g_{\flat}(q')}_{L^{2}}^{2}).
  \end{aligned}
\end{equation*}
Here $g_{\flat}$ signifies the isomorphism between $TM$ and $T^{*}M$ induced by the Riemannian metric $g$. In the final line we have used that every orbit $\gamma_{-}$ of $Q_{t}$ lies in a fixed level set $r=r_{-}$ and its action is given by $\mathscr{A}(\gamma_{-})=-\frac{1}{2}r_{-}^2$.

In particular, an $L^{2}$ bound on $q'$ ensures an a priori bound on the energy of $u$, an a priori $L^{2}$ bound on $p$, and an a priori bound on $r_{-}^{2}$. By definition of pseudogradient, the $L^{2}$ size of $q'$ is at most the $L^{2}$ size of $x'$, and hence, for each element $x$ in $\mathrm{CM}_{\ell}$ all moduli spaces needed to compute $\Theta(x)$ satisfy an a priori energy bound. In this case, the results in \S\ref{sec:priori-w1-p} imply $\Theta(x)$ is a finite sum in $\mathrm{CF}(Q_{t})$.

\subsubsection{Chain map property}
\label{sec:chain-map-property}

Let $\mathrm{CM}_{\ell}$ be as in \S\ref{sec:morse-homol-energy}. We emphasize that $\Theta$ will not directly define a chain map $\mathrm{CM}_{\ell}\to \mathrm{CF}(Q_{t},J)$, since $\mathrm{CM}_{\ell}$ only considers generators of positive length. However, if we let $\mathrm{CF}_{+}(Q,J):=\mathrm{CF}(Q,J)/\mathrm{CF}(H_{1},J)$, then:

\begin{prop}
  The induced map:
  \begin{equation*}
    \Theta:\mathrm{CM}_{\ell}\to \mathrm{CF}_{+}(Q_{t},J)
  \end{equation*}
  is a chain map, assuming that the construction of $Q_t$ is sufficiently generic so that the relevant moduli spaces are cut out transversally.
\end{prop}
\begin{proof}
  The proof is similar to the proof of Lemma \ref{lemma:differential}.

  Let $\mathscr{M}^{\Theta}(\gamma,x)$ be the subset of $\mathscr{M}^{\Theta}$ consisting of those $(u,q)$ so that $u$ is asymptotic to the Hamiltonian chord $\gamma$ and so that $q$ lies in the unstable manifold of $x$.

  Since we consider the projection of $\Theta$ to the quotient $\mathrm{CF}_{+}(Q_{t},J)$, we only need to consider $0$ and $1$ dimensional components of $\mathscr{M}^{\Theta}(\gamma,x)$ where $r(\gamma(t))\ge x_{1}$; see \S\ref{sec:spec-cofibr-sequ}. In particular, any sequence $(u_{n},q_{n})$ in such a component has $\norm{q_{n}'}_{L^{2}}$ bounded from below by a positive number, otherwise the action of $\gamma$ would be too close to zero by the energy estimate in \S\ref{sec:energy-bounds-quad}.

  Let us consider a 1-dimensional component $C$ of $\mathscr{M}^{\Theta}(\gamma, x)$. There are two possible types of breaking at the non-compact ends of $C$: the first one involves the breaking of a Morse trajectory, and the second involves the breaking of a Floer trajectory. A priori, $C$ can join the first type with the first type, the second type with the second type, or the first type with the second type. The crucial case is when $C$ joins different types, namely $\bd C = \{( (u_1, q_1), w_1), (v_2, (u_2, q_2) ) \}$ where $( (u_1, q_1), w_1) \in  \mathscr{M}^{\Theta}(\gamma, y) \times \Bar{\mathscr{M}}(y, x)$ and $(v_2, (u_2, q_2) ) \in \Bar{\mathscr{M}}(\gamma, \gamma') \times \mathscr{M}^{\Theta}(\gamma', x)$.

  From \S\ref{sec:param-moduli-space} we get $\mathfrak{o}(\mathscr{M}^{\Theta}(\gamma, x)) \simeq \mathfrak{o}_{\gamma} \otimes \mathfrak{o}_{x}$. This identification, and a choice of orientation of $C$, induces an identification $\Theta_{(u,q)}:\mathfrak{o}_{x} \to \mathfrak{o}_{\gamma}$. The idea is to compare $\Theta_{(u,q)}$ with  $d_{v_{2}} \circ \Theta_{(u_{2}, q_{2})}$ and $\Theta_{(u_{1}, q_{1})} \circ \bd_{w_{1}}$.

  Applying \S\ref{sec:bound-param-moduli} we get that $\Theta_{(u,q)} =  \Theta_{(u_{1}, q_{1})} \circ \bd_{w_{1}}$ if $\partial_{s} w_{1}$ is positively\footnote{Equivalently, $(u_{1}, q_{1})$ is a positive boundary point.} oriented in $\mathscr{M}^{\Theta}(\gamma, x)$ and $\Theta_{(u,q)} = - \Theta_{(u_{1}, q_{1})} \circ \bd_{w_{1}}$ if $\partial_{s} w_{1}$ is negatively oriented. On the other hand, by \S\ref{sec:inter-break-param}, $\Theta_{(u,q)} = d_{v_{2}} \circ \Theta_{(u_{2}, q_{2})}$ if $-\eta_{u_{2}}$ is positively oriented in $\mathscr{M}^{\Theta}(\gamma, x)$, and $\Theta_{(u,q)} = - d_{v_{2}} \circ \Theta_{(u_{2}, q_{2})}$ otherwise. See Figure \ref{fig:chainmap} for an illustration.

  If $C$ joins the same types of the breaking, by analogous arguments as in Lemma \ref{lemma:differential} we get that $\Theta_{(u_{1},q_{1})} \circ \bd_{w_{1}} + \Theta_{(u_{2},q_{2})} \circ \bd_{w_{2}} = 0$ and $d_{v_{1}} \circ \Theta_{(u_{1},q_{1})} + d_{v_{2}} \circ \Theta_{(u_{2}, q_{2})}=0$.

  For similar arguments see \cite[\S12 Lemma 3.7]{abouzaid_monograph}, \cite[\S3.2]{abbondandolo_schwarz}, \cite[\S3]{abbondandolo_schwarz_corrigendum}.
\end{proof}

\subsubsection{Compatibility with continuation lines}
\label{sec:comp-with-cont}

The proof is similar to \S\ref{sec:chain-map-property}, and here we just sketch the argument; see \cite[\S XII.3.4]{abouzaid_monograph} and also \cite[Theorem 7]{katic_milinkovic_pss} for similar results.

Let $V,V'$ be two pseudogradients on finite dimensional approximations to $\mathscr{P}_{\ell}$ and consider the continuation line equation from $V$ to $V'$ as in \S\ref{sec:continuation-lines}; see Figure \ref{fig:continuation-lines}.

Consider the parametric moduli space $\mathscr{M}(\gamma,x)$ of triples $(u,q,\tau)$ where $(u,q)$ solves \eqref{eq:theta-moduli-space} with asymptotic $\gamma$, and $q$ is the right endpoint of a solution $\xi(s)$ of the continuation line ODE defined on domain $(-\infty,\tau]$ asymptotic to $x$ at the negative end.

One considers the 1-dimensional component $\mathscr{M}_{1}(\gamma,x)\subset \mathscr{M}(\gamma,x)$ and the projection $\mathrm{pr}:\mathscr{M}_{1}(\gamma,x)\to \R$ given by $(u,q,\tau)\mapsto \tau$. For generic $\tau$, the fiber: $$\mathscr{M}_{0}(\tau;\gamma,x)=\mathrm{pr}^{-1}(\tau)$$ is a zero-dimensional manifold, the signed count of which determines the $x\mapsto \gamma$ coefficient of a map $\mathrm{CM}_{\ell}\to \mathrm{CF}_{+}$.

Similarly to \S\ref{sec:chain-map-property}, summing over all $x$ defines a chain map $\mathfrak{T}_{\tau}:\mathrm{CM}_{\ell}\to \mathrm{CF}_{+}$ depending on $\tau$. Standard Floer theoretic arguments imply that different choices of $\tau$ give chain homotopic maps. For $\tau<0$, $\mathfrak{T}_{\tau}$ equals the $\Theta$ map for the pseudogradient $V$. As $\tau\to \infty$, $\mathfrak{T}_{\tau}$ eventually agrees with the composition of the Morse continuation map $\mathfrak{c}$ with the $\Theta$ map for the pseudogradient $V'$, as desired.

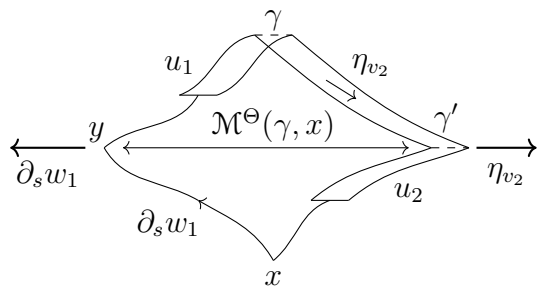
\begin{figure}[H]
  \centering
  \begin{tikzpicture}[]
    \draw[<->] (0,0) -- (3.8,0);
    \draw (1.75, 1.5) to [out=320, in=160] (4.1, 0) (2.25, 1.5) to [out=320, in=160] (4.6, 0);
    \draw (2.25, 1.5) to [out=190, in=20] (1.25, 0.7) (1.75, 1.5) to [out=190, in=20] (0.75, 0.7);
    \draw[<-] (3.1, 0.65) -- (2.7,0.9);
    \draw[dashed] (1.75, 1.5) -- (2.25 ,1.5) (4.1,0) -- (4.6,0);
    \draw (1, 0.7) to [out=240, in=40] (-0.25, 0) (2,-1.5) to [out=40, in = 200] (2.75, -0.7);
    \draw (-0.25, 0) to [out=300, in=150] (1,-0.7);
    \draw[<-](1,-0.7) to [out=330, in=120] (2,-1.5);
    \draw (4.1, 0) to [out=200, in=40] (2.5, -0.7) (4.6, 0) to [out=200, in=40] (3, -0.7);
    \draw (0.75, 0.7) -- (1.25, 0.7) (2.5, -0.7) -- (3, -0.7);
    \node at (2,0.3) {$\mathscr{M}^{\Theta}(\gamma, x)$};
    \node at (-0.35, 0.2) {$y$};
    \node at (2, 1.7) {$\gamma$};
    \node at (4.3, 0.4) {$\gamma'$};
    \draw[line width = 0.8 pt, ->](4.7, 0) -- (5.5,0);
    \draw[line width = 0.8 pt, <-]   (-1.5 , 0) -- (-0.5,0);
    \node at (5.1,-0.35) {$\eta_{v_{2}}$};
    \node at (-1,-0.35) {$\bd_{s} w_{1}$};
    \node at (0.6, -1) {$\bd_{s} w_{1}$};
    \node at (0.75, 1.1) {$u_1$};
    \node[below] at (2, -1.5) {$x$};
    \node at (3.8, -0.6) {$u_2$};
    \node at (3.3, 1.1) {$\eta_{v_{2}}$};
  \end{tikzpicture}
  \caption{Both $\bd_{s} w_{1}$ and $\eta_{u_2}$ point outward near the boundary of $C$.}
  \label{fig:chainmap}
\end{figure}

\subsubsection{The diagonal argument}
\label{sec:diagonal-argument-theta}

The energy estimate in \S\ref{sec:energy-bounds-quad} is used to prove the following lemma based on the diagonal argument from \cite{abbondandolo_schwarz}.
\begin{lemma}\label{lemma:technical-diagonal}
  The map $\Theta:\mathrm{CM}_{\ell}\to \mathrm{CF}_{+}(Q_{t},J)$ factors through the cofibration $\mathrm{CF}_{n(\ell)}/\mathrm{CF}_{0}\to \mathrm{CF}_{+}(Q_{t},J)$ where: $$n(\ell)=\mathrm{min}\set{n:x_{n+1}>\ell}.$$ Moreover, the induced map $\mathrm{CM}_{\ell}\to \mathrm{CF}_{n(\ell)}/\mathrm{CF}_{0}$ is a quasi-isomorphism.
\end{lemma}
\begin{proof}
  Since $E(u) \leq \frac{1}{2}(\norm{q'}_{L^2}^2 - r_{-}^2) \leq \frac{1}{2}(\ell^2 - r_{-}^2)$, the asymptotic chord $\gamma_{-}$ of $u$ must satisfy $r_{-} \leq \ell$, which proves that the map $\Theta$ factors through $\mathrm{CF}_{n(\ell)}/\mathrm{CF}_{0}$. To show that the induced map is a quasi-isomorphism, order the bases of $\mathrm{CM}_{\ell}$ and $\mathrm{CF}_{n(\ell)}/\mathrm{CF}_{0}$ by their spectral value $\sigma_{k}$. Given a generator $x \in \mathrm{CM}_{\ell}$, the stationary solution $u(s,t)=\gamma_{x}(t)=(x(t), g_{\sharp}(x'(t))$ satisfies the equation for $\Theta$ map; this implies, together with the energy estimate, that:
  $$
  \Theta \vert_{\mathfrak{o}_{x}\otimes  \mathscr{L}_{x}}= \varphi_x \otimes  \mathscr{L}(x, \gamma_{x})+ \sum_{\gamma_{-} , r_{-} < \norm{x'}_{L^{2}} } \sum_{u \in \mathscr{M}^{\Theta}(x, \gamma_-)} \Theta_{(u,x)}\otimes  \mathscr{L}(u,x),
  $$
  where $\varphi_x: \mathfrak{o}_{x} \to \mathfrak{o}_{\gamma_{x}}$ is an isomorphism between free rank-one $\Z$-modules (unique up to a sign), and the map $\mathscr{L}(x, \gamma_{x}):  \mathscr{L}_{x} \to  \mathscr{L}_{\gamma_{x}}$ is the canonical isomorphism. The corresponding matrix is upper triangular with isomorphisms on the diagonal, and hence it is a quasi-isomorphism. For the regularity of stationary solutions see \cite[pp.~259]{abbondandolo_schwarz}.
\end{proof}

\subsubsection{Conclusion of the argument}
\label{sec:concl-arg}

Pick an increasing sequence $\ell_{k} \to \infty$ of regular values of $\mathscr{E}$, and a sequence $\mathscr{P}^{\ell_{k}}_{K_{k}}$ of admissible finite dimensional approximation together with pseudo-gradients $V_{k}$; see \ref{sec:admiss-pseud}. For each $j<k$ pick a continuation data $\delta_{j,k}$ which induces continuation maps $\mathfrak{c}_{j,k}:\mathrm{HM}_{\ell_{j}}(V_{j}) \to \mathrm{HM}_{\ell_{k}} (V_{k})$ as in \ref{sec:continuation-lines}.

Applying Lemma \ref{lemma:technical-diagonal} we have that $\Theta_{k}: \mathrm{HM}_{\ell_{k}} \to \mathrm{H} ( \mathrm{CF}_{n(\ell_{k})}/ \mathrm{CF_{0}} )$ is an isomorphism, and moreover  \S\ref{sec:comp-with-cont} implies that the following square commutes:
\begin{equation}\label{eq:Theta_square}
  \begin{tikzcd}
    {\mathrm{HM}_{\ell_{j}}}\arrow[d,"{}"]\arrow[r,"{}"] &{\mathrm{H}( \mathrm{CF}_{n(\ell_{j})}/ \mathrm{CF_{0}} )}\arrow[d,"{}"]\\
    {\mathrm{HM}_{\ell_{k}}}\arrow[r,"{}"] &{\mathrm{H}( \mathrm{CF}_{n(\ell_{k})}/ \mathrm{CF_{0}} )}.
  \end{tikzcd}
\end{equation}
One can achieve the transversality needed for each step of the argument (using the perturbations in the definition of $Q_{t}$) since there are only countably many steps.

Passing to the direct limit and using \S\ref{sec:comp-with-sing} yields:
$$
H_{*}(\mathscr{P}, N; \mathscr{L}) \to \mathrm{HW}_{+}(\nu^*N; \mathscr{L}),
$$
completing the proof of Theorem \ref{theorem:main-iso}.

The remaining sections are concerned with various technicalities implicitly used in the preceding arguments.
\subsubsection{Adapted complex structure and conormal boundary conditions}
\label{sec:adapt-compl-struct}
An admissible complex structure $J$ is said to be \emph{adapted} to $N$ provided that, in addition to (a) and (b) from \S\ref{sec:admissible-data}, $J$ further satisfies:
\begin{enumerate}
\item[(c)] there is a smooth locally-defined involution $\mathfrak{r}$ whose fixed point set is $N$, so that the canonical extension $\mathfrak{R}$ to $T^{*}M$ preserves $J$ in a neighborhood of $\nu^{*}N$, and,
\item[(d)] the involution $\mathfrak{i}:(q,p)\mapsto (q,-p)$ is anti-complex in a neighborhood of $\nu^{*}N$.
\end{enumerate}

\begin{lemma}
  If $\alpha$ is a contact form on the ideal boundary $ST^{*}M$ which is invariant under the involution $\mathfrak{R}$ in (c) and anti-invariant under $\mathfrak{i}$ in (d), in a neighborhood of the ideal Legendrian boundary of $\nu^{*}N$, then there is a SFT type almost complex structure $J$ for $\alpha$ which is adapted to $N$, for the same $\mathfrak{R}$.
\end{lemma}
\begin{proof}
  Fix some involution $\mathfrak{r}$ once and for all. It is convenient to use a metric for which $\mathfrak{r}$ acts by isometries. The exponential map for such a metric gives coordinate charts $q_{1},\dots,q_{n}$ so that $N$ is identified with the plane $\Pi_{r}=\set{q_{r+1}=\dots=q_{n}=0}$ and $\mathfrak{r}$ is identified with the reflection through $\Pi_{r}$; in particular, the charts are supposed to be equivariant with respect to $\mathfrak{r}$.

  The standard $J_{0}$ in canonical coordinates satisfies (c) and (d), although is not of SFT type. One can use $J_{0}$ to define a linear coordinate chart for the space of $\omega$-tame almost complex structures. In such a local chart, the space of $J$ which satisfy (c) and (d) is identified with a convex open subset in some auxiliary vector space.

  Cover $N$ by such coordinate cubes, say $U_{1},\dots,U_{k}$, and suppose that the slightly shrunken cubes $e^{-\epsilon}U_{j}$ still cover $N$. We will construct a complex structure $J$ on $T^{*}U_{1},\dots,T^{*}U_{j}$ by induction on $j$ so that $J$ satisfies (c) and (d). In the inductive step, we can consider $J$ as defining a complex structure on $T^{*}(U_{j+1}\cap (U_{1}\cup \dots \cup U_{j}))$. By standard extension lemmas, one can extend $J$ to $T^{*}U_{j+1}$ keeping it unchanged on the shrunken cubes $e^{-\epsilon}U_{1}\cup \dots \cup e^{-\epsilon}U_{j}$. Indeed, in the convex space of complex structures satisfying (c) and (d), one can consider a path joining $J$ to $J_{0}$, say $J_{\tau}$. Fix $\mathfrak{r}$-invariant cut-off function $\beta$ which is $1$ on $e^{-\epsilon}U_{1}\cup \dots \cup e^{-\epsilon}U_{j}$ and supported in $U_{1}\cup \dots \cup U_{j}$. Then $J_{\beta}$ is a complex structure which satisfies (c) and (d). This finishes the inductive step.

  The next stage of the construction is to work on the ideal boundary. We will re-use the same cubes $U_{1},\dots,U_{k}$. Note that $q_{1},\dots,q_{n}$ and $\theta_{j}=p_{j}/\abs{p}$ form a coordinate system for $ST^{*}U$ so that the contact form is $\alpha=\sum \theta_{j}\d q_{j}$. The contact distribution is the orthogonal complement to $(q_{1},\dots,q_{n})$ times the tangent space to $S^{n-1}$ at $(\theta_{1},\dots,\theta_{n})$, with respect to the splitting $\R^{n}\times S^{n-1}$. These are the same space, and this identification induces a ``standard'' almost complex structure $J_{0}$ on $\xi$. The formula is simply the identity map from the $\theta$ coordinates to the $q$ coordinates, so that $\d\alpha(-,J_{0}-)=\sum \d\theta_{j}^{\otimes 2}+\d q_{j}^{\otimes 2}$. In particular, $J_{0}$ is compatible with the contact structure (recall that the symplectic structure on $\xi$ is well-defined up to conformal equivalence and does not depend on the form).

  As above, the space of $\d\alpha$-compatible complex structures on $\xi$ which satisfy (c) and (d) in each chart is convex, and one can inductively construct some complex structure $J_{\xi}$ on the contact distribution of $ST^{*}(U_{1}\cup \dots \cup U_{k})$ which satisfies (c) and (d). Since the space of complex structure is convex, we can extend $J_{\xi}$ to all of $ST^{*}M$.

  We can extend $J_{\xi}$ from $\xi=\ker\lambda\cap \ker \d r$ by defining $J_{\alpha}(Z)=R$. By definition, $J_{\alpha}$ is of SFT type. Since $\alpha$ is invariant and anti-invariant under (c) and (d), respectively, this extension satisfies (c) and (d).

  Let $J$ be the complex structure constructed on $T^{*}M$ which was not of SFT type but which satisfied (c) and (d) in a neighborhood of $\nu^{*}N$. Since the space of $\omega$-tame complex structures satisfying (c) and (d) is convex, we can interpolate from $J$ on the compact part to $J_{\alpha}$ in the symplectization end. This completes the construction.
\end{proof}

\subsubsection{Doubling and regularity at corners}
\label{sec:doubl-regul-at}

Let $J$ be an adapted almost complex structure, and suppose that $u$ is a $W^{1,p}_{\mathrm{loc}}$ solution\footnote{Throughout we assume $p>2$ when considering Sobolev spaces.} of the boundary value problem:
\begin{equation*}
  \left\{
    \begin{aligned}
      &u:(-1,0]\times [0,1]\to T^{*}M,\\
      &\bd_{s}u+J(u)(\bd_{t}u-X_{t}(u))=0,\\
      &u(s,0),u(s,1)\in \nu^{*}N\text{ and }u(0,t)\in L_{t},
    \end{aligned}
  \right.
\end{equation*}
where $L_{t}=T^{*}M_{q(t)}$ (some piecewise smooth path of fibers). Assume that $X_{t}$ is the Hamiltonian vector field for the quadratic system $Q_{t}$. Suppose that the Riemannian metric $g$ used to define the radial function $r$ is such that the local involution $\mathfrak{R}$ acts by isometries; this can be achieved by a small perturbation near $N$. We will now explain how to double $u$ across the boundary.

Focus on the corner at $0$, and let $S(\delta)=D(\delta)\cap (-1,0]\times [0,1]$. For sufficiently small $\delta$, $u(S(\delta))$ remains in the neighborhood $U$ of $\nu^{*}N$ where the involution $\mathfrak{R}$ is defined (since $u$ is of class $W^{1,p}_{\mathrm{loc}}$). Let $\mathfrak{a}:U\to U$ be the anti-symplectic involution defined by $\mathfrak{a}=\mathfrak{i}\mathfrak{R}$. Then $J\circ \d \mathfrak{a}=-\d\mathfrak{a}\circ J$. Analyzing how $\mathfrak{r}$ acts on $TM|_{N}$, one sees that $\mathfrak{R}(q,p)=(q,-p)$ for $(q,p)\in \nu^{*}N$, and thus $\mathfrak{a}$ acts identically on $\nu^{*}N$.

Let $\Sigma(\delta)=S(\delta)\cup \bar{S}(\delta)$, i.e., $\Sigma(\delta)=(-1,0]\times [-1,1]\cap D(\delta)$, as shown in Figure \ref{fig:double_corner}.
\begin{figure}[H]
  \centering
  \begin{tikzpicture}[xscale=-1]
    \begin{scope}
      \clip (0,-0.4) rectangle (1,1);
      \draw[pattern={Lines[angle=45,distance=2pt]},pattern color=black!50!white] (0,0) circle (0.4);
    \end{scope}
    \draw (0,-0.4)--(0,0)node[right]{$\Sigma(\delta)$} rectangle (1,1);
  \end{tikzpicture}
  \caption{Doubling the neighborhood of the corner.}
  \label{fig:double_corner}
\end{figure}
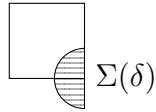

Consider the piecewise function $w:\Sigma(\delta)\to T^{*}M$ given by:
\begin{equation*}
  w(s,t)=\left\{
    \begin{aligned}
      u(s,t)&\text{ for }t\ge 0,\\
      \mathfrak{a}u(s,-t)&\text{ for }t\le 0.\\
    \end{aligned}
  \right.
\end{equation*}
Using that $J\d\mathfrak{a}=-\d\mathfrak{a}J$ we conclude that:
\begin{equation*}
  \bd_{s}w+J(w)\bd_{t}w=-J(w)\d\mathfrak{a}X_{-t}\mathfrak{a}(w)
\end{equation*}
holds for $t\le 0$. Since $X_{t}$ is the Hamiltonian vector field for $\frac{1}{2}r^{2}$ near $t=0$, and $r\circ \mathfrak{a}=r$ and $\mathfrak{a}$ is anti-symplectic, we conclude that $\d\mathfrak{a}X_{t}\mathfrak{a}=-X_{-t}$ holds near $t=0$. In particular, $w$ solves $\bd_{s}w+J(w)\bd_{t}w=J(w)X_{t}$ on all of $\Sigma(\delta)$.

Note that $w(0,t)\in L_{-t}$ for $t\le 0$; thus the moving Lagrangian boundary conditions might have a jump (discontinuity in $q'$) at $t=0$.

\subsubsection{Local elliptic estimate for moving Lagrangian boundary conditions}
\label{sec:regul-moving-lagr}

Let $\mathbb{H}$ be the closed left half plane and $\Omega(\delta)=D(\delta)\cap \mathbb{H}$. Consider $u:\Omega(1)\to \R^{2n}$ so that:
\begin{equation}\label{eq:linear-moving-lags}
  \left\{
    \begin{aligned}
      \bd_{s}u+J_{s,t}\bd_{t}u=A_{s,t},\\
      u(0,t)\in \set{q(s)}\times \R^{n},
    \end{aligned}
  \right.
\end{equation}
where $q:\R\to \R^{n}$ is a $W^{1,p}_{\mathrm{loc}}$ path, $A_{s,t}$ is some $L^{p}_{\mathrm{loc}}$ section, and $J_{s,t}$ is a $W^{1,p}$ family of almost complex structures so $\set{0}\times \R^{n}$ is totally real for every $s,t$.

\begin{prop}
  There are constants $C$ depending on $\norm{J}_{W^{1,p}}$ so that the following estimate holds:
  \begin{equation*}
    \norm{u}_{W^{1,p}(\Omega(1/3))}\le C(\norm{A}_{L^{p}(\Omega(2/3))}+\norm{q}_{W^{1,p}(\Omega(2/3))}+\norm{u}_{L^{\infty}(\Omega(2/3))}).
  \end{equation*}
  for all $u\in W^{1,p}_{\mathrm{loc}}$ which solve \eqref{eq:linear-moving-lags}.
\end{prop}
\begin{proof}
  Pick a traveling frame $X_{1},\dots,X_{n},Y_{1},\dots,Y_{n}$ on $\mathbb{H}$ so $X_{j}$ is the $j$th basis vector for $\set{0}\times \R^{n}$, and $Y_{j}=J_{s,t}X_{j}$. Write $u(s,t)$ as $q(s)+\sum a_{j}X_{j}+b_{j}Y_{j}$, and note that it suffices to bound the $W^{1,p}$ sizes of $w=(a_{1}+ib_{1},\dots,a_{n}+ib_{n})$; one requires $p>2$ in order to have the $W^{1,p}$ quadratic estimates.

  Observe that $w$ solves the equation:
  \begin{equation*}
    \left\{
      \begin{aligned}
        &\bd_{s}w+J_{0}\bd_{t}w=A'_{s,t}+B_{s,t}\cdot w,\\
        &w(s,0)\in \R^{n}\times \set{0},
      \end{aligned}
    \right.
  \end{equation*}
  where $\norm{A'}_{L^{p}}$ can be bounded by $C(\norm{A}_{L^{p}}+\norm{q'}_{L^{p}})$ and the $L^{p}$ size of $B_{s,t}$ is bounded by the $W^{1,p}$ size of $J$. Then we estimate:
  \begin{equation*}
    \norm{B_{s,t}\cdot w}_{L^{p}} \le \norm{B_{s,t}}_{L^{p}}\norm{w}_{L^{\infty}}.
  \end{equation*}
  Apply \cite[Lemma B.4.6]{mcduffsalamon} with $r=p$ to conclude the desired result.
\end{proof}

\subsubsection{$W^{1,p}$ a priori estimate}
\label{sec:a-priori-estimate-sobolev}

Let $u_{n}$ be a sequence of $W^{1,p}_{\mathrm{loc}}$ elements in the $\mathscr{M}^{\Theta}(Q_{t})$ moduli space so that $u(0,t)\in T^{*}M_{q_{n}(t)}$ and $\norm{q'_{n}}_{L^{p}}$ is bounded. Let $p_{n}(t)=u(0,t)$, considered as an element in $T^{*}M_{q_{n}(t)}$. The a priori estimate in \S\ref{sec:energy-bounds-quad} implies that $\norm{p_{n}}_{L^{2}}$ and $E(u_{n})$ are bounded. The goal in this section is to prove that $u_{n}$ satisfies an a priori $W^{1,p}$ bound, in the sense that:
\begin{equation}\label{eq:apriori-W1P}
  \sup_{n}\int_{\Sigma} \abs{\bd_{s}u_{n}}^{p}+\abs{\bd_{t}u_{n}}^{p}\d s\d t< \infty,
\end{equation}
where the norm is measured using a translation invariant metric. This $W^{1,p}$ estimate implies, among other things, that $u_{n}$ satisfies an a priori $C^{0}$ bound (i.e., the maximum principle).

Let $\Sigma=(-\infty,0]\times [0,1]\subset \C$. To begin, pick $\delta_{n}$ so that, after rescaling the domain of $u_{n}$ by setting $w_{n}(s,t)=u_{n}(\delta_{n}s,\delta_{n}t)$ we have:
\begin{equation*}
  \max_{z\in \delta_{n}^{-1}\Sigma}\norm{\d w_{n}}_{L^{p}(S(z))}=1,
\end{equation*}
where $S(z)$ is the square of side length $1$ centered at $z$. Standard asymptotic analysis at the infinite end (using the a priori energy bound) implies that there is some maximizer $z_{n}$.

First, if $\delta_{n}$ does not converge to $0$ along some subsequence, then $u_{n}$ is bounded in $W^{1,p}$ on any square of side length $1$. Standard elliptic bootstrapping implies that $u_{n}$ satisfies an a priori $C^{1}$ bound on $\Sigma(1)=(-\infty,-1]\times [0,1]$, say by $C$. Then we conclude that: $$\int_{\Sigma(1)} \abs{\bd_{s}u_{n}}^{p}+\abs{\bd_{t}u_{n}}^{p}\d s \d t\le C^{p-2}\norm{\d u}_{L^{2}}^{2}.$$ Since the $W^{1,p}$ size is bounded on $\Sigma\setminus \Sigma(1)=[-1,0]\times [0,1]$ we conclude the desired result.

Henceforth assume $\delta_{n}\to 0$ along some subsequence. We will derive a contradiction using bubbling analysis. One key step needed to deal with the quadratic system $Q_{t}$ is that one can estimate how small $\delta_{n}$ is using the $L^{2}$ bound on $p_{n}$. It will be important to estimate $r_{n}:=\max r(u_{n}(s,t))$.

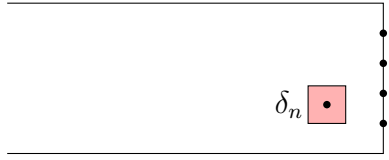
\begin{figure}[H]
  \centering
  \begin{tikzpicture}
    \draw (0,0)coordinate(A)--(5,0)coordinate(X)--+(0,2)coordinate(Y)--(0,2)coordinate(Z);
    \path[every node/.style={fill,circle,inner sep=1pt}] (X)--(Y) node[pos=0.2] {} node[pos=0.4] {} node[pos=0.6]{} node[pos=0.8]{};
    \draw[fill=white!70!red] (4,0.4) rectangle coordinate(X) +(0.5,0.5);
    \node[fill, circle, inner sep=1pt]at(X){};
    \node[shift={(-0.5,0)}] at (X) {$\delta_{n}$};
  \end{tikzpicture}
  \caption{The domain of $u_n$. The shaded square is centered at $\delta_{n}z_{n}$ and has side length $\delta_{n}$. After expanding $u_{n}$ to $w_{n}$, the square has side length $1$ and the $L^{p}$ size of $\d w_{n}$ on the square is exactly $1$.}
  \label{fig:domain-of-un}
\end{figure}

By the construction of $Q_{t}$, we can apply the strong maximum principle of \S\ref{sec:strong-maxim-princ} to conclude that $r_{n}\le \max_{t}\abs{p_{n}(t)}+1,$ where the $1$ should be thought of as an error term. Let $\max \abs{p_{n}(t)}=\abs{p_{n}(t_{n})}=e^{2\sigma_{n}}$.

There is maximal $\epsilon_{n}>0$ so that the interval $[t_{n}-\epsilon_{n},t_{n}+\epsilon_{n}]$ is mapped by $\abs{p_{n}}$ into $[e^{\sigma_{n}},e^{2\sigma_{n}}]$. In particular, we can estimate $\epsilon_{n}e^{2\sigma_{n}}\le \norm{p_{n}}_{L^{2}}^{2}$.

First consider the case when $\epsilon_{n}/\delta_{n}$ is bounded from above. The rescaling $w_{n}$ has an interval $I_{n}=\delta_{n}^{-1}[t_{n}-\epsilon_{n},t_{n}+\epsilon_{n}]$ in its boundary so that $w_{n}(I_{n})=u_{n}(\delta_{n}I_{n})$ has diameter at least $\sigma_{n}$ (noting that $\epsilon_{n}\to 0$ implies at least one endpoint $t_{n}\pm\epsilon_{n}$ is mapped by $\abs{p_{n}}$ to $e^{\sigma}$).

The Sobolev embedding theorem implies that the diameter of $w_{n}(I_{n})$ is bounded in terms of the $W^{1,p}$ size of $w_{n}$ on some union of squares of height $1$ covering $I_{n}$. Since there is an a priori finite number of squares needed to cover $I_{n}$ (as $\epsilon_{n}/\delta_{n}$ is bounded above), the diameter of $w_{n}(I_{n})$ is bounded. Thus the maximum $\sigma_{n}$ is bounded from above and consequently $\abs{X_{Q_{t}}(u_{n})}$ is also bounded. In this case, observe that the equation for the rescaling:
\begin{equation}\label{eq:rescaled-equation}
  \bd_{s}w_{n}+J(w_{n})\bd_{t}w_{n}=\delta_{n}J(w_{n})X_{Q_{t}}(w_{n})
\end{equation}
converges to the standard holomorphic curve equation $\bd_{s}w+J(w)\bd_{t}w$. Compactness results for solutions to perturbations of the holomorphic curve equations satisfying a priori $W^{1,p}$ bounds imply $w_{n}$ converges in $W^{1,p}_{\mathrm{loc}}$ to some holomorphic map; see, e.g., \cite[\S2.11]{wenc} or \S\ref{sec:priori-w1-p} below.

By recentering $w_{n}$ so that the maximizing point $z_{n}$ is located at the origin, one concludes that $w_{n}$ converges to a non-constant holomorphic map $w$. During the rescaling and recentering process, the domain of $w_{n}$ is some half-infinite rectangle in $\C$ which intersects the origin. Such a sequence of objects converges to either $\C$, a half-plane, or a ``quadrant.'' In all cases where there are boundary, the limiting map $w$ has fixed Lagrangian boundary conditions on either a fiber $T^{*}M_{q}$ or the conormal $\nu^{*}N$ (or both, in the case of a quadrant). Since these Lagrangians are strongly exact, i.e., satisfy $\lambda|_{L}=0$, we conclude that the limiting map has zero energy, and hence must be constant. This contradicts our earlier deduction, and hence we conclude that, if $\delta_{n}\to 0$, then $\epsilon_{n}/\delta_{n}$ must be unbounded.

The case when $\epsilon_{n}/\delta_{n}\to +\infty$ uses our estimate $\epsilon_{n}e^{2\sigma_{n}}\le \norm{p_{n}}^{2}_{L^{2}}$. Since the $L^{2}$ norm of $p_{n}$ is a priori bounded, we conclude that $\delta_{n}e^{2\sigma_{n}}$ converges to zero as $n\to\infty$. Recall that $r_{n}\le e^{2\sigma_{n}}+1$, and hence $\delta_{n}r_{n}$ also converges to zero. Since $\abs{X_{Q_{t}}}\le cr+b$, the equation \eqref{eq:rescaled-equation} converges to the standard holomorphic curve equation. If $w_{n}(0)$ remains bounded, then $w_{n}$ converges to a limiting holomorphic plane, half-plane, or corner, producing a contradiction as above. If $w_{n}(0)$ is unbounded, then it eventually enters the symplectization end. After translation by the Liouville flow we may suppose $w_{n}(0)$ converges in the symplectization end. The local $W^{1,p}$ bounds imply $w_{n}$ converges on compact subsets to a holomorphic object, again yielding a contradiction.

Thus $\delta_{n}$ must be bounded from below and $u_{n}$ satisfies an a priori $W^{1,p}$ bound.

\subsubsection{A priori $W^{1,p}$ bound implies $W^{1,p}_{\mathrm{loc}}$ compactness}
\label{sec:priori-w1-p}

In this section we prove that any sequence $u_{n}\in \mathscr{M}^{\Theta}(Q_{t})$ so that $\norm{q_{n}'(t)}_{L^{2}}$ is bounded has a subsequence which converges in $W_{\mathrm{loc}}^{1,p}$. Such a compactness result is implicit in the Floer theoretic arguments in \S\ref{sec:chain-map-property} and \S\ref{sec:diagonal-argument-theta}.

First, it follows from the energy estimate in \S\ref{sec:energy-bounds-quad} that $r_{-}^{2}$ is bounded along the sequence $u_{n}$. The strong maximum principle then implies that the images of $u_{n}$ remain in a fixed compact set.

By \S\ref{sec:a-priori-estimate-sobolev}, $u_{n}$ satisfies the a priori $W^{1,p}$ bound \eqref{eq:apriori-W1P}. By the Sobolev embedding theorem, we can replace $u_{n}$ by a subsequence so that it converges in $C^{0}_{\mathrm{loc}}$. Then the (locally defined) differences $w_{n,m}=u_{n}-u_{m}$ satisfy the Cauchy sequence property in the $C^{0}$ distance. The strategy is to show that $w_{n,m}$ locally solve equations of the form in \S\ref{sec:regul-moving-lagr}.

In local coordinates one computes:
\begin{equation}\label{eq:lc_difference}
  \bd_{s}w_{n,m}+J(u_{n})\bd_{t}w_{n,m}=[J(u_{m})-J(u_{n})]\bd_{t}u_{m}+J(u_{n})X_{t}(u_{n})-J(v_{n})X_{t}(v_{n}).
\end{equation}
Using the doubling trick if necessary, we may suppose that the local coordinate representation of $w_{n,m}$ is defined on the left half disk $(it_{0}+D(\epsilon))\cap (-\infty,0]\times \R$ or the full disk $s_{0}+it_{0}+D(\epsilon)$, with $s_{0}<-\epsilon$. In the half-disk case, we suppose the moving Lagrangian boundary conditions is given by $q(s)\times \R^{n}$.

One observes that $J_{s,t}=J(u_{n})$ is uniformly bounded in $W^{1,p}$, and $w_{n,m}(0,t)\in 0\times \R^{n}$. Either the elliptic estimate with Lagrangian boundary in \S\ref{sec:regul-moving-lagr} applies or standard interior estimates apply and one concludes:
\begin{equation*}
  \norm{w_{n,m}}_{W^{1,p}}\le C(\norm{A_{n,m}}_{L^{p}}+\norm{w_{n,m}}_{C^{0}}),
\end{equation*}
where $A_{n,m}$ is the right hand side of \eqref{eq:lc_difference}. Since $u_{n}$ is Cauchy in $C^{0}$, we conclude that both $\norm{w_{n,m}}_{C^{0}}$ converges to zero and that $\norm{A_{n,m}}_{L^{p}}\to 0$. Therefore, small local coordinate representations of $u_{n}$ are Cauchy in $W^{1,p}$, hence $u_{n}$ converges in $W^{1,p}_{\mathrm{loc}}$.

\appendix

\section{Coherent orientations for determinant lines}
\label{sec:coher-orient}

Our approach to the problem of orienting the moduli spaces is inspired by \cite{fh_coherent,oh_1997_cotangent, abbondandolo_schwarz,seidel_book,katic_milinkovic, penka_phd, penka_gt, abbondandolo_schwarz_corrigendum,abouzaid_monograph,zapolsky-orientation}.

\subsection{Linear theory for infinite strips}
\label{sec:linear-theory}

\subsubsection{Asymptotic operators}
\label{sec:asymptotic-operators}

An \emph{asymptotic operator} is a differential operator of the form $A=-J_{0}\bd_{t}-S(t)$, acting on the space of smooth functions $[0,1]\to \R^{2n}$ taking boundary values in $\R^{n}$. We require that $J_{0}$ is the standard almost complex structure and $S(t)$ is a symmetric matrix. One says that $A$ is \emph{non-degenerate} if $A$ has zero kernel.

\subsubsection{Paths in the linear symplectic group}
\label{sec:paths-lin-sympl}
Let $\Phi_{t}$ be the \emph{fundamental solution} of the ordinary differential equation $A\eta=0$, i.e.,
\begin{equation*}
  \bd_{t}\Phi_{t}=J_{0}S(t)\Phi_{t}\text{ and }\Phi_{0}=1\in \R^{2n\times 2n}.
\end{equation*}
Then $\Phi_{t}$ is valued in the linear symplectic group. Moreover, every smooth path of symplectic matrices starting at the identity arises in this fashion.

Every solution to $A\eta=0$ satisfies $\eta(t)=\Phi_{t}\eta(0)$. Thus the condition that $A$ has zero kernel when restricted to the sections over $[0,1]$ which take boundary values in $\R^{n}$ is equivalent to requiring that $\Phi_{1}\R^{n}$ is transverse to $\R^{n}$.

\subsubsection{Cauchy-Riemann operators on strips}
\label{sec:cauchy-riem-oper}

A \emph{Cauchy-Riemann operator with asymptotic operators $A_{\pm}=-J_{0}\bd_{t}-S_{\pm}(t)$} on the strip $\R\times [0,1]$ is the data of a Cauchy-Riemann operator: $$D=\bd_{s}+J_{0}\bd_{t}+S(s,t)$$ on the trivial bundle $\R^{2n}$ so that $S(\pm s,t)-S_{\pm}(t)$ is in $L^{p}([0,\infty)\times [0,1])$.

\subsubsection{Sobolev spaces}
\label{sec:sobolev-spaces}
The domain of a Cauchy-Riemann operator is $W^{1,p}(\R^{2n},\R^{n})$, i.e., we restrict to sections which take $\R^{n}$-boundary values. The codomain is $L^{p}(\R^{2n})$. See \cite[Appendix B]{mcduffsalamon} and \cite{cant_thesis} for more details.

\subsubsection{Fredholm property and space of Cauchy-Riemann operators}
\label{sec:determ-line-bundle}

Denote the space of Cauchy-Riemann operators satisfying the properties in \S\ref{sec:cauchy-riem-oper} by: $$\mathrm{CR}:=\mathrm{CR}(A_{-},A_{+}).$$ One topologizes $\mathrm{CR}$ as an affine space, i.e., if $D_{0}$ is a chosen basepoint in $\mathrm{CR}$, then every other Cauchy-Riemann operator can be expressed as $D_{0}+B$ where $B$ is a smooth $L^{p}$-integrable section of the bundle of homomorphisms $\R^{2n}\to \R^{2n}$. In particular, $\mathrm{CR}$ is a contractible space.

Consider now $\mathrm{Fred}:=\mathrm{Fred}(W^{1,p},L^{p})$. Suppose henceforth that the asymptotics $A_{\pm}$ are non-degenerate. It is well-known that every element of $\mathrm{CR}$ induces a Fredholm operator $W^{1,p}\to L^{p}$. Moreover, the natural map $\mathrm{CR}\to \mathrm{Fred}$ is continuous when the codomain is topologized via the Banach space topology.

\subsubsection{Gluing operation}
\label{sec:gluing-operation}
We follow \cite{fh_coherent}; see also \cite{floer-hofer-sh-i,seidel_book,FOOO_part2,abouzaid_monograph}.

For every real number $R>0$, define a \emph{gluing map}: $$\mathfrak{G}_{R}:\mathrm{CR}(A_{-},A)\times \mathrm{CR}(A,A_{+})\to \mathrm{CR}(A_{-},A_{+}),$$ by the formula:
\begin{equation*}
  \mathfrak{G}_{R}(D_{-},D_{+}):=\beta_{-}(s)D_{-}^{-R}+\beta_{+}(s)D_{+}^{R},
\end{equation*}
where:
\begin{enumerate}
\item $D^{R}$ is obtained by conjugating $D$ with the isomorphism $\eta(s,t)\mapsto \eta(s+R,t)$,
\item $\beta_{-}(s)+\beta_{+}(s)=1$ and $\beta_{-}(s)=1$ for $s\le -1$, $\beta_{+}(s)=1$ for $s\ge 1$.
\end{enumerate}
The importance of having matching asymptotic operator $A=-J_{0}\bd_{t}-S(t)$ is the following: for any choice of $\delta,\rho>0$, one can choose $R_{0}$ large enough that the glued operator $D_{R}:=\mathfrak{G}_{R}(D_{-},D_{+})$ satisfies:
\begin{equation}\label{eq:important_estimate}
  \norm{D_{R}(\xi)-(\bd_{s}\xi+J_{0}\bd_{t}\xi+S(t)\xi)}_{L^{p}([-\rho,\rho])}\le \delta \norm{\xi}_{W^{1,p}([-\rho,\rho])},
\end{equation}
for $R\ge R_{0}$. This estimate is important when relating the kernels/cokernels of $D_{\pm}$ to the kernel/cokernels of $D_{R}$.

\subsubsection{Kernel and cokernel gluing}
\label{sec:kernel-gluing}
Let $C_{\pm}$ represent the cokernels of $D_{\pm}$. As in \S\ref{sec:gluing-operation}, there is a gluing operation $C_{-}\oplus C_{+}\to L^{p}$, given by: $$\eta_{-}(s,t)\oplus \eta_{+}(s,t)\mapsto \eta_{-}(s+R,t)+\eta_{+}(s-R,t).$$

Consider the map $\Delta_{R}:W^{1,p}\oplus C_{-}\oplus C_{+}\to L^{p}$ which applies $D_{R}$ to the first factor and glues the cokernel elements as above.
\begin{prop}\label{prop:kernel-cokernel-gluing}
  There is a uniformly bounded projection $\Pi_{R}:W^{1,p}\to K_{-}\oplus K_{+}$ so that $\Delta_{R}\oplus \Pi_{R}:W^{1,p}\oplus C_{-}\oplus C_{+}\to L^{p}\oplus K_{-}\oplus K_{+}$ is an isomorphism with a uniformly bounded inverse as $R\to\infty$.
\end{prop}
\begin{proof}
  The reader is referred to \cite{fh_coherent} for the proof.
\end{proof}

Therefore the kernel of $D_{R}$ is canonically-up-to-homotopy identified with $K_{-}\oplus K_{+}$, and similarly the cokernel is identified with $C_{-}\oplus C_{+}$. In particular, one obtains the well-known result that the Fredholm index of $D_{R}$ is the sum of the Fredholm indices of $D_{-},D_{+}$. More importantly for us, kernel gluing establishes homotopically canonical relationships between the determinant lines of the operators $D_{-},D_{+},D_{R}$, see \S\ref{sec:equiv-ident-determ}.

\subsubsection{Determinant lines}
\label{sec:determ-lines-axioms}

To each Fredholm operator $D:X\to Y$ one can associate its \emph{determinant line} $\det(D):=\det(\ker D\oplus (\coker D)^{\vee}),$ where $\det(V)$ is the top exterior power. In particular, $\det(0)=\R$.

It is well-known that $\det(D)\to \mathrm{Fred}(X,Y)$ can be topologized as a continuous real line bundle with the following properties:
\begin{enumerate}[label=(\roman*)]
\item If $\varphi:\mathrm{Fred}(X,Y)\to \mathrm{Fred}(X',Y')$ is given by conjugation $\varphi:D\mapsto BDA^{-1}$ with isomorphisms $A:X\to X'$ and $B:Y\to Y'$ then the induced map:
  \begin{equation*}
    \ker D\oplus (\coker D)^{\vee}\mapsto \ker \varphi(D)\oplus (\coker \varphi(D))^{\vee},
  \end{equation*}
  given by multiplication by $A$ on $\ker D$ and precomposition with $B^{-1}$ on $(\coker D)^{\vee}$, is a continuous map $\det(D)\mapsto \varphi^{*}\det(D)$.
\item If $D$ is an isomorphism, then the conjugation action in (i) preserves the canonical identification $\det(D)=\det(0)=\R$.
\end{enumerate}

See \cite[\S A.2]{mcduffsalamon}, \cite{fh_coherent, seidel_book, abouzaid_monograph, zinger_determinant} for further discussion.

\subsubsection{Determinant lines for Cauchy-Riemann operators on infinite strips}
\label{sec:determ-lines-cauchy}

Pulling back by $\mathrm{CR}(A_{-},A_{+})\to \mathrm{Fred}(W^{1,p},L^{p})$ induces a line bundle $\det(D)\to \mathrm{CR}(A_{-},A_{+})$. Since each space $\mathrm{CR}(A_{-},A_{+})$ is contractible, this line bundle is orientable.

\subsubsection{Group of automorphisms and conjugation action}
\label{sec:group-of-automorphisms}

Introduce the group of automorphisms $G$ consisting of maps $g:\R\times [0,1]\to \mathrm{U}(n)$ so that $g$ maps the boundary into $\mathrm{SO}(n)$ and the matrix difference $g-1$ lies in $W^{1,p}$.

Then $G$ acts on $W^{1,p}$ and $L^{p}$ by left-multiplication, and hence $G$ acts on the space of Cauchy-Riemann operators $\mathrm{CR}(A_{-},A_{+})$ by conjugation, say $D\mapsto gDg^{-1}$.

As in \S\ref{sec:determ-lines-axioms}, there are induced isomorphisms $g_{*}:\det(D)\to g^{*}\det(D)$ for each $g\in G$. Since $\det(D)\to \mathrm{CR}(A_{-},A_{+})$ is orientable, it makes sense to say whether or not $g_{*}$ preserves or reverses orientation. Denote by $G_{+}\subset G$ the set of automorphisms which preserve orientation.

\subsubsection{Equivariant identification of determinant lines via gluing}
\label{sec:equiv-ident-determ}

By the gluing construction in \S\ref{sec:gluing-operation}, there is an identification:
\begin{equation}\label{eq:homotopy_canonical}
  \det(D_{-})\otimes \det(D_{+})\to \det(\mathfrak{G}_{R}(D_{-},D_{+})),
\end{equation}
for $R$ sufficiently large, given by gluing the kernels and cokernels together, i.e., using the homotopically canonical isomorphism:
\begin{equation*}
  \ker(D_{-})\oplus \coker(D_{-})^{\vee}\oplus \ker(D_{+})\oplus \coker(D_{+})^{\vee}\to \ker(D_{R})\oplus \coker(D_{R})^{\vee},
\end{equation*}
for $R$ sufficiently large, and turning $\otimes$ into $\wedge$ when going from left to right in \eqref{eq:homotopy_canonical}.

If $g\in G$ is compactly supported, the gluing construction shows that $g_{*}$ reverses orientation on $\mathrm{CR}(A_{-},A)$ if and only if it does on $\mathrm{CR}(A,A_{+})$; see \cite{fh_coherent} for more details. As a consequence, the group of orientation preserving automorphisms $G_{+}$ is independent of $A_{\pm}$.

\subsubsection{Path components in the group of automorphisms}
\label{sec:path-comp-group}
Restricting $g\in G$ to the boundary of the strip gives two continuous paths in $\mathrm{SO}(n)$ based at $1$. Clearly the map $\pi_{0}(G)\mapsto \pi_{1}(\mathrm{SO}(n))\times \pi_{1}(\mathrm{SO}(n))$ is a group homomorphism. Moreover, it is easy to see this homomorphism is an isomorphism, using that $\pi_{1}(\mathrm{SO}(n))\to \pi_{1}(\mathrm{U}(n))$ is trivial and $\pi_{2}(\mathrm{U}(n))=0$.

When $n\ge 3$ there is isomorphism $\pi_{0}(G)\to \Z/2\times \Z/2$, while if $n=2$, the isomorphism is $\pi_{0}(G)\to \Z\times \Z$, and if $n=1$, $\pi_{0}(G)$ is trivial.

It is obvious that $2\Z\times 2\Z$ lies in $\pi_{0}(G_{+})$. Moreover, some analysis shows that $(1,1)$ lies in $\pi_{0}(G_{+})$; indeed, $(1,1)$ can be represented by a $t$-independent map $g_{s}\in \mathrm{SO}(n)$, and one checks that:
\begin{equation*}
  g_{\rho s}^{-1}D_{0}g_{\rho s}=D_{0}+\rho g'_{\rho s},
\end{equation*}
if $D_{0}=\bd_{s}+J_{0}\bd_{t}+\delta$ for some constant $\delta\ne \pi\Z$ (this ensures the asymptotic is non-degenerate). It follows easily that $g_{\rho s}^{-1}D_{0}g_{\rho s}$ converges to $D_{0}$ in $L^{p}$ as $\rho\to 0$. Since $D_{0}$ is an isomorphism, see \cite{cant_thesis,salamon1997}, it has a path connected neighborhood consisting of isomorphisms. The canonical orientation of isomorphisms is continuous on this neighborhood, by \S\ref{sec:determ-lines-axioms}. Since conjugation preserves the canonical orientation, one concludes that $g_{*}$ also preserves any global orientation. See \cite[Lemma 1.13]{fh_coherent} for the same argument.

The next proposition shows that $(1,0)$ and $(0,1)$ actually do reverse orientation.
\begin{prop}\label{prop:orientation_reversing}
  If $g$ lies in the path component of $G$ represented by $(1,0)$, then $g_{*}$ reverses any global orientation of the determinant line bundle. In other words, the quotient group $G/G_{+}$ is generated by $g$.
\end{prop}
\begin{proof}
  This appears to be a fairly deep fact, and follows from the computations in \cite[\S4]{abbondandolo_schwarz_corrigendum} and \cite[\S8.1.2]{FOOO_part2}, by a sequence of gluing operations.

  Indeed, since $g=(1,0)$ can be represented by an element compactly supported in $$D(\epsilon)\cap (\R\times [0,1]),$$ one can use the equivariance under gluing trick to show that $(1,0)$ reverses orientation by analyzing the conjugation action on practically any space of Cauchy-Riemann operators on $\R^{2n}$ with $\R^{n}$ boundary conditions; see Figure \ref{fig:analyzing-orientation}.

  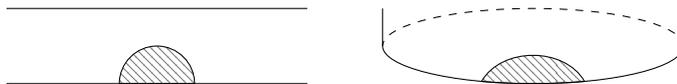
\begin{figure}[H]
    \centering
    \begin{tikzpicture}
      \draw (0,0)--(4,0) (0,1)--+(4,0);
      \draw[pattern={north west lines}, pattern color=black!50!white] (2.5,0) arc (0:180:0.5) --cycle;
      \begin{scope}[shift={(7,0.5)}]
        \draw[dashed] (2,0) arc (0:180:2 and 0.5);
        \draw (2,0) arc (0:-180:2 and 0.5);
        \draw (2,0)--+(0,.5) (-2,0)--+(0,.5);
        \draw[pattern={north west lines}, pattern color=black!50!white] (-70:2 and 0.5)to[out=120,in=60](-110:2 and 0.5) arc (-110:-70:2 and 0.5);
      \end{scope}
    \end{tikzpicture}
    \caption{Analyzing whether or not an element reverses orientation. Via equivariance under gluing, we can ``transport'' the compactly supported element to any Riemann surface with boundary.}
    \label{fig:analyzing-orientation}
  \end{figure}

  This observation shows that, if there is \emph{any} Riemann surface $\Sigma$ with boundary, and a suitable space of Cauchy-Riemann operators for $\R^{2n}$ valued sections with $\R^{n}$ boundary conditions over $\Sigma$, and the action of $g$ reverses orientation, then the action of $g$ will also reverse orientation in our context. Thus we can appeal to the computations in \cite[\S4]{abbondandolo_schwarz_corrigendum} (which apply to a half-infinite cylinder) or \cite[\S8.1.2]{FOOO_part2} (which apply to a disk). This completes the proof.
\end{proof}

\subsubsection{Conley-Zehnder indices as Fredholm indices}
\label{sec:conl-zehnd-indic}

For a non-degenerate asymptotic operator $A$, define the index $\mathrm{CZ}(A)$ to be the Fredholm index of $\mathrm{CR}(A_{0},A)$, picking the reference operator $A_{0}=-J_{0}\bd_{t}-C,$ where $C$ is the matrix of complex conjugation. The choice of $A_{0}$ is not particularly important, and we use the choice from \cite{cant_thesis}. In any case, it follows that $\mathrm{CZ}(A_{0})=0$, and the Fredholm index of $\mathrm{CR}(A_{1},A_{2})$ equals $\mathrm{CZ}(A_{2})-\mathrm{CZ}(A_{1})$.

\subsubsection{Orientation lines for asymptotic operators}
\label{sec:orientation-lines}

Let $\mathfrak{o}(A_{1},A_{2})$ be the free $\Z$-module (of rank one) with an identification between the two choices of generator and the choices of global orientation of $\det(D)\to \mathrm{CR}(A_{1},A_{2})$, with $\mathfrak{o}(A):=\mathfrak{o}(A_{0},A)$.

Since $\mathrm{CR}(A_{0},A_{0})$ has a basepoint $D=\bd_{s}-A_{0}$ which is an isomorphism, the orientation line for $\det(D)\to \mathrm{CR}(A,A_{0})$ is canonically identified with $\mathfrak{o}(A)$ by the gluing construction. Similar application of gluing yields $\mathfrak{o}(A_{1}) \otimes \mathfrak{o}(A_{1},A_{2})\simeq\mathfrak{o}(A_{2})$, in particular, a choice of orientation for $\det(D)\to \mathrm{CR}(A_{1},A_{2})$ induces an isomorphism $\mathfrak{o}(A_{2})\to \mathfrak{o}(A_{1})$. This perspective is used when defining the Floer theoretic operations in \S\ref{sec:floer-differential-local-coefficients}, \S\ref{sec:continuation-maps-hf}, and \S\ref{sec:isom-from-morse}.

\subsection{Linear theory for half-infinite strips}
\label{sec:linear-theory-half}

Let $\Sigma$ denote the space of Lagrangians $\Pi\subset \R^{2n}$ which satisfy $$\Pi=(\Pi\cap \R^{n})\oplus (\Pi\cap J\R^{n}),$$ and let $\Sigma_{d}\subset \Sigma$ be the subset where $\dim(\Pi\cap \R^{n})=d$. Note that $\Pi\in \Sigma$ is completely determined by $\Pi\cap \R^{n}$, and hence $\Sigma_{d}$ is diffeomorphic to the Grassmann manifold of $d$-planes in $\R^{n}$. See \S\ref{sec:linear-conormals} for more details.

To every $\Pi\in \Sigma$ and non-degenerate asymptotic operator $A=-J_{0}\bd_{t}-S(t)$, let $\mathrm{CR}(A,\Pi)$ denote the set of operators of the form $\bd_{s}+J_{0}\bd_{t}+S(s,t)$ on $(-\infty,0]\times [0,1],$ defined on $W^{1,p}$ sections $\xi$ satisfying $\xi(s,i)\in \R^{n}$ for $i=0,1$ and $\xi(0,t)\in \Pi$, and so that $S(s,t)-S(t)$ lies in $L^{p}$.

\begin{figure}[H]
  \centering
  \begin{tikzpicture}
    \draw (0,0)--node[below]{$\R^{n}$}(5,0)--node[right]{$\Pi$}+(0,1)coordinate(Y)--node[above]{$\R^{n}$}(0,1)coordinate(X);
    \draw[dashed] (X)--node[left]{$A$}+(0,-1);
    \path (0,0)--node{$\bd_{s}+J_{0}\bd_{t}+S(s,t)$}(Y);
  \end{tikzpicture}
  \caption{Cauchy-Riemann operators on the half-infinite strip.}
\end{figure}

As in \S\ref{sec:linear-theory}, $\mathrm{CR}(A,\Pi)$ is a contractible space of Fredholm operators, see \S\ref{sec:fredholm-index-half}, and therefore the natural determinant line bundle over $\mathrm{CR}(A,\Pi)$ is orientable.

\subsubsection{Linear conormals and anti-symplectic involutions}
\label{sec:linear-conormals}

Observe that $\Pi$ is fixed under complex conjugation. Indeed, a Lagrangian is fixed under complex conjugation if and only if it lies in $\Sigma$.

On the other hand, every linear Lagrangian $L$ determines an anti-symplectic involution, which acts identically on $L$ and multiplies by $-1$ on $J_{0}L$. If $\Pi$ is a linear conormal, then the involution through $\Pi$ fixes $\R^{n}$.

\subsubsection{Fredholm property and index for half-infinite strips}
\label{sec:fredholm-index-half}

Doubling using the anti-symplectic involutions from \S\ref{sec:linear-conormals}, one shows that, for a non-degenerate asymptotic $A$, Cauchy-Riemann operators in $\mathrm{CR}(A,\Pi)$ are Fredholm.

\begin{prop}
  The Fredholm index of $D\in \mathrm{CR}(A,\Pi)$ is equal to:
  \begin{equation*}
    \mathrm{Index}(D)=d-n-\mathrm{CZ}(A),
  \end{equation*}
  where $d=\dim(\Pi\cap \R^{n})$.
\end{prop}
\begin{proof}
  It is clear that the index depends only on $A,d,n$; denote it by $\mathfrak{i}(A,d,n)$. A gluing argument shows that it suffices to prove it for any fixed asymptotic operator. Introduce the reference operator:
  \begin{equation*}
    A_{\delta}=-J_{0}\bd_{t}+\delta,
  \end{equation*}
  for $\delta\in (0,\pi)$, so that $\mathrm{CZ}(A_{\delta})=-n$, as shown in \cite[\S2.3.1.3]{cant_thesis}. Thus it suffices to prove that:
  \begin{equation*}
    \mathfrak{i}(A_{\delta},d,n)=d.
  \end{equation*}
  Fixing $D_{\delta}=\bd_{s}+J_{0}\bd_{t}-\delta$, one uses basic Fourier analysis to prove that:
  \begin{equation*}
    D_{\delta}u=0 \iff u(s,t)=e^{\delta s}X,
  \end{equation*}
  where $X$ is a fixed vector in $\Pi\cap \R^{n}$. The formal adjoint $D_{\delta}^{*}=-\bd_{s}+J_{0}\bd_{t}-\delta$ is defined on the space of sections which take values in $J_{0}\Pi$ when $s=0$ and in $\R^{n}$ when $t=0,1$. Again using Fourier analysis, one shows $D_{\delta}^{*}$ is injective. Thus $D_{\delta}$ is surjective and has a kernel identified with the $d$-dimensional space $\Pi\cap \R^{n}$. This implies $\mathfrak{i}(A_{\delta},d,n)=d$, as desired.
\end{proof}

\subsubsection{Orientation lines}
\label{sec:orientation-lines-1}

As in \S\ref{sec:orientation-lines}, let $\mathfrak{o}(A,\Pi)$ denote the orientation line for the determinant line bundle over $\mathrm{CR}(A,\Pi)$. Gluing gives a canonical isomorphism:
\begin{equation*}
  \mathfrak{o}(A_{-},\Pi)\simeq \mathfrak{o}(A_{-},A_{+})\otimes \mathfrak{o}(A_{+},\Pi),
\end{equation*}
where $\mathfrak{o}(A_{-},A_{+})$ is the orientation line for the determinant line of $\mathrm{CR}(A_{-},A_{+})$.

\subsubsection{Group of automorphisms}
\label{sec:group-of-automorphisms-1}
The group of automorphisms $G(\Pi)$ consists of maps:
\begin{equation*}
  \left\{
    \begin{aligned}
      &g:(-\infty,0]\times [0,1]\to \mathrm{U}(n),\\
      &g(s,0),g(s,1)\in \mathrm{SO}(\R^{n}) \text{ and }g(0,t)\in \mathrm{SO}(\Pi),\\
      &1-g\in W^{1,p},
    \end{aligned}
  \right.
\end{equation*}
where $\mathrm{SO}(\Pi)\subset \mathrm{U}(n)$ consists of the elements which fix $\Pi$ and preserve its orientation.

\subsubsection{Path components in group of automorphisms}
\label{sec:path-comp-group-1}
Arguing as in \S\ref{sec:path-comp-group-1}, one shows that $\pi_{0}(G)$ is classified by the homotopy class of a based loop satisfying the constraints defined in Figure \ref{fig:homotopy_class_Pi}.

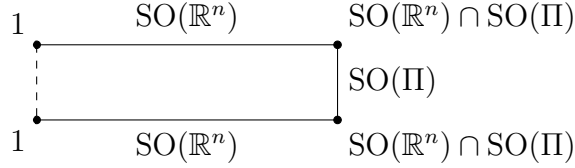
\begin{figure}[H]
  \centering
  \begin{tikzpicture}
    \draw (0,0)--node[below]{$\mathrm{SO}(\R^{n})$}(4,0)--node[pos=0,below right]{$\mathrm{SO}(\R^{n})\cap \mathrm{SO}(\Pi)$}node[pos=1,above right]{$\mathrm{SO}(\R^{n})\cap \mathrm{SO}(\Pi)$}node[pos=1,draw,circle,inner sep=1pt,fill]{}node[pos=0,draw,circle,inner sep=1pt,fill]{}node[right]{$\mathrm{SO}(\Pi)$}+(0,1)coordinate(Y)--node[above]{$\mathrm{SO}(\R^{n})$}(0,1)coordinate(X);
    \draw[dashed] (X)--node[pos=1,draw,circle,inner sep=1pt,fill]{}node[pos=0,draw,circle,inner sep=1pt,fill]{}node[pos=1,below left]{$1$}node[pos=0,above left]{$1$}+(0,-1);
  \end{tikzpicture}
  \caption{The path component of $g\in G(\Pi)$ is classified by the homotopy class of a based loop.}
  \label{fig:homotopy_class_Pi}
\end{figure}

In the case $\Pi\cap \R^{n}=\R^{d}$, then $\mathrm{SO}(\R^{n})\cap \mathrm{SO}(\Pi)$ is the subgroup of $\mathrm{O}(d)\times \mathrm{O}(n-d)$ of elements with determinant $+1$. In order to determine which elements of $G(\Pi)$ preserve orientation on $\mathrm{CR}(A,\Pi)$ we digress for a moment on a linear-algebraic structure related to conormals.

\subsubsection{Canonical isomorphism between conormals}
\label{sec:canon-ident-conormal}
Consider a vector space $E$ with a complex structure $J$. Let us agree to say that two totally real subspaces $L_{0},L_{1}$ are \emph{conormally related} if the unique anti-complex involution fixing $L_{0}$ preserves $L_{1}$. This is a symmetric relation. For example, with $\C^{n}$ and $L_{0}=\R^{n}$, the set of totally real subspaces conormally related to $\R^{n}$ are the linear conormals from \S\ref{sec:linear-conormals}.

If $L_{0},L_{1}$ are conormally related, there is a \emph{canonical} isomorphism $L_{0}\to L_{1}$; indeed, there are direct sum decompositions:
\begin{equation*}
  L_{1}=(L_{1}\cap L_{0})\oplus (L_{1}\cap JL_{0})\text{ and }L_{0}=(L_{0}\cap L_{1})\oplus (L_{0}\cap JL_{1}),
\end{equation*}
and we define the canonical isomorphism $C_{J}(L_{0},L_{1})$ to act identically on $L_{0}\cap L_{1}$ and by $J$ on the second factor $L_{0}\cap JL_{1}$. We extend $C_{J}(L_{0},L_{1})$ to all of $E$ by requiring that $C_{J}(L_{0},L_{1})J=JC_{J}(L_{0},L_{1})$. It then holds that $C_{J}(L_{0},L_{1})=C_{J}(L_{1},L_{0})$.

\subsubsection{Orientation preserving group of automorphisms}
\label{sec:orient-pres-group}

Define $G_{+}(\Pi)$ to be the automorphisms which preserve any global orientation of the determinant line on $\mathrm{CR}(A,\Pi)$. As in \S\ref{sec:equiv-ident-determ}, a gluing argument shows that the group is independent of the asymptotic operator $A$.

Given $g_{z}\in G(\Pi)$ we can consider the map $\bd \Sigma\to \mathrm{SO}(n)$ given by:
\begin{equation*}
  [g]_{\mathfrak{m}}=\left\{
    \begin{aligned}
      &g_{z}&&\text{ for }t(z)=0,1,\\
      &C^{-\mathfrak{m}}g_{z}C^{\mathfrak{m}}&&\text{ for }s(z)=0,
    \end{aligned}
  \right.
\end{equation*}
where $\mathfrak{m}\in \set{-1,1}$ and where $C=C_{J_{0}}(\Pi,\R^{n})$. One notes that this is a continuous family of unitary matrices which preserve $\R^{n}$, and which is based at the identity (as $s\to -\infty$).

\begin{lemma}\label{lemma:some-choice-of-sign}
  For some choice of $\mathfrak{m}$, if $[g]_{\mathfrak{m}}\in \pi_{1}(\mathrm{SO}(n),1)$ is trivial then $g\in G_{+}(\Pi)$.
\end{lemma}
\begin{proof}
  This is trivial if $\Pi=\R^{n}$, so suppose $\Pi\ne \R^{n}$.

  Note that $\mathrm{SO}(\R^{n})\cap \mathrm{SO}(\Pi)$ always contains $\pm 1$. Consider a path $\gamma$ in $\mathrm{SO}(\R^{n})$ which start at $+1$ and ends at $-1$ so that $\gamma^{2}$ generates $\pi_{1}(\mathrm{SO}(\R^{n}))$.

  Then $C^{\mathfrak{m}}\gamma(1-t)C^{-\mathfrak{m}}$ is a path in $\mathrm{SO}(\Pi)$ which starts at $-1$ and ends at $1$. Consider the element $g_{\mathfrak{m}}$ in $G(\Pi)$ represented by the diagram shown in Figure \ref{fig:special-element}.

  This special element has the property that $[g]_{\mathfrak{m}}$ is trivial. We claim that exactly one of $g_{-},g_{+}$ preserves orientation. This can be as follows, suppose: $$\gamma(t)q_{1}=\cos(\pi t)q_{1}+\sin(\pi t)q_{2}$$ where $q_{1}\in \R^{n}\cap \Pi$ and $q_{2}\in \R^{n}\cap (J_{0}\Pi)$ (and suppose $\gamma(t)$ fixes $q_{3},\dots,q_{n}$).

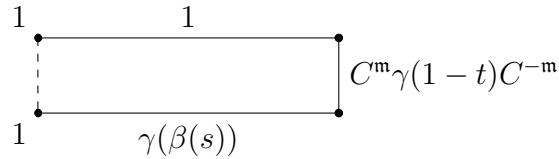
\begin{figure}[H]
    \centering
    \begin{tikzpicture}
      \draw (0,0)--node[below]{$\gamma(\beta(s))$}(4,0)--node[pos=1,draw,circle,inner sep=1pt,fill]{}node[pos=0,draw,circle,inner sep=1pt,fill]{}node[right]{$C^{\mathfrak{m}}\gamma(1-t)C^{-\mathfrak{m}}$}+(0,1)coordinate(Y)--node[above]{$1$}(0,1)coordinate(X);
    \draw[dashed] (X)--node[pos=1,draw,circle,inner sep=1pt,fill]{}node[pos=0,draw,circle,inner sep=1pt,fill]{}node[pos=1,below left]{$1$}node[pos=0,above left]{$1$}+(0,-1);
    \end{tikzpicture}
    \caption{The special element $\mathfrak{g}_{\mathfrak{m}}\in G(\Pi)$; here $\beta$ is cut off function so $\beta(0)=1$ and $\beta(-\infty)=0$.}
    \label{fig:special-element}
  \end{figure}  
  One computes:
  \begin{equation*}
    \begin{aligned}
      C^{-1}\gamma(t)Cq_{1}&=\cos(\pi t)q_{1}-\sin(\pi t)p_{1}\\ C\gamma(t)C^{-1}q_{1}&=\cos(\pi t)q_{1}+\sin(\pi t)p_{1}.
    \end{aligned}
  \end{equation*}
  Thus $C^{-1}\gamma(t)C$ and $C\gamma(t)C^{-1}$ rotate $q_{1}$ in opposite directions. In particular:
  \begin{equation*}    C^{-\mathfrak{m}}\gamma(1-t)C^{\mathfrak{m}}=C^{\mathfrak{m}}\gamma(t)C^{-\mathfrak{m}}.
  \end{equation*}
  Therefore the product $g_{\mathfrak{m}}g_{-\mathfrak{m}}$ reverses orientation; see Figure \ref{fig:product-special-element}.

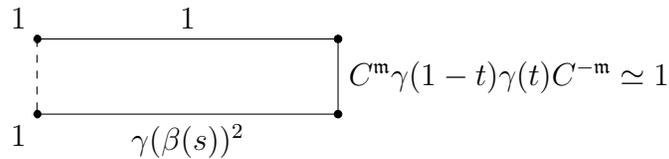
\begin{figure}[H]
    \centering
    \begin{tikzpicture}
      \draw (0,0)--node[below]{$\gamma(\beta(s))^{2}$}(4,0)--node[pos=1,draw,circle,inner sep=1pt,fill]{}node[pos=0,draw,circle,inner sep=1pt,fill]{}node[right]{$C^{\mathfrak{m}}\gamma(1-t)\gamma(t)C^{-\mathfrak{m}}\simeq 1$}+(0,1)coordinate(Y)--node[above]{$1$}(0,1)coordinate(X);
    \draw[dashed] (X)--node[pos=1,draw,circle,inner sep=1pt,fill]{}node[pos=0,draw,circle,inner sep=1pt,fill]{}node[pos=1,below left]{$1$}node[pos=0,above left]{$1$}+(0,-1);
    \end{tikzpicture}
    \caption{The product $g_{\mathfrak{m}}g_{-\mathfrak{m}}$ reverses orientation because $\gamma(\beta(s))^{2}$ has winding number $1$ and Proposition \ref{prop:orientation_reversing} applies.}
    \label{fig:product-special-element}
  \end{figure}
 
  If a product of two elements reverses orientation, then at least one of the elements preserves orientation; say $\mathfrak{m}$.

  On the other hand, the element $g_{\mathrm{ti}}$ described in Figure \ref{fig:t-invariant-special-element} preserves orientation because it is $t$-invariant, and clearly has $[g_{\mathrm{ti}}]_{\mathfrak{m}}=1$ for either choice of $\mathfrak{m}$.
  
  Thus, by applying $g_{\mathrm{ti}}$ and $g_{\mathfrak{m}}$ we may alter any given element $g$ \emph{without changing} $[g]_{\mathfrak{m}}$ or whether or not $g$ reverses orientation, until $g$ can be represented by an element which equals $1$ on $s=0$, $t=1$ (and is potentially non-constant on the $t=0$ boundary). Proposition \ref{prop:orientation_reversing} then shows that $[g]_{\mathfrak{m}}=1\implies g\in G_{+}(\Pi)$, as desired.
\end{proof}

\begin{figure}[H]
    \centering
    \begin{tikzpicture}
      \draw (0,0)--node[below]{$\gamma(\beta(s))$}(4,0)--node[pos=1,draw,circle,inner sep=1pt,fill]{}node[pos=0,draw,circle,inner sep=1pt,fill]{}node[right]{1}+(0,1)coordinate(Y)--node[above]{$\gamma(\beta(s))$}(0,1)coordinate(X);
      \draw[dashed] (X)--node[pos=1,draw,circle,inner sep=1pt,fill]{}node[pos=0,draw,circle,inner sep=1pt,fill]{}node[pos=1,below left]{$1$}node[pos=0,above left]{$1$}+(0,-1);
    \end{tikzpicture}
    \caption{This special element preserves orientation because it is represented by a $t$-invariant element, and the rescaling argument of \S\ref{sec:path-comp-group} applies.}
    \label{fig:t-invariant-special-element}
  \end{figure}
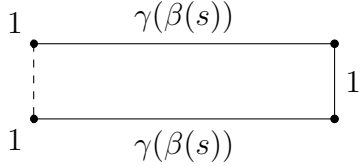

\subsection{Orienting the non-linear moduli spaces}
\label{sec:orienting-non-linear}

\subsubsection{Outline of the approach}
\label{sec:outline-approach}

Our approach follows \cite{oh_1997_cotangent}. Let $\mathscr{M}$ be a moduli space of solutions to Floer's equation. To each $u\in\mathscr{M}$ one associates a linearized operator $D_{u}$. The orientability of $\mathscr{M}$ reduces to the orientability of a determinant line bundle $\det(D)\to \mathscr{M}$, whose fiber over $u$ is the determinant line of $D_{u}$.

We define $D_{u}$ using local coordinates near $u$; for each choice of coordinates $\psi$, one obtains a local coordinate representation $D^{\psi}_{u}$. One arranges that $D^{\psi}_{u}$ is valued in a contractible space of Cauchy-Riemann operators $\mathrm{CR}$, e.g., $\mathrm{CR}(A_{-},A_{+})$ or $\mathrm{CR}(A_{-},\Pi)$, or a space of operators which admits a deformation retraction onto one of these spaces.

Different choices of coordinates $\psi$ give conjugate elements of $\mathrm{CR}$, under the action of a group of automorphisms $G$, e.g., the groups considered in \S\ref{sec:group-of-automorphisms} and \S\ref{sec:group-of-automorphisms-1}.

The linearized operator $D_{u}$ is defined as an abstract limit, in the category theory sense, of the local coordinate representations $D^{\psi}_{u}$. The determinant line bundle over $\det(D)\to \mathscr{M}$ is topologized as a locally trivial line bundle whose structure group (a subgroup of $\Z/2$) is determined by the action of $G$ on the orientation line of the determinant line bundle over $\mathrm{CR}$. See \S\ref{sec:orientation-lines} and \S\ref{sec:orientation-lines-1}.

In particular, if the choice of coordinates $\psi$ can be restricted so that the transition functions are always valued in the subgroup $G_{+}\subset G$ of orientation preserving automorphisms, then $\det(D)\to \mathscr{M}$ can be oriented.

The rest of this section is aimed at making this informal discussion precise.

\subsubsection{Frames for the conormal and the vertical subbundle}
\label{sec:fram-conorm-vert}
Throughout we fix a linear conormal $\Pi$ so that $\dim(\Pi\cap \R^{n})=\mathrm{codim}(N)$. We assume that $T\nu^{*}N$ and the vertical subbundle $V$ of $T^{*}M$ are conormally related, and hence there exist unitary frames $F:\R^{2n}\to TT^{*}M$ so that $F(\R^{n})=T\nu^{*}N$ and $F(\Pi)=V$ based at any point in $\nu^{*}N$; such frames can be constructed using appropriate canonical coordinates.

Denote $C_{0}=C_{J_{0}}(\Pi,\R^{n})$ and $C_{1}=C_{J}(T\nu^{*}N,V)$. It will be important that a frame $F$ for $T\nu^{*}N$ satisfies:
\begin{equation*}
  F(\Pi)=V\iff C_{1}F=FC_{0};
\end{equation*}
this can be proved by considering each summand in the decomposition:
\begin{equation*}
  \R^{2n}=(\R^{n}\cap \Pi)\oplus (\R^{n}\cap J_{0}\Pi)\oplus (J_{0}\R^{n}\cap \Pi)\oplus (J_{0}\R^{n}\cap J_{0}\Pi)
\end{equation*}
separately.

We fix the sign $\mathfrak{m}$ so that the conclusion of Lemma \ref{lemma:some-choice-of-sign} holds.

\subsubsection{Coordinates can be chosen symplectically}
\label{sec:canon-coord-can}
Let $\mathfrak{F}\to T^{*}M$ be the unitary frame bundle. A choice of metric with very large injectivity radius produces a smooth map $\Phi:\mathfrak{F}\times B(1)\to T^{*}M$ which restricts to an open embedding denoted $\Phi_{F}$ on each fiber $\set{F}\times B(1)$. Taking a metric pulled back from the base for which $N$ is the fixed point set of an isometric involution, one can ensure the following properties:

\begin{enumerate}[label=(\alph*)]
\item\label{item:moser-lag-a} $\d\Phi_{F}(0)=F$
\item\label{item:moser-lag-b} $\Phi_{F}^{-1}(\nu^{*}N)=\R^{n}\cap B(\delta)$ for each $F$ tangent to $\nu^{*}N$,
\item\label{item:moser-lag-c} $\Phi_{F}^{-1}(T^{*}M_{q})=\Pi\cap B(\delta)$ for each $F$ tangent to $V$,
\item\label{item:moser-lag-d} $\Phi_{F}C_{0}^{\mathfrak{m}}=\Phi_{FC_{0}^{\mathfrak{m}}}$ for all $F$ satisfying $F(\Pi)=V$.
\end{enumerate}

This coordinate map is essentially a coherent choice of smooth embeddings of balls. A fairly straightforward application of the Moser isotopy technique yields:
\begin{prop}\label{prop:symplectic_coordinates}
  The coordinate map $\Phi:\mathfrak{F}\times B(1)\to T^{*}M$ can be deformed so that it is symplectic on $\mathfrak{F}\times B(\delta)$, for $\delta$ sufficiently small, preserving the properties \ref{item:moser-lag-a}-\ref{item:moser-lag-d}.
\end{prop}
\begin{proof}
  One constructs a correcting flow $\psi_{t,F}:B(\delta)\to B(1)$ so that $\Phi_{F}\psi_{1,F}$ is symplectic on $B(\delta)$. The construction is parametric in $F$. Because $\Phi^{*}_{F}\omega-\omega_{0}$ vanishes at $0$ one can arrange that $\d\psi_{t,F}(0)=\id$. The details are left to the reader.
\end{proof}

\subsubsection{Asymptotic data at chords}
\label{sec:asymptotic-data}

Let $\gamma$ be a path with endpoints on $\nu^{*}N$. There are unitary frames $F(t)$ for $\gamma^{*}TT^{*}M$ so $F(0),F(1)$ span $T\nu^{*}N$ and $F(t)C_{0}$ spans $V$. One can first construct $F(t)C_{0}^{\mathfrak{m}}$ and then alter it at the endpoints so that $F(0),F(1)$ span $T\nu^{*}N$. Fix once and for all such a unitary frame $F_{\gamma}(t)$ for each chord.

Given the unitary frame, one considers the travelling coordinate chart $\Phi_{t}=\Phi_{F(t)}$. In this coordinate chart, one defines the linearized operator $A_{\gamma}$ as in \cite[\S3]{cant_thesis}. Briefly, it is defined by the formula:
\begin{equation*}
  \d\Phi_{t}(0)A_{\gamma}\eta:=\lim_{h\to 0}h^{-1}\mathrm{NL}(\Phi_{t}(h\eta(t))),
\end{equation*}
where $\mathrm{NL}(\gamma(t))=-J(\gamma(t))(\gamma'(t)-X_{t}(\gamma(t)))$. One shows that $A_{\gamma}=-J_{0}\bd_{t}-S(t)$ is a valid asymptotic operator.

From this construction, we extract:
\begin{enumerate}
\item the orientation line $\mathfrak{o}_{\gamma}=\mathfrak{o}(A_{\gamma})$, as in \S\ref{sec:orientation-lines},
\item the Conley-Zehnder index $\mathrm{CZ}(\gamma)=\mathrm{CZ}(A_{\gamma})$, as in \S\ref{sec:conl-zehnd-indic}.
\end{enumerate}

\subsubsection{Linearization procedure for infinite strips}
\label{sec:line-proc}
Let $u:\Sigma\to T^{*}M$ be a finite energy solution to Floer's equation with $\Sigma=\R\times [0,1]$.

Then $u$ is asymptotic at its punctures to chords $\gamma_{-}$ and $\gamma_{+}$. Pick a map $F_{u}:\Sigma\to \mathfrak{F}$, writing $\Phi_{z}=\Phi_{F_{u}(z)}$, and suppose that:
\begin{enumerate}
\item\label{item:line-proc-1} $u(z)$ lies in the image of $\Phi_{z}(B(\delta))$,
\item\label{item:line-proc-2} $F_{u}(s,t)=F_{\gamma_{-}}(t)$ for $s\le s_{0}$ and $F_{u}(s,t)=F_{\gamma_{+}}(t)$ for $s\ge s_{1}$,
\item\label{item:line-proc-3} $F_{u}(s,i)$ points along $T\nu^{*}N$, for $i=0,1$.
\end{enumerate}
Consider the partially defined frame for $V$:
\begin{equation}\label{eq:partially-defined}
  \left\{
    \begin{aligned}
      &C_{1}^{\mathfrak{m}}F_{u}(z)&&\text{ for }t(z)=0,1,\\
      &F_{u}(z)C_{0}^{\mathfrak{m}}&&\text{ for }s(z)\in (-\infty,s_{0}]\cup [s_{1},\infty).\\
    \end{aligned}
  \right.
\end{equation}
This frame agrees on the overlaps, because of the observation in \S\ref{sec:fram-conorm-vert}. We require the frames satisfy the homotopical condition:
\begin{enumerate}[resume]
\item\label{item:line-proc-4} the conormal frame $F_{u}^{*}$ extends to a frame of $u^{*}V$ defined on $\Sigma$.
\end{enumerate}
Since $\pi_{2}(\mathrm{U}(n),\mathrm{SO}(n),1)\to \pi_{1}(\mathrm{SO}(n),1)$ is surjective, we can always achieve \ref{item:line-proc-4} by precomposing $F_{u}(z)$ with a map $\Sigma\to \mathrm{U}(n)$.

If the conditions \ref{item:line-proc-1} through \ref{item:line-proc-4} are satisfied, we say the choice of frame is \emph{admissible}.

As in \S\ref{sec:asymptotic-data}, condition \ref{item:line-proc-1} allows us to consider the local coordinate representation:
\begin{equation*}
  u(s,t)=\Phi_{z}(\mu(s,t)),
\end{equation*}
Setting $\mathrm{NL}(u(s,t))=\bd_{s}u(s,t)+J_{s,t}(u)(\bd_{t}u-X_{s,t}(u))$, we have the formula for the linearized operator:
\begin{equation*}
  \d\Phi_{z}(\mu(s,t))D_{u}^{\Phi}(\eta)=\lim_{\epsilon \to 0}\epsilon^{-1}\mathrm{NL}(\Phi_{z}(\mu(s,t)+\epsilon \eta(s,t)));
\end{equation*}
see \cite[\S4.4]{cant_thesis} for further discussion.

It is straightforward to show that $D^{\Phi}_{u}$ is a Cauchy-Riemann operator on the trivial bundle over the infinite strip for the domain dependent complex structure $J^{\Phi}_{z}$ given by $\d\Phi_{z}(\mu)^{-1}J_{z}(u(s,t))\d\Phi_{z}(\mu)$. Since the coordinates are symplectic, $J^{\Phi}_{z}$ is tame.

It is important to note that, by our requirement that $\Phi_{z}$ eventually agrees with the coordinates used to linearize the ODE in \S\ref{sec:asymptotic-data}, that $D^{\Phi}_{u}$ is asymptotic to $\bd_{s}-A_{\gamma_{-}}$ and $\bd_{s}-A_{\gamma_{+}}$ at the positive and negative ends.

The space of Cauchy-Riemann operators with a tame domain-dependent complex structure deformation retracts onto the space $\mathrm{CR}(A_{\gamma_{-}},A_{\gamma_{+}})$, because the space of tame complex structures is contractible. This retraction establishes a canonical isomorphism:
\begin{equation}\label{eq:can-iso-Phi}
  \text{the orientation line for $\det(D_{u}^{\Phi})$} \simeq \mathfrak{o}(A_{\gamma_{-}},A_{\gamma_{+}});
\end{equation}
see \S\ref{sec:orientation-lines}. On the other hand, it is well-known that for two such choices of coordinates $u(z)=\Phi_{0,z}(\mu(z))$ and $u(z)=\Phi_{1,z}(\mu'(z))$, one has:
\begin{equation}\label{eq:invariance-conju}
  \d\Phi_{0,z}(\mu)D^{\Phi_{0}}_{u}\d\Phi_{0,z}(\mu)^{-1}=\d\Phi_{1,z}(\mu')D^{\Phi_{1}}_{u}\d\Phi_{1,z}(\mu')^{-1}.
\end{equation}
This conjugation establishes an isomorphism between the orientation lines of $\det(D_{u}^{\Phi_{0}})$ and $\det(D_{u}^{\Psi_{1}})$.

\begin{prop}\label{prop:coherent-orient-1}
  Suppose $\Phi_{0},\Phi_{1}$ are admissible coordinates. The isomorphisms induced by \eqref{eq:invariance-conju} preserve the isomorphisms \eqref{eq:can-iso-Phi}. In particular, if one defines $\det(D_{u})$ as the abstract limit of $\det(D_{u}^{\Phi})$ over the isomorphisms induced by \eqref{eq:invariance-conju}, then $\det(D)$ is an orientable real line bundle, and the corresponding orientation line is canonically identified with $\mathfrak{o}(A_{\gamma_{-}},A_{\gamma_{+}})$.
\end{prop}
\begin{proof}
  Let $P_{i}(z)=\d\Phi_{i,z}(\mu_{i}(z))$, and $M(z)=P_{1}(z)^{-1}P_{0}(z)$. Then $M(z)$ is a family of matrices in $\mathrm{Sp}(\R^{2n})$ over the strip and which preserve $\R^{n}$ when $z$ lies in boundary.

  Let $J_{1,z}=M(z)^{*}J_{0}$, and choose a path $J_{\tau,z}$ of $\omega_{0}$-tame complex structures so that $J_{0,z}=J_{0}$. Let $E_{\tau}(z)$ be the unique matrix which fixes $\R^{n}$ and so $E_{\tau}(z)^{*}J_{1,z}=J_{\tau,z}$.

  Then $M_{\tau}(z):=M(z)E_{\tau}(z)$ satisfies $M_{\tau}^{*}(z)J_{0}=J_{\tau,z}$. Consequently, $M_{1}(z)$ reverses orientation if and only if $M_{0}(z)$ reverses orientation. Moreover, $M_{0}(z)v=M_{1}(z)v$ for $z\in \bd \Sigma$ and $v\in \R^{n}$.

  It is clear that $M_{0}(z)$ restricts to $\bd\Sigma$ to two paths in $\pi_{1}(\mathrm{GL}(\R^{n}),1)$. These two paths represent the same element because of the homotopical condition \ref{item:line-proc-4} for admissible coordinates. Indeed, $P_{0},P_{1}$ induces conormal frames $P_{0}^{*},P_{1}^{*}$ for $u^{*}V$ along the $t=0,1$ boundaries. These frames extend to all values of $t$, and hence the homotopy class of the change of trivialization $s\mapsto (P_{1}^{*}(s,t))^{-1}P_{0}^{*}(s,t)$, considered as an element in $\pi_{1}(\mathrm{GL}(\R^{n}),1)$, is independent of $t$.

  It then follows from the arguments in \S\ref{sec:path-comp-group} that conjugation by $M_{0}$ preserves orientation. This completes the proof.
\end{proof}

\subsubsection{Linearization procedure for half-infinite strips}
\label{sec:half-proc}

Let $\Sigma=(-\infty,0]\times [0,1]$, and suppose that $u$ solves the boundary value problem:
\begin{equation*}
  \left\{
    \begin{aligned}
      &\bd_{s}u+J(u)\bd_{t}u=A_{s,t}(u)\\
      &u(s,0),u(s,1)\in \nu^{*}N\\
      &u(0,t)\in TM^{*}_{q(t)},
    \end{aligned}
  \right.
\end{equation*}
where $q(t)$ is a $W^{1,p}$ path.

Given $u:\Sigma\to T^{*}M$, pick a map $z\in \Sigma\mapsto F_{u}(z)\in \mathfrak{F}$. Similarly to \S\ref{sec:line-proc}, consider travelling frames which satisfy:
\begin{enumerate}
\item\label{item:half-proc-1} $u(z)$ lies in the image of $\Phi_{z}(B(\delta))$,
\item\label{item:half-proc-2} $F_{u}(z)=F_{\gamma_{-}}(t)$ for $s(z)\le s_{0}$,
\item\label{item:half-proc-3} $F_{u}(z)$ points along $T\nu^{*}N$, for $t(z)=0,1$,
\item\label{item:half-proc-4} $F_{u}(z)\Pi=V$ holds for $s(z)$ near $0$,
\item\label{item:half-proc-5} $F_{u}$ is constant in a neighborhood of the corners $(0,0)$ and $(0,1)$.
\end{enumerate}

Similarly to \eqref{eq:partially-defined}, consider the frame for $V$ defined by:
\begin{equation}\label{eq:half-partially-defined}
    \left\{
    \begin{aligned}
      &C_{1}^{\mathfrak{m}}F_{u}(z)&&\text{ for }t(z)=0,1,\\
      &F_{u}(z)C_{0}^{\mathfrak{m}}&&\text{ for }s(z)\in (-\infty,s_{0}]\cup \set{s=0}.\\
    \end{aligned}
  \right.
\end{equation}

This frame of $u^{*}V$ is defined on the domain shown in Figure \ref{fig:domain-of-frame}.

One uses the result of \S\ref{sec:fram-conorm-vert} and \ref{item:half-proc-5} to obtain a smooth transition at the corners. The final property we consider is:
\begin{enumerate}[resume]
\item\label{item:half-proc-6} the frame in \eqref{eq:half-partially-defined} extends to a frame of $u^{*}V$ over all of $\Sigma$.
\end{enumerate}
As in \S\ref{sec:line-proc}, this condition can be attained by precomposing $\Phi_{z}$ by a linear change of coordinates. If the properties \ref{item:half-proc-1} through \ref{item:half-proc-6} hold we say that choice of travelling frame $F_{u}$ is \emph{admissible}; such frames always exist.

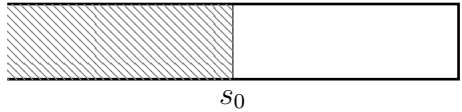
\begin{figure}[H]
  \centering
  \begin{tikzpicture}
    \draw[line width=1pt] (0,0)--(6,0)--+(0,1)--(0,1);
    \draw (3,1)--+(0,-1)node[below]{$s_{0}$};
    \fill[pattern={north west lines},pattern color=black!50!white] (0,0) rectangle (3,1);
  \end{tikzpicture}
  \caption{Over the marked region, including the solid boundary, there is a frame $\R^{n}\to u^{*}V$ determined by the choice of coordinates.}
  \label{fig:domain-of-frame}
\end{figure}

Using property \ref{item:moser-lag-d}, $\Phi_{F_{u}(z)}C_{0}^{\mathfrak{m}}=\Phi_{F_{u}(z)C_{0}^{\mathfrak{m}}}$ holds along $s=0$, so $\Phi_{z}^{-1}(V)=\Pi\cap B(\delta)$ holds because of \ref{item:moser-lag-c}. Therefore, in admissible coordinates the linearized operator has the following form:
\begin{equation*}
  D^{\Phi}_{u}(\eta)=\bd_{s}\eta+J_{z}^{\Phi}\bd_{t}\eta+S(s,t)\eta,
\end{equation*}
where $\eta$ takes the $\R^{n}$ and $\Pi$ boundary values as in \S\ref{sec:linear-theory-half} and $J_{z}^{\Phi}\bd_{t}\eta+S(s,t)\eta$ converge to $-A_{\gamma_{\pm}}$ as $s\to \pm \infty$. Here $J_{z}^{\Phi}:=\d\Phi_{z}(\mu(z))^{-1}J(u(z))\d\Phi_{z}(\mu(z))$ is $\omega_{0}$-tame.

Let $\mathscr{J}(\Pi)$ be the set of domain dependent $\omega_{0}$-tame complex structures $J_{z}$ so that $\Pi$, $\R^{n}$ are conormally related when $z=(0,0)$ or $z=(0,1)$. Standard arguments show that $\mathscr{J}(\Pi)$ is contractible in such a way that $D_{u}^{\Phi}$ can be canonically deformed through Fredholm operators until it is in the set $\mathrm{CR}(A_{-},\Pi)$ from \S\ref{sec:linear-theory-half}.

Thus the orientation line for $\det(D_{u}^{\Phi})$ is canonically equivalent to $\mathfrak{o}(A_{-},\Pi)$ and the Fredholm index of $D^{\Phi}_{u}$ is $d-n-\mathrm{CZ}(\gamma_{-})$; see \S\ref{sec:orientation-lines-1} and \S\ref{sec:fredholm-index-half}.

\begin{prop}\label{prop:coherent-orient-2}
  The identification between the orientation line of $\det(D_{u}^{\Phi})$ and $\mathfrak{o}(A_{-},\Pi)$ is independent of the choice of admissible coordinates $\Phi$, i.e., it commutes with the conjugation isomorphisms.
\end{prop}
\begin{proof}
  As in Proposition \ref{prop:coherent-orient-1}, set $P_{i}(z)=\d\Phi_{i,z}(\mu_{i}(z))$ and $M(z)=P_{1}(z)^{-1}P_{0}(z)$. One thinks of $M(z)$ as the linearized change of coordinates; it is a family of matrices in $\mathrm{Sp}(\R^{2n})$ over the strip and which preserve $\R^{n}$ when $t=0,1$, and which preserve the fixed linear conormal $\Pi$ when $s=0$.

  The Cauchy-Riemann operators $D^{\Phi_{i}}_{u}$ are conjugate to one another via $M(z)$, and so the problem reduces to showing that this conjugation action preserves orientation.

  Let $J_{1,z}=M(z)^{-1}J_{0}M(z)$, and pick a path $J_{\tau,z}$ of $\omega_{0}$-tame almost complex structures so that $\R^{n}$ and $\Pi$ remain conormally related for $J_{\tau,z}$ for $z=(0,0)$ and $z=(0,1)$. One can construct $E_{\tau}(z)$ so that $E_{\tau}(z)$ preserves $\R^{n}$ for $t(z)=0,1$ and preserves $\Pi$ for $s(z)=0$, and so that:
  \begin{equation*}
    J_{\tau,z}=E_{\tau}(z)^{-1}J_{1,z}E_{\tau}(z).
  \end{equation*}
  One first constructs $E_{\tau}$ at the corners, then extends to the boundary edges, and then extends to the interior. Then $M_{\tau}(z)=M(z)E_{\tau}(z)$ is a family of matrices so $M_{\tau}(z)^{*}J_{0}=J_{\tau,z}$; it follows that $M_{1}$ preserves orientation if and only if $M_{0}$ does. Note that $M_{0}$ is complex linear and preserves $\R^{n}$ and $\Pi$ on the respective boundary edges. The construction can be arranged so that $M_{\tau}(z)\to \id$ as $s(z)\to-\infty$, because the frames along the asymptotic end are unitary to begin with.

  One observes that $M_{0}(z)$ and $C_{0}^{-\mathfrak{m}}M_{0}(z)C_{0}^{\mathfrak{m}}$ are frames of $\R^{n}$ (along the boundary) which agree on the corners. By the homotopical assumption \ref{item:half-proc-6} these frames extend to frames of $\R^{n}$ defined on the full domain. By Lemma \ref{lemma:some-choice-of-sign}, it follows that $M_{0}$ preserves orientation, as desired.
\end{proof}

\subsection{Orientation lines for parametric moduli spaces}
\label{sec:param-moduli-space}

Let $P$ be a finite dimensional parameter space, and let $\mathrm{NL}:\mathscr{E}\times P\to \mathscr{F}$ be a smooth map between Banach spaces so that $0$ is a regular value. For our application to \S\ref{sec:isom-from-morse}, one should think of $\mathscr{E}$ as a local chart in the Banach manifold of $W^{1,p}$ maps $\Sigma\to W$ and $P$ a finite dimensional space of $W^{1,p}$ paths $q(t)$ in the base $M$, and:
\begin{equation*}
  \mathrm{NL}(u,q)=\bd_{s}(u-q(t))+J(u-q(t))(\bd_{t}(u-q(t))-X_{t}(u-q(t))),
\end{equation*}
where the symbol $u-q(t)$ should be understood in local coordinates. In the local model $u$ is assumed to take the boundary values as in \S\ref{sec:linear-theory-half}.

Let $\mathscr{M}=\mathrm{NL}^{-1}(0)$, and let:
\begin{equation*}
  D_{u,q}(\eta)=\pd{\mathrm{NL}(u,q)}{u}\eta\text{ and }B_{u,q}(v)=\pd{\mathrm{NL}(u,q)}{q}v\text{ and }T_{u,q}(\eta,v)=D_{u,q}(\eta)+B_{u,q}(v).
\end{equation*}
Since $\mathscr{M}$ is cut transversally, $T_{u,q}$ is surjective for each $(u,q)\in \mathscr{M}$.

\begin{prop}\label{prop:can-ident-param}
  There is a canonical identification:
  \begin{equation*}
    \mathfrak{o}(T_{u,q})\simeq \mathfrak{o}(D_{u,q})\otimes\mathfrak{o}(TP_{q})
  \end{equation*}
\end{prop}
\begin{proof}
  Abbreviate $T=D+B:X\times Y\to \mathscr{F}$ with $D:X\to \mathscr{F}$ and $B:Y\to \mathscr{F}$.

  Consider the following splittings:
  \begin{enumerate}
  \item $X=\ker D\oplus \im D$ and $Y=\im B\oplus \ker B$,
  \item $\im B=V\oplus(\im D\cap \im B)$ where $V$ is a linear complement to $\im D$,
  \end{enumerate}

  There is a canonical identification:
  \begin{equation*}
    \ker T=\ker D\oplus (\im D\cap \im B)\oplus \ker B.
  \end{equation*}
  There is the standard (lexicographic, i.e., write things out in order) identification:
  \begin{equation}\label{eq:lexico-reordering}
    \begin{aligned}
      \mathfrak{o}(D)\otimes \mathfrak{o}(Y)
      &\simeq \mathfrak{o}(\ker D\oplus V \oplus \im B\oplus \ker B),\\
      &\simeq \mathfrak{o}(\ker D\oplus V \oplus V\oplus (\im B\cap \im D)\oplus \ker B),\\
      &\simeq \mathfrak{o}(\ker D\oplus (\im B\cap \im D) \oplus \ker B),\\
      &=\mathfrak{o}(T)\\
    \end{aligned}
  \end{equation}
  where we use the complex orientation on $V\oplus V$ in the fourth line. This completes the proof.
\end{proof}

A particular case of interest is when $\dim \ker T_{u,q}=\dim \coker T_{u,q}=0$. In this case $\mathfrak{o}(T)$ is canonically oriented and hence each point $(u,q)\in \mathscr{M}$ determines a canonical identification:
\begin{equation*}
  \mathfrak{o}(TP_{q})\simeq \mathfrak{o}(D_{u,q});
\end{equation*}
these identifications are used in \S\ref{sec:isom-from-morse}.

\subsubsection{The boundary of a parametric moduli space}
\label{sec:bound-param-moduli}

Another case of interest is when $P$ is a manifold with boundary, e.g., $$P=\set{x_{1}\ge 0}\subset \R^{n},$$ and $\dim \ker T_{u,q}=1$. In this case one can restrict to $\bd P$ and obtain a restricted total differential $T_{u,q}^{\bd}$. Let us assume that the restricted total differential is also surjective; it follows that it is an isomorphism (i.e., has index $0$). Moreover, the implicit function theorem implies $\mathscr{M}$ is a $1$-manifold whose boundary is the inverse image of $\bd P$.

One then has two identifications:
\begin{equation*}
  \mathfrak{o}(T_{u,q})\simeq \mathfrak{o}(D_{u,q})\otimes \mathfrak{o}(TP)\text{ and }\mathfrak{o}(T^{\bd}_{u,q})\simeq \mathfrak{o}(D_{u,q})\otimes \mathfrak{o}(T\bd P).
\end{equation*}
Pick an orientation of $\mathscr{M}$. Then it makes sense to say that $(u,q)$ is a negative or positive boundary point (since $\mathscr{M}$ can be parametrized in such a way that its projection to $P$ is an immersion transverse to $\bd P$, near the boundary points $\bd \mathscr{M}$)

\begin{claim}
  Use the canonical identification $\mathfrak{o}(T\bd P)\to \mathfrak{o}(TP)$ given by addition of an outwards normal vector as the final basis vector. Let $(u,q)$ be a boundary point. The identifications:
  \begin{equation*}
    \mathfrak{o}(\mathscr{M})\simeq \mathfrak{o}(D_{u,q})\otimes \mathfrak{o}(TP)\text{ and }\Z\simeq \mathfrak{o}(D_{u,q})\otimes \mathfrak{o}(TP)
  \end{equation*}
  differ by a sign $\mathfrak{m}_{\mathrm{ind}}\mathfrak{m}_{\mathrm{dir}}$ where $\mathfrak{m}_{\mathrm{ind}}$ depends only on the index of $D_{u,q}$ and $\mathfrak{m}_{\mathrm{dir}}$ depends only on whether $(u,q)$ is a negative or positive boundary point. Indeed, with our conventions $\mathfrak{m}_{\mathrm{ind}}=+1$ and $\mathfrak{m}_{\mathrm{dir}}$ is the sign of $(u,q)$ as a boundary point.
\end{claim}
\begin{proof}
  This involves unravelling the definition of the identification given in the proof of Proposition \ref{prop:can-ident-param}, in particular \eqref{eq:lexico-reordering}.
\end{proof}

\subsubsection{Interior breaking in a parametric moduli space}
\label{sec:inter-break-param}

Let $\mathscr{M}$ be a one-dimensional component of a parametric moduli space, and suppose that $\mathscr{M}$ has a non-compact end oriented as $[0,\infty)$. In the context of \S\ref{sec:isom-from-morse} such an end corresponds to a breaking of the form shown in Figure \ref{fig:interior-breaking}. Note that the positive orientation of $\mathscr{M}$ points towards the breaking; see the proof of Lemma \ref{lemma:differential} for similar discussion in the context of the Floer differential.

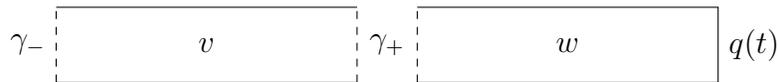
\begin{figure}[H]
  \centering
  \begin{tikzpicture}
    \draw (0,0)coordinate(A) -- +(4,0)coordinate(B) (0,1) -- +(4,0) (4.8,0)coordinate(C) --+(4,0)coordinate(D)--node[right]{$q(t)$}+(4,1)--+(0,1);
    \draw[dashed] (A)--node[left]{$\gamma_{-}$}+(0,1) (B)--+(0,1)coordinate(X) (C)--+(0,1)coordinate(Y);
    \path (X)--node{$\gamma_{+}$}(C) (X)--node{$v$}(A) (Y)--node{$w$}(D);
  \end{tikzpicture}
  \caption{An interior breaking in a parametric moduli space.}
  \label{fig:interior-breaking}
\end{figure}
Let $u$ be obtained by gluing $v$ and $w$, and assume that $(w,q)$ is a rigid element of the parametric moduli space, while $v$ is rigid-up-to-translation.

Proposition \ref{prop:kernel-cokernel-gluing} identifies $\ker(D_{v})\simeq \ker(D_{u,q})$ and $\coker(D_{w})\simeq \coker(D_{u,q})$, and hence an identification:
\begin{equation}\label{eq:identification-v-w-u}
  \mathfrak{o}(D_{v})\simeq \mathfrak{o}(D_{v})\otimes \mathfrak{o}(D_{w})\otimes \mathfrak{o}(TP)\to \mathfrak{o}(D_{u,q})\otimes \mathfrak{o}(TP)\simeq \mathfrak{o}(T_{u,q})=\mathfrak{o}(\mathscr{M}),
\end{equation}
where $\mathfrak{o}(D_{w})\otimes \mathfrak{o}(TP)$ is oriented as a rigid element of the lower dimensional moduli space.
\begin{claim}
  The identification \eqref{eq:identification-v-w-u} sends the orientation of $\mathscr{M}$ pointing towards the breaking to the generator $-\eta_{v}$ of $\mathfrak{o}(D_{v})$.
\end{claim}
\begin{proof}
  This involves a standard analysis of the gluing isomorphism in Proposition \ref{prop:kernel-cokernel-gluing}; see, e.g, \cite{fh_coherent} and the references therein.
\end{proof}
The results in this section are used to prove that the $\Theta$ map is a chain map in \S\ref{sec:chain-map-property}.

\subsubsection{Identifying orientation line of a chord with the orientation line of a half-strip}
\label{sec:ident-orient-line}

Linear gluing $\mathrm{CR}(A_{0},A_{\gamma})\times \mathrm{CR}(A_{\gamma},\Pi)\to \mathrm{CR}(A_{0},\Pi)$ establishes a canonical isomorphism between the orientation lines:
\begin{equation}\label{eq:canon-ident-Pi}
  \mathfrak{o}_{\gamma}\otimes \mathfrak{o}(D_{u})\to \mathfrak{o}(A_{0},\Pi),
\end{equation}
where we use the definition $\mathfrak{o}_{\gamma}=\mathfrak{o}(A_{0},A_{\gamma})$ and the identification $\mathfrak{o}(D_{u})\simeq \mathfrak{o}(A_{\gamma},\Pi)$ described in Proposition \ref{prop:coherent-orient-2}; see Figure \ref{fig:canon-ident-line-glue} and \S\ref{sec:gluing-operation} for the details of the gluing operation.

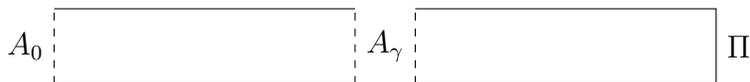
\begin{figure}[H]
  \centering
  \begin{tikzpicture}
    \draw (0,0)coordinate(A) -- +(4,0)coordinate(B) (0,1) -- +(4,0) (4.8,0)coordinate(C) --+(4,0)coordinate(D)--node[right]{$\Pi$}+(4,1)--+(0,1);
    \draw[dashed] (A)--node[left]{$A_{0}$}+(0,1) (B)--+(0,1)coordinate(X) (C)--+(0,1)coordinate(Y);
    \path (X)--node{$A_{\gamma}$}(C) (X)--(A) (Y)--(D);
  \end{tikzpicture}
  \caption{Linear gluing.}
  \label{fig:canon-ident-line-glue}
\end{figure}

In particular, fixing a generator $\mathfrak{g}$ of $\mathfrak{o}(A_{0},\Pi)$ gives us an identification $\mathfrak{o}(D_{u})\simeq \mathfrak{o}_{\gamma}$; this identification is used in the definition of the $\Theta$ map in \S\ref{sec:isom-from-morse}; it depends on the choice of $\mathfrak{g}$ up to a sign.

\section{Morse theory for the energy functional}
\label{sec:appendix-B}

\subsection{Generic metrics and the Morse property}
\label{sec:gener-metr-morse}

In this section, we explain how to find \textit{bumpy metrics} $g$, which are metrics with the property that all critical points of $\mathscr{E}:\mathscr{P}\to [0,\infty)$ with positive critical value are Morse. The case $N=\set{p,q}$ follows from the arguments in \cite[\S18]{milnor_morse_theory}. In the general case, we follow \cite{anosov_closed_geodesics}. For more applications of bumpy metrics see \cite[\S4]{stojisavljevic_zhang}, \cite[\S3]{abbondandolo_schwarz_estimates_RFH}, and \cite{klingenberg_lectures,frauenfelder_arnold_givental,oancea_closed_geodesics}.

Given a Riemannian metric $g_0$, let $\eta, \xi$ be variations along the geodesic $x \in Crit(\mathscr{E})$. Since $x$ satisfies boundary conditions $x(0), x(1) \in N$, we assume that $\eta(0), \eta(1), \xi(0)$ and $\xi(1)$ are tangent to $N$. By the second variation formula \cite[Theorem 13.1]{milnor_morse_theory}, the Hessian of $\mathscr{E}$ is given by:
\begin{equation*}
  d^2 \mathscr{E}(x)(\xi, \eta)=-2 \int_0^1g_{0} \left(\xi, \frac{D^2 \eta}{dt^2} + R_{g_{0}}(x', \eta) x' \right) dt,
\end{equation*}
where $R_{g_{0}}$ is the Riemannian curvature tensor. Since the null space of the Hessian is the set of Jacobi fields along $x$ which are tangent to $N$ at the endpoints, $\mathscr{E}$ is Morse away from the set of constant paths if there are no such Jacobi fields. On the other hand, there is a bijective correspondence between Jacobi fields $J$ along $x$ and equivariant vector fields $\Tilde{J}(t) = D\varphi^t(x'(0)) \Tilde{J}(0)$ given by $J \mapsto (J, \nabla J)$, where $\varphi_{g_{0}}^t$ is the geodesic flow on $TM$; see \cite[Lemma 3.1.6]{klingenberg_lectures}. Because of this correspondence, the Morse condition for $\mathscr{E}$ is equivalent to the transversality of $\varphi^1_{g_{0}}(\nu^* N)$ and $\nu^* N$ away from the zero section.

Let $S^* M$ be the ideal contact boundary of $T^*M$, and let $\Lambda_{N} \subset S^*M$ be the Legendrian boundary of the conormal bundle $\nu^* N$. Let $H_{g_{0}}(p):= \| p \|_{g_{0}}$, and note that $H_{g_{0}}$ defines normalized co-geodesic flow $\Phi_{g_{0}}^t$ on $T^*M \setminus M$. Since $\Phi_{g_{0}}^t(e^s p) = e^s \Phi_{g_{0}}^t(p)$, $\Phi_{g_{0}}^t$ is equivariant with respect to the Liouville flow and hence induces an autonomous flow on $S^* M$. Consider the map:
\begin{equation*}
  F_{g_{0}}:S^* M \times \R_{+}\to S^* M \times S^* M
  \hspace{.5cm}\text{given by}\hspace{.5cm}(p, t)\mapsto (p, \Phi_{g_{0}}^t(p)).
\end{equation*}
The transversality condition $\varphi^1_{g}(\nu N) \pitchfork \nu N$ is equivalent to $F_{g_{0}} \pitchfork \Lambda_{N} \times \Lambda_{N}$. Let $\mathcal{G}_{\epsilon}(g_0)$ be the Banach manifold given by: $$\mathcal{G}_{\epsilon}(g_0):= \{\text{symmetric tensor fields $g$}: \sum \epsilon_{k} \|g-g_{0}\|_{C^k} < 1 \};$$
for similar use of such a Banach manifold see \cite{floer-lag,mcduffsalamon,wendl-sft}. One picks $\epsilon_{0}$ large enough that $\mathcal{G}_{\epsilon}(g_{0})$ consists only of metrics, and picks $\epsilon_{k}$ decaying sufficiently rapidly that $\mathcal{G}_{\epsilon}(g_{0})$ contains enough compactly supported bump functions.

Following \cite[\S4]{anosov_closed_geodesics}, introduce the universal map $F(p,t,g)=F_{g}(p,t)$.

One says that $(p, t_{0} ,g)$ is an \textit{injective} point if there is $s_{0} \in (0, t_{0})$ such that the geodesic $\gamma(t):=\Phi^{t}_{g}(p)$ satisfies $\gamma^{-1}\{\gamma(s_{0})\}=\{ s_{0} \}$. It follows from the arguments in \cite[Lemma 2]{anosov_closed_geodesics} that $\d F(p,t_{0},g)$ is surjective for all injective points, hence $F$ restricted to the set of all injective points is transverse to any submanifold.

The set $\mathcal{G}^{*}_{\mathrm{reg}}$ of regular values of the projection $F^{-1}(\Lambda_{N}\times \Lambda_{N})\to \mathcal{G}_{\epsilon}(g_{0})$ is of the second Baire category. For $g \in \mathcal{G}^{*}_{\mathrm{reg}}$, if $(p,t_{0},g)$ is injective point, we have: $$\d F_{g} (p,t_{0}) T_{(p,t_{0})}(S^*M \times \R_{+})\oplus T_{\left(p, \Phi^{t_0}_{g}(p) \right)}( \Lambda_{N} \times \Lambda_{N}) = T_{\left(p,\Phi^{t_0}_{g}(p) \right)} (S^* M \times S^*M).$$

It is left to show that all points $(p, T, g)$ such that $F(p, T, g) \in \Lambda_N \times \Lambda_N$ are injective for generic $g$, or equivalently, there is no closed geodesic of $g$ which passes through $N$ orthogonally. This follows from the fact that $F$ restricted to the set of injective points is transverse to $\Delta_{N}:= \{(p,p) \mid p \in \Lambda_{N} \}$.

The projection to the third factor $(p, t, g) \mapsto g$ is a Fredholm map, and when restricted to the set of all injective points in $F_{g}^{-1}(\Delta_N)$, it has index $-n-1$ (which is equal to the ``expected'' dimension of $F_{g}^{-1}(\Delta_N) \subset S^*M \times \R_{+}$). Since the index is negative, regular metrics $g$ satisfy $F_{g}^{-1}(\Delta_N)= \emptyset$, i.e., closed geodesics are not orthogonal to $N$ for generic metrics $g$. Let $\mathcal{G}'_{reg}$ be the set of such metrics.

It follows that for every metric $g$ in the Baire generic set $\mathcal{G}'_{reg} \cap \mathcal{G}^{*}_{reg}$, all non-constant geodesics orthogonal to $N$ at the endpoints are non-degenerate critical points of $\mathscr{E}_{g}$.

\subsection{Relative Hurewicz theorem and Morse homology}
\label{sec:relat-hurew-morse}

In this section we show that \eqref{eq:want_to_show_injective} is an isomorphism via a geometric argument similar to the one in \S\ref{sec:comp-with-sing}. We will focus on proving that \eqref{eq:want_to_show_injective} is injective, as this is what is necessary to deduce the Arnol'd chord conjecture. The arguments relating Morse theory to homotopy groups are similar to those relating Morse theory to singular homology and pseudo-cycle cobordism; see \cite{schwarz_pseudocycles,zinger_pseudocycles} and \cite{mcduffsalamon,wilkins_quantum_steenrod,wilkins_pseudocycles} for more details.

For simplicity, and without any serious loss of generality, we replace $\mathscr{P}$ by one of its finite dimensional approximations in \S\ref{sec:finite-dimens-appr}.

Recall that we assume that $\pi_{j}(\mathscr{P},N,\mathrm{pt})=0$ for $j<k$.

Let $\Sigma_{k}$ be the union of the unstable disks of the index $\le k$ critical points. By an inductive process, one can find a homotopy $\varphi_{t}:\mathscr{P}\times [0,1]\to \mathscr{P}$ so that $\varphi_{0}=\id$, $\varphi_{t}(N)\subset N$, and $\varphi_{1}$ takes a neighborhood $V$ of $\Sigma_{k}\setminus U$ into $N$, where $U$ is a small neighborhood of the index $k$ critical points, see Figure \ref{fig:unstable-projection} for an illustration of $U$. The reason it is possible to construct such a homotopy is because $\Sigma_{k}\setminus U$ deformation retracts onto the union of the unstable disks of the index $<k$ critical points, and these lower dimensional unstable disks represent trivial elements in $\pi_{j}(\mathscr{P},N,\mathrm{pt})$ (because this relative homotopy group is trivial).

Now let $f:(D^{k},\bd D^{k},\mathrm{pt})\to (\til{\mathscr{P}} ,N',\mathrm{pt})$ represent some homotopy class which lies in the kernel of \eqref{eq:want_to_show_injective}. First suppose that the count of trajectories is zero on chain level. By perturbing $f$ to be in general position, we conclude a subset of $2K$ points in the interior of the disk $\set{x_{1}^{-},x_{1}^{+},\dots,x_{K}^{-},x_{K}^{+}}$ so that both $x_{i}^{\pm}$ have a flow line to the same index $k$ critical point $y_{i}$ contributing opposite signs in the Morse differential.

By picking $U$ sufficiently small and flowing $f$ for a sufficiently long time, we may suppose that $f^{-1}(U)$ consists of $2K$ disks centered on the points $x_{i}^{\pm}$. Moreover, each disk is mapped by $f$ onto a disk which projects diffeomorphically onto a small piece of the unstable manifold of $y_{i}$; see Figure \ref{fig:unstable-projection}, and the two induced orientations of the unstable manifold at $y_{i}$ differ. We also suppose that $D\setminus f^{-1}(U)$ is mapped into the neighborhood $V$ (which can be achieved by flowing long enough).

After having deformed $f$ in this manner, consider the homotopy $\varphi_{t}\circ f$. This homotopy is relative the boundary of $f$, and the resulting map $\varphi_{1}\circ f$ is valued in $N'$ except for the $2K$ small disks $f^{-1}(U)$; see Figure \ref{fig:almost-in-N}.

By construction, the restriction of $\varphi_{1}\circ f$ to the $x_{i}^{-}$ disk equals the restriction to the $x_{i}^{+}$ disk composed with an orientation preserving diffeomorphism.

The \emph{relative homotopy addition theorem} from \cite[\S VII]{bredon_GT} states that such a disk (decomposed into subdisks) represents a decomposition in relative homotopy groups. Since $N'$ is simply connected and $\pi_{k}(\mathscr{P},N,\mathrm{pt})$ is commutative for $k\ge 2$, the theorem implies that $\varphi_{1}\circ f$, and hence $f$, can be written as a cancelling sum: $$f=u_{1}-u_{1}+\dots+u_{K}-u_{K}=0$$ in $\pi_{k}(\mathscr{P},N,\mathrm{pt})$; here $\pm u_{i}$ is induced by the restriction of $\varphi_{1}\circ f$ to the disk around $x_{i}^{\pm}$. In this special case where the count of trajectories is zero on chain level, we have shown that \eqref{eq:want_to_show_injective} is injective.

In the general case when the count of trajectories equals $\d(\sum z_{i})$ where $z_{i}$ are index $k+1$ critical points, one can add to $f$ the appropriate number of unstable spheres around the $z_{i}$ critical points. This yields a map defined on a union of a disk and some number of spheres. Each such sphere can be joined to the basepoint and thereby be interpreted as the zero element of $\pi_{k}(\mathscr{P},N,\mathrm{pt})$. Thus we can add these elements to $f$ to obtain the same element of $\pi_{k}(\mathscr{P},N,\mathrm{pt})$. The altered representative will now have a count of trajectories which vanishes on chain level, and we can apply the preceding argument.

This completes the proof of injectivity of \eqref{eq:want_to_show_injective}. Surjectivity of \eqref{eq:want_to_show_injective} follows from similar arguments, and is left to the reader.

\begin{figure}[H]
  \centering
  \begin{tikzpicture}[scale=1.5]
    \fill[black!10!white] (-0.5,-0.6) rectangle (0.5,0.6);
    \draw (-3,0)--coordinate[pos=0.5](y)(3,0);
    \foreach \x in {-1.4,-1,...,1.4} {
      \draw[->,black!50!white] (\x,1)--+(0,-0.8);
      \draw[->,black!50!white] (\x,-1)--+(0,0.8);
    }
    \draw[line width=1pt,red] plot[variable=\x,domain=-2:-1,samples=100] (\x,{0.5+0.1*sin(90*\x)}) to[out=0,in=180] (-0.5,0)--(0.5,0)to[out=0,in=180] (1,0.6);
    \draw[line width=1pt,red] plot[variable=\y,domain=1:2,samples=100] (\y,{0.5+0.1*sin(90*\y)});
    \draw[variable=\x,domain=-2:-1,samples=100,line width=1pt,red] plot (\x,{-0.5+0.1*sin(90*\x)}) to[out=0,in=180] (-0.5,0)--(0.5,0)to[out=0,in=180] (1,-0.4);
    \draw[variable=\x,domain=1:2,samples=100,line width=1pt,red] plot (\x,{-0.5+0.1*sin(90*\x)});
    \draw[dashed] (-0.5,-0.6) rectangle (0.5,0.6);
    \node at (0,0) [above] {$y_{i}$};
  \end{tikzpicture}
  \caption{Projecting disks around $x^\pm_i$ (both disks shown in red) onto a piece of the unstable manifold of $y_i$ (shown as a horizontal black line). The dashed box is $U$.}
  \label{fig:unstable-projection}
\end{figure}
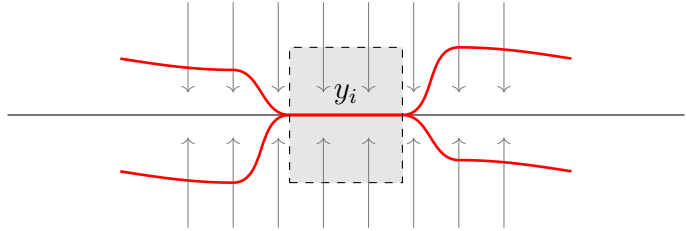

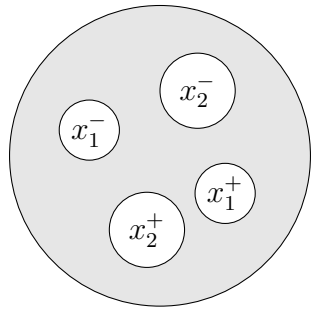
\begin{figure}[H]
  \centering
  \begin{tikzpicture}
    \draw[fill=black!10!white,even odd rule] (0,0) circle (2) (60:1) circle (0.5)node{$x_{2}^{-}$} (330:1) circle (0.4)node{$x_{1}^{+}$} (260:1) circle (0.5)node{$x_{2}^{+}$} (160:1) circle (0.4)node{$x_{1}^{-}$};
  \end{tikzpicture}
  \caption{After the homotopy, the disk decomposes into $K$ cancelling pairs of disks. The shaded region is mapped into $N'$.}
  \label{fig:almost-in-N}
\end{figure}

\bibliography{citations}
\bibliographystyle{alpha}
\end{document}